\documentclass[10pt]{amsart}

\makeatletter
\def\@settitle{\begin{center}%
  \baselineskip14\p@\relax
    \normalfont\LARGE

  \@title
  \end{center}%
}
\makeatother
\usepackage{epigraph}
\usepackage[dvipsnames]{xcolor}
\usepackage[hyperfootnotes=false]{hyperref}
\usepackage{enumerate}
\usepackage{tikz}
\usepackage{tikz-cd}
\usepackage{hyperref}
\usepackage{amsthm,amsfonts,amsmath,amscd,amssymb}
\usepackage{xcolor,float}
\usepackage[all]{xy}
\usepackage{mathrsfs}
\usepackage{graphics,graphicx}
\usepackage{caption}
\usepackage{comment}
\usepackage{array}
\newcolumntype{P}[1]{>{\centering\arraybackslash}p{#1}}
\newcolumntype{M}[1]{>{\centering\arraybackslash}m{#1}}

\usepackage{epigraph}
\let\oldmarginpar\marginpar
\renewcommand\marginpar[1]{\-\oldmarginpar[\raggedleft\footnotesize #1]%
	{\raggedright\footnotesize #1}}
\theoremstyle{plain}
\newtheorem{assumption}{Assumption}
\newtheorem{thm}{Theorem}[section]
\newtheorem{lemma}[thm]{Lemma}
\newtheorem*{theorem*}{Theorem}
\newtheorem*{corollary*}{Corollary}
\newtheorem{example}[thm]{Example}
\newtheorem{prop}[thm]{Proposition}

\newtheorem{cor}[thm]{Corollary}

\theoremstyle{definition}

\newtheorem{definition}[thm]{Definition}

\newtheorem{remark}[thm]{Remark}

\numberwithin{equation}{section}

\newcommand{\g}{\gamma}

\renewcommand{\L}{\mathbb{L}}

\newcommand{\N}{\mathbb{N}}
\newcommand{\Z}{\mathbb{Z}}

\newcommand{\R}{\mathbb{R}}
\newcommand{\C}{\mathbb{C}}

\newcommand{\sD}{\mathscr{D}}

\newcommand{\SC}{\mathcal{C}}

\newcommand{\La}{\Lambda}

\newcommand{\la}{\lambda}

\newcommand{\dd}{\partial}

\newcommand{\sse}{\subset}
\newcommand{\lr}{\longrightarrow}

\newcommand{\HH}{\operatorname{HH}}
\newcommand{\Br}{\operatorname{Br}}

\newcommand{\Aut}{\operatorname{Aut}}

\newcommand{\wt}{\widetilde}

\newcommand{\Lag}{\operatorname{Lag}}
\newcommand{\Seed}{\operatorname{Seed}}

\def\Op{{\mathcal O}{\it p}}
\newcounter{daggerfootnote}

\newcommand{\bG}{\mathbb{G}}

\newcommand{\bL}{\mathbb{L}}

\newcommand{\cC}{\mathcal{C}}

\usepackage{dsfont}
\allowdisplaybreaks

\def \vertbar [#1](#2,#3,#4){
    \draw [#1] (#2,#3) -- (#2,#4);
    \draw [fill=white] (#2,#3) circle [radius=0.1];
    \draw [fill=black] (#2,#4) circle [radius=0.1];
}

\usepackage{mathdots}
\usepackage{afterpage}
\makeatletter
\providecommand{\leftsquigarrow}{%
  \mathrel{\mathpalette\reflect@squig\relax}%
}
\newcommand{\reflect@squig}[2]{%
  \reflectbox{$\m@th#1\rightsquigarrow$}%
}
\makeatother

\makeatletter
\def\Ddots{\mathinner{\mkern1mu\raise\p@
\vbox{\kern7\p@\hbox{.}}\mkern2mu
\raise4\p@\hbox{.}\mkern2mu\raise7\p@\hbox{.}\mkern1mu}}
\makeatother

\def \horline [#1](#2,#3,#4){
    \draw [#1] (#2,#4) -- (#3,#4);
    \draw [fill=white] (#2,#4) circle [radius=0.1];
    \draw [fill=black] (#3,#4) circle [radius=0.1];
}

\def \crossing (#1,#2)(#3,#4){
\draw (#1,#2) -- (#3,#4);
\draw (#1,#4) -- (#3,#2);
}

\DeclareFontFamily{U}{mathb}{}
\DeclareFontShape{U}{mathb}{m}{n}{
  <-5.5> mathb5
  <5.5-6.5> mathb6
  <6.5-7.5> mathb7
  <7.5-8.5> mathb8
  <8.5-9.5> mathb9
  <9.5-11.5> mathb10
  <11.5-> mathbb12
}{}

\usetikzlibrary{decorations.markings, arrows}
\tikzset{tangent/.style={decoration={markings,mark=at position #1 with {
      \coordinate (tangent point-\pgfkeysvalueof{/pgf/decoration/mark info/sequence number}) at (0pt,0pt);
      \coordinate (tangent unit vector-\pgfkeysvalueof{/pgf/decoration/mark info/sequence number}) at (1,0pt);
      \coordinate (tangent orthogonal unit vector-\pgfkeysvalueof{/pgf/decoration/mark info/sequence number}) at (0pt,1);
      }},postaction=decorate},
    use tangent/.style={
        shift=(tangent point-#1),
        x=(tangent unit vector-#1),
        y=(tangent orthogonal unit vector-#1)
    },
    use tangent/.default=1
    }

\begin{document}

	\title{A Lagrangian filling for every cluster seed}
	

	\author{Roger Casals}
	\address{University of California Davis, Dept. of Mathematics, USA}
	\email{casals@math.ucdavis.edu}
	 
	\author{Honghao Gao}
	\address{Tsinghua University, Yau Mathematical Sciences Center, Beijing, China}
	\email{gaohonghao@tsinghua.edu.cn}

\maketitle

\vspace{-1cm}
\begin{abstract}
We show that each cluster seed in the augmentation variety contains an embedded exact Lagrangian filling. This resolves the matter of surjectivity of the map from Lagrangian fillings to cluster seeds. The main new technique to produce these Lagrangian fillings is the construction and study of a quiver with potential associated to curve configurations. We prove that its deformation space is trivial and show how to use it to manipulate Lagrangian fillings with $\L$-compressing systems via Lagrangian disk surgeries.\\
\end{abstract}


\section{Introduction}\label{sec:intro}

We show that each cluster seed contains an embedded exact Lagrangian filling. Heretofore, the surjectivity of the map from Lagrangian fillings to cluster seeds remained open for essentially all braids. The argument is based on applying Lagrangian disk surgeries to an initial Lagrangian filling with an $\L$-compressing system. We are able to avoid the appearance of immersed curves, a known technical issue of this problem, by introducing and studying a new object: a quiver with potential for each such $\L$-compressing system. The manuscript first establishes the key properties of these new quivers with potentials, including their rigidity and invariance. We then show how to use these properties to construct a Lagrangian filling for every cluster seed.\\

{\bf Scientific context}. The study of Lagrangian submanifolds has been central to some of the developments in symplectic topology, see for instance \cite{Gromov85,OS04,Seidel08,FOOO09, Abouzaid12,Auroux15}. Correspondingly, the study of Lagrangian fillings of Legendrian knots has played an important role in low-dimensional contact topology, see e.g.~\cite{EliashbergPolterovich96,EHK,EkholmLekili23}. For instance, in many interesting cases, Weinstein manifolds can be studied through Lagrangian skeleta (cf.~e.g.~\cite{CieliebakEliashberg12,RSTZ14,CasalsMurphy}) and those can be constructed and studied using Lagrangian fillings (see e.g.~\cite{BourgeoisEkholmEliashberg12,CasalsLagSkel}). From an algebraic viewpoint, the study of both Floer-theoretic and sheaf-theoretic invariants of Legendrian links is tightly connected to Lagrangian fillings, see e.g.~\cite{Chekanov,EkholmEtnyreSullivan05b,EkholmEtnyreSullivan05a,GKS_Quantization,EkholmEtnyreNgSullivan13,EHK,STZ_ConstrSheaves,STWZ}. More recently, techniques developed to study Lagrangian fillings (see \cite{CasalsZaslow,CasalsWeng22}) were applied to solve some problems in the area of cluster algebras, including the construction of cluster algebras in Richardson varieties (\cite[Theorem 1.1]{CGGLSS}) and the proof that the twist map for positroids equals the Donaldson-Thomas transformation (\cite[Theorem B]{CLSW23}).

Let $\beta$ be a positive braid word. Consider the Legendrian link $\La_{\beta}\sse(\R^3,\xi_{st})$ obtained as the Legendrian lift of the rainbow closure of $\beta$, as defined in \cite[Section 2.2]{CasalsNg} or cf.~Section \ref{ssec:prelim_symp}. Here $\xi_{st}=\ker\{dz-ydx\}$ where $(x,y,z)\in\R^3$ are Cartesian coordinates.

Consider its augmentation variety $X(\La_\beta,T)$, where $T\sse\La_\beta$ is a set of marked points, one per component. See Section \ref{ssec:aug} or \cite[Section 5.1]{CasalsNg} and references therein for more details on such augmentation varieties. This affine algebraic variety $X(\La_\beta,T)$ is smooth and isomorphic to the open (double) Bott-Samelson variety with pair of braids $(\beta,e)$, where $e$ is the identity braid, up to trivial $\C^\times$-factors. It is known that its ring of regular functions $\C[X(\La_\beta,T)]$ is a cluster algebra. See \cite{ShenWeng,GSW} for both these facts or cf.~Section \ref{ssec:aug} below. An alternative, intrinsically symplectic construction of such cluster structures is provided in \cite{CasalsWeng22} via the microlocal theory of sheaves.


A salient property of this cluster algebra structure constructed on $\C[X(\La_\beta,T)]$ is that, in all known cases, an oriented embedded exact Lagrangian filling $L$ of $\La_\beta\sse(\R^3,\xi_{st})$, embedded in the symplectization $(\R^4,\la_{st})$ of $(\R^3,\xi_{st})$, gives a cluster seed in $\C[X(\La_\beta,T)]$. Specifically, the filling $L$ gives an open toric chart in $X(\La_\beta,T)$ and there is a choice of $\L$-compressing system $\Gamma$ for $L$ that endows this toric chart with cluster $\mathcal{A}$-coordinates, thus defining a seed $\mathfrak{c}(L,\Gamma)$ in $\C[X(\La_\beta,T)]$. The description of the cluster seed is explicit for Lagrangian fillings $L$ associated to pinching sequences of $\beta$, see \cite{CasalsNg} and \cite{GSW}, and \cite{CasalsZaslow} provides a diagrammatic calculus to describe more general cluster seeds. 

Note that these embedded exact Lagrangians $L$ for $\La_\beta$  are typically surfaces of higher genus. In addition, if $L$ and $L'$ are compactly supported Hamiltonian isotopic, then the toric charts in $X(\La_\beta,T)$ associated to these seeds must be equal. Again, see \cite{EHK,GSW} or \cite{CasalsWeng22}. Such an invariance property has been successfully used to distinguish Lagrangian fillings, cf.~\cite{CasalsHonghao,CasalsZaslow,GSW}.\\

{\bf The Question}. Let $\Lag(\La_\beta)$ be the set of Hamiltonian isotopy classes of embedded exact Lagrangian fillings of $\La_\beta$ in the symplectization $(\R^4,\la_{st})$ of $(\R^3,\xi_{st})$. A Hamiltonian isotopy class is given by the equivalence relation $L\sim L'$ iff there exists a compactly supported Hamiltonian diffeomorphism $\varphi\in\mbox{Ham}^c(\R^4,\la_{st})$ such that $\varphi(L)=L'$. By \cite[Section 3.5]{EHK}, cf.~also \cite[Theorem 3.6]{CasalsNg}, there exists a map $\mathfrak{C}^\circ:\Lag(\La_\beta)\lr \mbox{Toric}(X(\La_\beta,T))$, where $\mbox{Toric}(X(\La_\beta,T))$ is the set of open unparametrized algebraic toric charts $(\C^\times)^d\sse X(\La_\beta,T)$ in the affine variety $X(\La_\beta,T)$, where $d=\dim_\C X(\La_\beta,T)$. An $\L$-compressing system $\Gamma$ for a Lagrangian filling $L\in\Lag(\La_\beta)$ endows the toric chart $\mathfrak{C}^\circ(L)$ with toric coordinates $A(\Gamma)$, e.g.~ see \cite[Section 4.6]{CasalsWeng22}. ($\L$-compressing systems will be momentarily defined, see also Definition \ref{def:compressible}.) Let $\Lag^c(\La_\beta)$ be the set of pairs $(L,\Gamma)$ consisting of a Lagrangian filling in $L\in\Lag(\La_\beta)$ and an $\L$-compressing system $\Gamma$ for $L$, up to Hamiltonian isotopy, satisfying the condition that $A(\Gamma)$ are cluster coordinates for a cluster seed of the cluster algebra structure in $\C[X(\La_\beta,T)]$ above. Finally, let $\Seed(X(\La_\beta,T))$ be the set of cluster seeds in $\C[X(\La_\beta,T)]$. In summary, at its coarsest level, the constructions cited above, sending a Lagrangian filling with an $\L$-compressing system $(L,\Gamma)$ to the cluster seed $\mathfrak{c}(L,\Gamma):=(\mathfrak{C}^\circ(L),A(\Gamma
))$, yields a map of sets:
$$\mathfrak{C}:\Lag^c(\La_\beta)\lr\Seed(X(\La_\beta,T)),\quad (L,\Gamma)\mapsto \mathfrak{C}(L,\Gamma):=\mathfrak{c}(L,\Gamma).$$


In our view, the surjectivity and injectivity of this map $\mathfrak{C}$ is a central open problem in low-dimensional contact and symplectic topology. It lies at the core of the study of Legendrian knots in $(\R^3,\xi_{st})$ and, more generally, the symplectic topology of Weinstein 4-manifolds. Note that \cite[Prop.~5.3]{CaoKellerQin22} implies that the forgetful map $\iota:\Seed(X(\La_\beta,T))\lr\mbox{Toric}(X(\La_\beta,T))$ sending a cluster seed to its underlying toric chart is injective. Thus, surjectivity of $\mathfrak{C}$ would imply that $\mathfrak{C}^\circ$ surjects onto $\iota(\Seed(X(\La_\beta,T)))$.

The state of affairs is as follows:

\begin{itemize}
    \item[-] Injectivity of $\mathfrak{C}$ is a generalization of the nearby Lagrangian conjecture for surfaces with boundary to a statement about embedded exact Lagrangians in Weinstein neighborhoods of Lagrangian skeleta. At core, it states that any embedded exact Lagrangian in a neighborhood of the arboreal Lagrangian skeleton consisting of $L$ and Lagrangian disks attached to it is either Hamiltonian isotopic to $L$ or to those Lagrangians obtained from it by Lagrangian surgeries along the disks in the skeleton. By \cite{EliashbergPolterovich96}, $\mathfrak{C}$ is injective if $\La_\beta$ is the max-tb Legendrian unknot. Injectivity of $\mathfrak{C}$ remains open for any other Legendrian link $\La_\beta\sse(\R^3,\xi_{st})$.\\

    \item[-] Surjectivity of $\mathfrak{C}$ is a reconstruction statement, from an algebraic invariant back to actual 4-dimensional symplectic topology. Indeed, it inputs algebraic data, provided by the ring of functions $\C[X(\La_\beta,T)]$ and a seed for its cluster structure, and should output symplectic topological data, an embedded exact Lagrangian filling of $\La_\beta$ and an $\L$-compressing system.\\
    
    By the finite type classification of cluster algebras, established in \cite{FominZelevinsky_ClusterII}, $\Seed(X(\La_\beta,T))$ is known to be a finite set only for a few exceptional cases, i.e.~the ADE cases. For those ADE cases, $\mathfrak{C}$ can be verified to be surjective by direct computation: this has recently been established in \cite{Hughes21D,AnBaeLee21ADE} by using weaves \cite{CasalsZaslow}. See also \cite{AnBaeLee21DEtilde} for the affine ADE cases where $\Seed(X(\La_\beta,T))$ is still finite up to the natural tame quotient. These finite type and affine type are rare: essentially all braids $\beta$ have $\Seed(X(\La_\beta,T))$ be an infinite set,  cf.~\cite[Section 4]{GSW} or \cite[Section 5]{CasalsLagSkel}. Confer \cite{CasalsHonghao,CasalsNg,CasalsWeng22,STWZ,STW} and references therein for partial results on $\mathfrak{C}$ and \cite[Section 5]{CasalsLagSkel} for further discussions. Surjectivity of $\mathfrak{C}$ remains open for any other Legendrian link $\La_\beta\sse(\R^3,\xi_{st})$.
\end{itemize}

There has been significant activity in recent times in the study of Lagrangian fillings, see e.g.~\cite{AnBaeLee21ADE,AnBaeLee21DEtilde,orsola,CasalsHonghao,CasalsNg,CasalsLi22,CGGLSS,CasalsWeng22,CasalsZaslow,EHK,GSW,Golovko,Hughes21D,Hughes_Atype,STWZ}. These results use techniques from and connecting to the microlocal theory of sheaves, Floer theory and cluster algebras among others. That said, all fall objectively short of showing anything close to the surjectivity of $\mathfrak{C}$. This manuscript introduces a genuinely new technique: the definition, study and use of a quiver with potential $(Q,W)$ to construct Lagrangian fillings.\\

{\bf The Main Result}. The goal of this manuscript is to establish the surjectivity of the map $\mathfrak{C}$. In order to state the main result, which is a stronger version of surjectivity, we introduce the following concepts. Let $L\sse(\R^4,\la_{st})$ be an exact oriented Lagrangian filling, $\g\sse L$ an embedded oriented curve and $\La_\g$ its Legendrian lift to the ideal contact boundary $\dd\Op(L)$ of a convex open neighborhood $\Op(L)$ of $L$ in $\R^4$. Note that $\Op(L)$ is symplectomorphic to a convex neighborhood of the zero section in $(T^*L,\la_{st})$. By definition, $\g$ is said to be $\bL$-compressible if there exists a properly embedded Lagrangian 2-disk $D\sse (T^*\R^2\setminus \Op(L))$ such that $\dd\overline{D}\cap \dd\Op(L)=\La_\g\sse\R^4$ and the union of $\overline{D}\cup \nu_\gamma$ is a smooth Lagrangian disk, where $\nu_\gamma\sse\Op(L)$ is the Lagrangian conormal cone of $\gamma$.

A collection $\Gamma=\{\gamma_1,\ldots,\gamma_b\}$ of oriented embedded curves in $L$, with a choice of $\mathbb{L}$-compressing disks $\sD(\Gamma)=\{D_1,\ldots,D_b\}$, one for each curve, is said to be an $\L$-compressing system for $L$ if $D_i\cap D_j=\emptyset$ for all $i,j\in[b]$ and the homology classes of the curves in $\Gamma$ form a basis of $H_1(L;\Z)$. Here we use the notation $[n]$ for the set $\{1,\ldots,n\}$, where $n\in\N$ is a natural number. Two $\mathbb{L}$-compressing systems $\Gamma,\Gamma'$ for $L$ are said to be equivalent if there exists a sequence of triple point moves and bigon moves, i.e.~Reidemeister IIIs and non-dangerous tangencies, that applied to the curves in $\Gamma$ lead to the curves in $\Gamma'$. See Section \ref{sec:curves} for further discussions and cf.~Figure \ref{fig:Moves} and Figure \ref{fig:Moves2} for such moves. Given a Lagrangian filling $L$ with an $\L$-compressing system $\Gamma$ and a disk $D\in\sD(\Gamma)$, Lagrangian disk surgery produces an embedded exact Lagrangian filling $\mu_D(L)$. For details on Lagrangian disk surgeries, see Section \ref{sec:Symplectic} below or cf.~\cite{Polterovich_Surgery,MLYau17}.

A priori, it is not clear whether $\mu_D(L)$ inherits an $\L$-compressing system from $(L,\Gamma)$: curves in $\Gamma$ might become immersed under Lagrangian disk surgery. For instance, see Example \ref{ex:Example_Immersed} in Section \ref{ssec:effectLag} and adjacent discussion. Therefore, the set of curves $\mu_D(\Gamma)$ obtained from $\Gamma$ after Lagrangian disk surgery on $D$ might not be an $\L$-compressing system. This is a known problem in this approach, cf.~Section \ref{ssec:iteration} for more details. The main technical achievement of this manuscript is to show that this can be corrected, showing that certain $\L$-compressing systems persist under arbitrary sequences of Lagrangian disk surgeries. The quiver with potential that we introduce and study is a crucial ingredient for implementing this correction.

From the previous subsection, we recall that \cite{GSW} constructs a cluster algebra structure in $\C[X(\La_\beta,T)]$, cf.~also \cite{CasalsWeng22,ShenWeng}. It has the property that, in all cases relevant for this article, an embedded exact Lagrangian filling $L$ for $\La_\beta$ endowed with an $\L$-compressing system $\Gamma$ defines a cluster seed $\mathfrak{c}(L,\Gamma)$ for that cluster structure. In particular, it defines a quiver $Q(\mathfrak{c}(L,\Gamma))$ for that cluster seed $\mathfrak{c}(L,\Gamma)$. In addition, the cluster structure is such that the mutable (a.k.a.~unfrozen) vertices of $Q(\mathfrak{c}(L,\Gamma))$ are in bijection with the curves of the $\L$-compressing system $\Gamma$. The frozen vertices of $Q(\mathfrak{c}(L,\Gamma))$ are related to base points and have no significant role in this article, cf.~\cite{CasalsWeng22,GSW} for details on frozens.

The core result of the manuscript is the following theorem, which implies the surjectivity of $\mathfrak{C}$ for all positive braids $\beta$. The entire article is devoted to developing a new technique that proves this result.

\begin{thm}\label{thm:main}
Let $\La_\beta\sse(\R^3,\xi_{st})$ be the Legendrian link associated to a positive braid word $\beta$, and $T\sse\La_\beta$ a set of marked points with one marked point per component.

Then there exists an embedded exact Lagrangian filling $L\sse(\R^4,\la_{st})$ of $\La_\beta$ and an $\mathbb{L}$-compressing system $\Gamma$ for $L$ such that the following holds:

\begin{itemize}
    \item[(i)] If $\mu_{v_\ell}\ldots\mu_{v_1}$ is any sequence of mutations, where $v_1,\ldots,v_\ell$ are mutable vertices of the quiver $Q(\mathfrak{c}(L,\Gamma))$ associated to the cluster seed $\mathfrak{c}(L,\Gamma)$ of $L$ in $\C[X(\La_\beta,T)]$, then there exists a sequence of embedded exact Lagrangian fillings $L_k$ of $\La_\beta$, each equipped with an $\mathbb{L}$-compressing system $\Gamma_k$, with associated cluster seeds
    $$\mathfrak{c}(L_k,\Gamma_k)=\mu_{v_k}\ldots\mu_{v_1}(\mathfrak{c}(L,\Gamma))$$
    in $\C[X(\La_\beta,T)]$, for all $k\in[\ell]$.\\

    \item[(ii)] Each $\mathbb{L}$-compressing system $\Gamma_k$ for $L_k$ is such that Lagrangian disk surgery on $L_k$ along any Lagrangian disk in $\sD(\Gamma_k)$ yields an $\mathbb{L}$-compressing system. In addition, $\Gamma_{k+1}$ is equivalent to such an $\mathbb{L}$-compressing system via a sequence of triple point moves and local bigon moves.\hfill$\Box$\\

\end{itemize}
\end{thm}
\color{black}
Theorem \ref{thm:main} is a reconstruction result that we prove by first introducing a new object: a quiver with potential $(Q(\SC),W(\SC))$ associated to a curve configuration $\SC$. The algebraic notion of a quiver with potential\footnote{The word potential has different meanings in the literature. Here a quiver with potential is meant in the sense of \cite{DWZ}. It is unrelated to the potentials from \cite{PascaleffTonkonog20}, for instance, which count holomorphic disks bounded by monotone Lagrangians and thus vanish for exact Lagrangians.} was introduced in \cite{DWZ}. The curve configurations that we use for Theorem \ref{thm:main} can be extracted from the Legendrian link $\La_\beta$ and its initial Lagrangian filling $L$. Intuitively, this quiver with potential encodes an arboreal skeleton for the Weinstein relative pair $(\C^2,\La_\beta)$, cf.~\cite[Section 2]{Eliashberg18_Revisited} or \cite[Section 2]{CasalsLagSkel}.

In a nutshell, Theorem \ref{thm:main} states that we can arbitrarily perform Lagrangian disk surgeries to the initial Lagrangian filling $L$ of $\La_\beta$ so as to construct new Lagrangian fillings reaching any cluster seed. Sections \ref{sec:curves} and \ref{sec:Symplectic} explain how these Lagrangian surgeries induce a mutation to the associated quivers with potential. See \cite[Definition 5.5]{DWZ} or Section \ref{sssec:DWZ} below for the definition of mutation in this context. In general, quivers with potential cannot be arbitrarily mutated. Nevertheless, \cite[Corollary 8.2]{DWZ} shows that quivers with potential that are rigid can be arbitrarily mutated. See \cite[Section 6.10]{DWZ} or Section \ref{ssec:nondeg_def} below for the definition of rigid. Our results will show that the right-equivalence class of this particular quiver with potential $(Q(\SC),W(\SC))$ associated to $\La_\beta$ is rigid, it is invariant under triple point moves and bigon moves, and changes according to a mutation (of a quiver with potential) under Lagrangian disk surgeries.



\color{black}
Thanks to the constructions in Sections \ref{sec:curves} and \ref{sec:QP_PlabicFence}, we are able to use this rigidity and properties in Section \ref{sec:Symplectic} to show that we can geometrically realize any arbitrary sequence of algebraic mutations in the cluster algebra $\C[X(\La_\beta,T)]$ by a sequence of Lagrangian disk surgeries on that arboreal skeleton for $(\C^2,\La_\beta)$. Near the entirety of the manuscript is devoted to the construction, study and use of this new quiver with potential.

\begin{remark}\label{rmk:Weinstein} Each embedded exact Lagrangian filling $L$ in Theorem \ref{thm:main} gives a closed embedded exact Lagrangian surface $\overline{L}$ in $W(\La_\beta)$, the Weinstein 4-fold obtained by attaching a Weinstein 2-handle to each component of $\La_\beta$. If the cluster seeds of $L,L'$ are different, then $\overline{L}$ and $\overline{L'}$ are not Hamiltonian isotopic in $W(\La_\beta)$, only possibly Lagrangian isotopic, cf.~\cite[Proposition 7.11]{CasalsNg}. Theorem \ref{thm:main} shows that each of these (typically infinitely many) different closed exact Lagrangians in $W(\La_\beta)$ is embedded in a different closed arboreal Lagrangian skeleton, cf.~\cite[Section 2.4]{CasalsLagSkel}. Note also that for each $(L_k,\Gamma_k)$,  \cite[Theorem 1.13]{GPS2} implies that the object $\mathscr{L}_k:=C_1\oplus\ldots\oplus C_{\pi_0(\La_\beta)}\oplus T^*_{p_1}D_1\oplus\ldots\oplus T^*_{p_b}D_b$ is a compact generator of the wrapped Fukaya category of $W(\La_\beta)$, where $p_i\in D_i$ are interior points, $T^*_{p_i}D_i$ the local cotangent fibers and $C_j$ are the co-cores of the Weinstein 2-handles. Therefore, Theorem \ref{thm:main} geometrically constructs a compact generator for each vertex of the cluster exchange graph of the cluster algebra $\C[X(\La_\beta,T)]$. In particular, when the dg-algebras $\mbox{End}(\mathscr{L}_k)$ are non-positively graded, Theorem \ref{thm:main} geometrically constructs bounded $t$-structures for the wrapped Fukaya category.\hfill$\Box$
\end{remark}

{\bf The map is surjective.} The construction of the cluster algebra structure for $\C[X(\La_\beta,T)]$ has the following property, relating Theorem \ref{thm:main} to the map $\mathfrak{C}$; cf.~\cite{CasalsWeng22,GSW} or Section \ref{ssec:aug} below. Given a Lagrangian filling $L$ of $\La_\beta$ and an $\L$-compressing system $\Gamma$, the cluster variables in $\mathfrak{c}(L,\Gamma)$ are indexed by the curves in $\Gamma$ and the arrows in the quiver of $\mathfrak{c}(L,\Gamma)$ record geometric intersections of curves in $\Gamma$. The cluster variables are described by a microlocal parallel transport, cf.~\cite[Section 4]{CasalsWeng22}. In addition, the cluster seed $\mu_v(\mathfrak{c}(L,\Gamma))$ obtained by algebraically mutating at a vertex $v=\gamma$ of the quiver, indexed by some $\gamma\in\Gamma$, is precisely $\mathfrak{c}(\mu_{D_\gamma}(L,\Gamma))$, where $D_\gamma$ is the $\L$-compressing disk associated to the curve $\gamma\in\Gamma$. Therefore, Theorem \ref{thm:main} implies the desired surjectivity:

\begin{cor}\label{cor:main}
Let $\La_\beta\sse(\R^3,\xi_{st})$ be the Legendrian link associated to a positive braid word $\beta$, $T\sse\La_\beta$ a set of marked points, one per component, and $X(\La_\beta,T)$ its augmentation variety. Then 
$$\mathfrak{C}:\Lag^c(\La_\beta)\lr\Seed(X(\La_\beta,T))$$
is surjective, i.e.~ each cluster seed is induced by an embedded exact Lagrangian filling endowed with an $\L$-compressing system.
\end{cor}

We conjecture that the map $\mathfrak{C}$ is injective. Corollary \ref{cor:main} and a proof that $\mathfrak{C}$ is injective would settle the core symplectic aspects of the classification of Hamiltonian isotopy classes of Lagrangian fillings for $\La_\beta$, proving that they are equivalent to studying a class of cluster algebras, an algebraic matter.\\


{\bf Acknowledgements}. We are grateful to Mikhail Gorsky, James Hughes, Daping Weng and Alex Wright for useful discussions. We thank the referees for their comments and suggestions. We also thank Joel Hass for kindly explaining his work with P.~Scott to us. R.~Casals is supported by the NSF CAREER DMS-1942363, a Sloan Research Fellowship of the {\color{black} Alfred P. Sloan Foundation} and a UC Davis College of L\&S Dean's Fellowship. Sloan Foundation. H.~Gao is supported by a Tsinghua University start-up grant and {\color{black}the Tsinghua University Dushi Program}.

{\bf Notation}. We denote by $[n]$ the set $\{1,\ldots,n\}$, as indicated above. The group of compactly supported diffeomorphisms of a smooth manifold $\Sigma$ is denoted by $\mbox{Diff}^c(\Sigma)$. In this article, a quiver will refer to a multidigraph with no loops, but possibly with 2-cycles. The set of vertices of $Q$ is denoted by $Q_0$ and its set of arrows by $Q_1$. We often abbreviate {\it quiver with potential} to QP, as in \cite{DWZ}. For specificity, we use the ground ring $R=\C$ for the (complete) path algebra of the quiver.

{\bf Structure of the manuscript}. The study and use of our quiver with potential (QP) is as follows:
\begin{itemize}
    \item[(i)] Section \ref{sec:curves} introduces and develops the new concept: the quiver with potential $(Q(\SC),W(\SC))$ associated to a curve configuration $\SC$. Proposition \ref{lem:triplepoint} and Lemma \ref{lem:bigon} show that the right equivalence class of $(Q(\SC),W(\SC))$ is invariant under triple point moves and local bigon moves. Then Lemma \ref{lem:reduced_parts} and Proposition \ref{prop:quivermutation} show that the QP associated to a reduction of the $\g$-exchange on $\SC$ yields the QP-mutation of $(Q(\SC),W(\SC))$ at the corresponding vertex.\\

    \item[(ii)] Section \ref{sec:QP_PlabicFence} introduces the curve configurations $\SC(\bG)$ associated to plabic fences $\bG$. Proposition \ref{thm:QPnondeg} shows that the curve QP associated to $\SC(\bG)$ is rigid.\\

    \item[(iii)] Section \ref{sec:Symplectic} constructs an initial Lagrangian filling with an $\L$-compressible system whose associated curve configuration is of the form $\SC(\bG)$ for a plabic fence $\bG$. It then develops a few necessary technical results until Section \ref{ssec:proof_main}, where we show how to use the rigidity of the QP for $\SC(\bG)$ to prove Theorem \ref{thm:main}. It also includes Theorem \ref{thm:main_general}, a variant of Theorem \ref{thm:main}.
\end{itemize}

\section{The curve quiver with potential}\label{sec:curves}

Let $\Sigma$ be an oriented surface and $\SC=\{\g_1,\ldots,\g_b\}$, $b\in\N$, a collection of embedded oriented closed connected curves $\g_i\sse\Sigma$, $i\in[b]$, whose pairwise intersections are all transverse. Two such collections $\SC$ and $\SC'$ will be considered equal if there {\color{black} exists} a diffeomorphism $\varphi\in\mbox{Diff}^c(\Sigma)$ such that $\varphi(\SC)=\SC'$. In particular, two such collections $\SC$ and $\SC'$ related by compactly supported isotopies are considered to be equal. We assume that the only intersection points of curves in $\SC$ are double intersection points, i.e.~exactly two curves $\g_i,\g_j\in\SC$ intersect transversely at a given intersection point. From now onward, we always assume that $\Sigma$ is compact or at least of finite topological type.

\begin{definition}\label{def:configuration}
Let $\Sigma$ be an oriented surface and $\SC=\{\g_1,\ldots,\g_b\}$, $b\in\N$, a collection of embedded oriented closed connected curves $\g_i\sse\Sigma$, $i\in[b]$. By definition, $\SC$ is said to be a curve configuration, or simply a configuration, if the classes $[\g_1],\ldots,[\g_b]$ in $H_1(\Sigma;\Z)$ form a basis.\hfill$\Box$ 
\end{definition}

In particular, for a curve configuration $\SC=\{\g_1,\ldots,\g_b\}$ in $\Sigma$, we must have $b=b_1(\Sigma)$. The collections of curves $\SC$ needed to prove Theorem \ref{thm:main} will be configurations, i.e.~consist of $b_1(\Sigma)$ curves whose homology classes span $H_1(\Sigma;\Z)$. Throughout this manuscript, from now onward, we assume all collections of curves $\SC$ that we use satisfy such hypothesis, i.e.~they are curve configurations.


\subsection{The QP $(Q(\SC),W(\SC))$ associated to $\SC$} Let $\SC$ be a curve configuration. Consider the integers
{\color{black} $$A_{ij}(\SC):=|\{p\in\g_i\cap\g_j: \mbox{sign}(p)\mbox{ is positive}\}|\in\N,$$}
where $\mbox{sign}(p)=\pm$ is the sign of the intersection point $p$. 

\begin{definition}
Let $\SC$ be a curve configuration. The quiver $Q(\SC)$ is defined to have vertex set $Q(\SC)_0=\{\g_1,\ldots,\g_b\}$ and the arrow set $Q(\SC)_1$ is given by the condition that the number of arrows from $\g_i$ to $\g_j$ is $A_{ij}(\SC)$, $i,j\in[b]$.\hfill$\Box$
\end{definition}

The quiver $Q(\SC)$ is referred to as the curve quiver of $\SC$. Since the curves in $\SC$ are embedded, the quiver $Q(\SC)$ has no loops; it might nevertheless have 2-cycles. The arrows from the vertex associated to $\g_i$ to that of $\g_j$ are in natural bijection with the positive intersection points $p\in\g_i\cap\g_j$ between $\g_i$ and $\g_j$: we indistinctly identify arrows in $Q(\SC)$ and such intersection points $p\in\Sigma$. Let us now discuss the preliminaries to introduce the potential $W(\SC)$.

\begin{definition} By definition, an $\ell$-gon $P$ bounded by $\SC$, $\ell\in\N$ and $\ell\geq2$, is a closed contractible subset $P\sse\Sigma$ with a piecewise smooth boundary $\dd P$ and embedded interior such that:

\begin{itemize}
    \item[(i)] There are $\ell$ connected components in $\dd P\setminus V$, where $V\sse\dd P$ is the set of non-smooth points of $\dd P$, i.e.~$V$ is the set of vertices of $P$. That is, there are $\ell$ sides to $P$.\\
    
    \item[(ii)] For each smooth connected component $\dd P_j\sse(\dd P\setminus V)$, $j\in[\ell]$, there exists a $\g_{i_j}$ such that $\dd P_j\sse\g_{i_j}$ and the orientations either coincide for all $j\in[\ell]$ or they are opposite for all $j\in[\ell]$. That is, each oriented side of $P$ is an oriented subspace of a curve in $\SC$ or each oriented side of $P$ is an oriented subspace of a curve in $\overline{\SC}$. Here $\overline{\SC}=\{-\g_1,\ldots,-\g_b\}$ denotes the same configuration of curves as in $\SC$ where we have switched the orientation of each curve $\g\in\SC$.\\

    \item[(iii)] At a small enough neighborhood $U\sse\Sigma$ of a vertex $v\in V$, which locally is given by the intersection of two curves $\g_i,\g_j\in\SC$, the intersection $P\cap(U\setminus(U\cap(\g_i\cup\g_j)))$ is a unique quadrant. That is, vertices of an $\ell$-gon only use one of the quadrants; in a combinatorial sense, $\ell$-gons bounded by $\SC$ are convex.\hfill$\Box$
\end{itemize}
\end{definition}

\noindent See Figure \ref{fig:Example_Polygon} for an instance of an $\ell$-gon bounded by a curve configuration $\SC$.

\begin{center}
	\begin{figure}[H]
		\centering
		\includegraphics[scale=0.5]{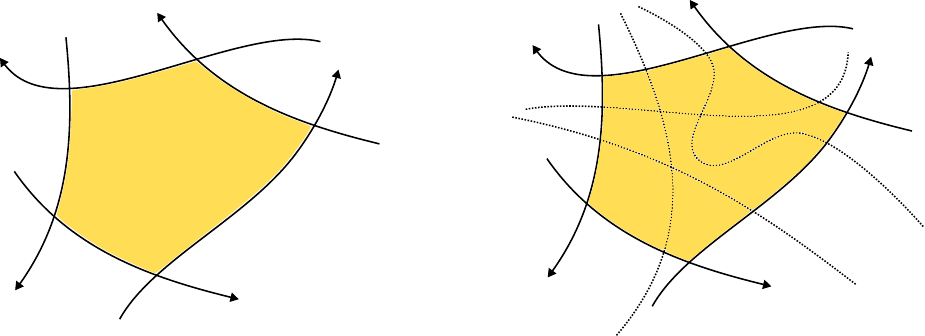}
		\caption{(Left) An example of an $\ell$-gon in a curve configuration $\SC$, with $\ell=5$, drawn in yellow. Pieces of curves in $\SC$ are drawn in black. (Right) A drawing illustrating that an $\ell$-gon for a configuration $\SC$ might have pieces of curves in $\SC$ passing through the interior of the $\ell$-gon.}
		\label{fig:Example_Polygon}
	\end{figure}
\end{center}

\noindent An $\ell$-gon $P$ bounded by $\SC$ determines a cyclically ordered set of vertices. Conversely, it is uniquely determined by its cyclically ordered set of vertices and its orientation, clockwise or counter-clockwise. Since the vertices of $P$ must be intersection points between the curves in $\SC$, which bijectively correspond to arrows in $Q(\SC)$, such $P$ is uniquely determined by a (cyclic) word of composable arrows starting and ending at the same vertex, along with its orientation. In other words, by a signed monomial in $\HH_0(Q(\SC))$, the trace space of the path algebra $\C\langle Q(\SC)\rangle$ of $Q(\SC)$. This correspondence is written $P=v_1\dots v_\ell$ where $v_1,\ldots,v_\ell$ are the vertices of $P$ read according to the order induced by the orientation of $\dd P$.

\begin{remark}
Note that the connected components of $\Sigma\setminus(\g_1\cup\ldots\cup \g_b)$ might or might not be $\ell$-gons bounded by $\SC$, due to the orientations of the curves in $\SC$. Also, typically polygons bounded by $\SC$ are not connected components of $\Sigma\setminus(\g_1\cup\ldots\cup \g_b)$ as they might have curves in $\SC$ cross through them, see e.g.~Figure \ref{fig:Example_Polygon} (right).\hfill$\Box$
\end{remark}

Let $\Gamma_\ell^+$, resp.~$\Gamma_\ell^-$, be the set of $\ell$-gons $P$ bounded by $\SC$ where the orientation of $\dd P$ coincides with, resp.~it is opposite to, the orientations of $\Sigma$.

For each intersection point of two curves $\g_i,\g_j\in\SC$ that represents an arrow from $\g_i$ to $\g_j$ in $Q(\SC)$ we decorate (shade) two consecutive quadrants as follows. The tangent vectors of $\g_i$ and $\g_j$, in this order, are an oriented basis of the tangent space at the intersection point. Then we shade the two quadrants adjacent to the part of $\g_j$ where $\g_j$ points outwards from the intersection point; see Figure \ref{fig:OrientationSigns}, where the shading of these two quadrants is depicted. In other words, the two shaded quadrants are those in the half-plane to the left of $\g_i$, where we traverse $\g_i$ according to its orientation. Similarly, if the intersection point represents an arrow from $\g_j$ to $\g_i$ in $Q(\SC)$, so that the basis $\g_i$ and $\g_j$ (in this order) gives the reverse orientation, then we shade the two quadrants adjacent to the part of $\g_i$ where $\g_i$ points outwards from the intersection point.

\begin{center}
	\begin{figure}[H]
		\centering
		\includegraphics[scale=0.5]{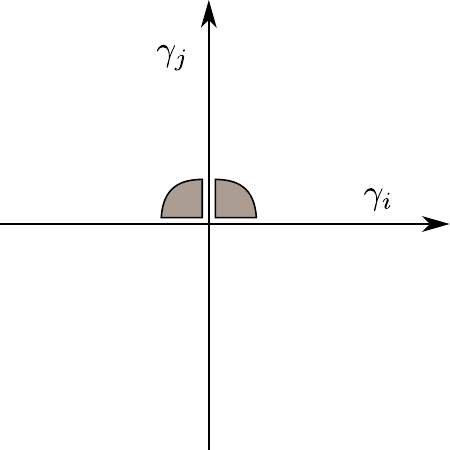}
		\caption{The shading of two quadrants at an intersection point. The two consecutive quadrants that are shaded are the first and the second quadrants, where we take the oriented basis as the two axes.}
		\label{fig:OrientationSigns}
	\end{figure}
\end{center}

Given a polygon $P=v_1\ldots v_\ell$ bounded by $\SC$, each vertex $v_i$ of $P$ is assigned the sign $\sigma(v_i;v_1\ldots v_\ell)=1$ if $v_i$ uses a non-shaded quadrant, and the sign $\sigma(v_i;v_1\ldots v_\ell)=-1$ if $v_i$ uses a shaded quadrant. By definition, the vertex sign $\sigma(v_1\ldots v_\ell)$ of the polygon $P$ is
$$\sigma(v_1\ldots v_\ell)=\prod_{i=1}^{\ell}\sigma(v_i;v_1\ldots v_\ell).$$

\begin{definition}\label{def:quiverC}
The potential $W(\SC)\in \HH_0(Q(\SC))$ of $Q(\SC)$ is defined by
$$W(\SC)=\sum_{v_1\dots v_\ell\in\Gamma_\ell^+}\sigma(v_1\dots v_\ell)\cdot v_\ell\dots v_1\quad-\sum_{w_1\dots w_\ell\in\Gamma_\ell^-}\sigma(w_1\dots w_\ell)\cdot w_1\dots w_\ell,$$
where the sums run over all possible $\ell\in\N$, $\ell\geq2$, and all possible elements of $\Gamma_\ell^{\pm}$. The pair $(Q(\SC),W(\SC))$ is referred to as the curve quiver with potential of $\SC$.  We often abbreviate and refer to such a pair as a {\it curve QP} or a {\it cQP}.\hfill$\Box$
\end{definition}

\noindent In Definition \ref{def:quiverC} we always write the vertices on the boundary left to right as read counter-clockwise; in this manner the monomial is an actual cycle in the quiver $Q(\SC)$.

We often consider QPs up to right-equivalence, i.e.~up to automorphisms of the path algebra, which we now define. Following the notation from \cite{DWZ}, let $\C\langle Q\rangle$ denote the path algebra of a quiver $Q$ and $\C\langle\langle Q\rangle\rangle$ its completion. Here we view $\C\langle\langle Q\rangle\rangle$ as a topological algebra via the $\mathfrak{m}$-adic topology, where $\mathfrak{m}$ is the two-sided ideal
generated by the arrow span of $Q$. The following is \cite[Definition 4.2]{DWZ}:

\begin{definition}[\cite{DWZ}]\label{def:right_equivalent}
Let $(Q,W)$ and $(Q',W')$ be two QPs on the same vertex set. A right-equivalence between $(Q,W)$ and $(Q',W')$ is a $\C$-algebra isomorphism $\varphi:\C\langle\langle Q\rangle\rangle\lr\C\langle\langle Q\rangle\rangle$ such that $\varphi|_\C=\mbox{id}$ and $\varphi(W)$ is cyclically equivalent to $W'$.\hfill$\Box$
\end{definition}

In Definition \ref{def:right_equivalent}, two potentials $W,W'$ are said to be cyclically equivalent if their difference $W-W'$ lies in the
closure of the span of all elements of the form $a_1\cdots a_d-a_2\cdots a_da_1$, where $a_1,\ldots,a_d$ is a cyclic path, cf.~\cite[Definition 3.2]{DWZ}.

\subsection{Properties of curve QPs under planar moves}\label{ssec:planar_moves} A configuration of curves $\SC$ can be modified by compactly supported smooth isotopies of $\Sigma$. The combinatorics of $\SC$, including the intersection pattern and polygons bounded by $\SC$, do not change under such isotopies. We can modify $\SC$ more significantly by choosing one curve $\g\in\SC$ and smoothly isotope it to another curve $\g'\sse\Sigma$. The new configuration $\SC':=(\SC\cup\{\g'\})\setminus\{\g\}$ has different combinatorics than that of $\SC$. In this article, we will consider two configurations $\SC$ and $\SC'$ equivalent if they can be connected by a sequence of triple moves and bigon moves; moves that we now introduce.
\subsubsection{Behavior under triple point moves.} Let $\g_1,\g_2,\g_3\in\SC$. The two moves in Figure \ref{fig:Moves} will be referred to as triple point moves. In general, triple moves will refer to any local move that is smoothly isotopic to either of the two moves in Figure \ref{fig:Moves}, possibly after switching orientations of arrows. The two moves in the figure capture all triple moves, up to rotational symmetries. A triple move applied to a configuration $\SC$ is a local operation: there exists a neighborhood $U\sse\Sigma$ such that $\SC\cap U$ is as in Figure \ref{fig:Moves} (left), the new configuration $\SC'$ coincides with $\SC$ outside of $U$, and $\SC'\cap U$ is as in Figure \ref{fig:Moves} (right). That is, this is a change in a configuration $\SC$ which is compactly supported, as the boundary conditions in the local model in Figure \ref{fig:Moves} coincide before and after the move.

\begin{center}
	\begin{figure}[H]
		\centering
		\includegraphics[scale=0.5]{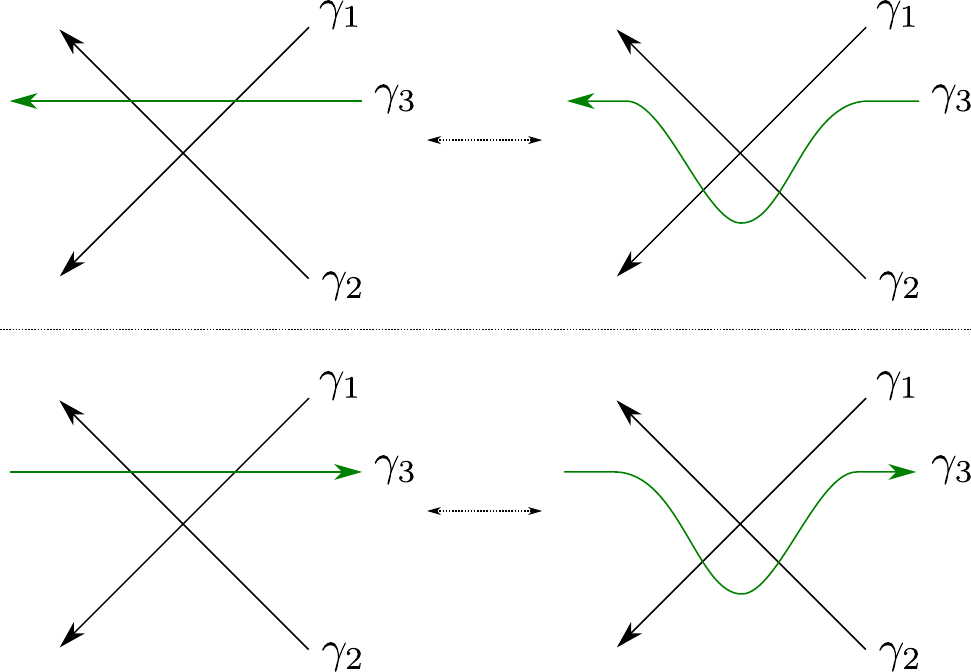}
		\caption{The two triple point moves.}
		\label{fig:Moves}
	\end{figure}
\end{center}

Let us study how the curve QP $(Q(\SC),W(\SC))$ behaves under such triple point moves applied to $\SC$.

\begin{prop}\label{lem:triplepoint}
Let $(Q(\SC),W(\SC))$ be a curve QP associated to $\SC$. Then $(Q(\SC),W(\SC))$ is invariant under triple point moves, up to right-equivalence.
\end{prop}



\begin{proof} There are two cases to consider for a triple point move, depending on the orientations, see Figure \ref{fig:Moves} for the two cases. First, we consider the local model in Figure \ref{fig:TriplePointMove_Case1_Before} (left) and denote by $\g_i^{in}$, resp.~$\g_i^{out}$, the tail, resp.~the head, of the segment of $\g_i$ in the local model; the head is where the arrow is drawn.

\begin{center}
	\begin{figure}[H]
		\centering
		\includegraphics[scale=0.6]{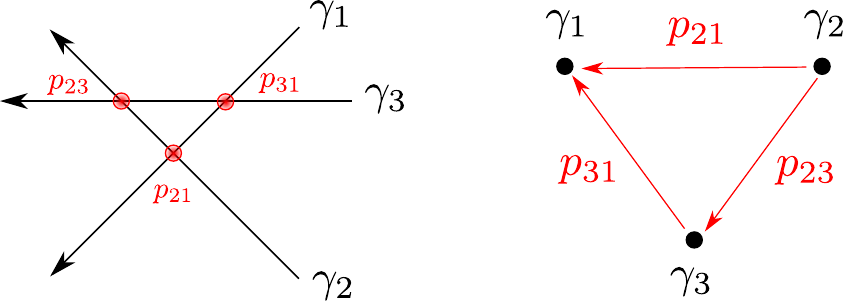}
		\caption{(Left) Local model for Case I before triple move. (Right) Local quiver. The local variables for this local model are $p_{23},p_{31}$ and $p_{21}$, corresponding to the three intersection points (left), or equivalently, arrows of the quiver (right).}
		\label{fig:TriplePointMove_Case1_Before}
	\end{figure}
\end{center}

Let $R=(r_{ij})$ be the matrix such that $r_{ij}$ equals the sums of monomials on $p_{lk}$ that correspond to corners of regions that intersect the local model entering at $\gamma_l^{in}$ and exiting at $\gamma_k^{out}$, from either side. That is, a monomial $p_{i_1i_2}\cdots p_{i_{s-1}i_s}$ on the three variables $p_{lk}$ appears as a summand in the entry $r_{ij}$ if and only if there exists an $\ell$-gon in $\SC$ whose vertices in the boundary correspond to a monomial (in the global potential) of the form $\alpha(p_{i_1i_2}\cdots p_{i_{s-1}i_s})\beta$, where $\alpha,\beta$ are monomials without $p_{lk}$ variables, and the $\ell$-gon enters the local model via $\g_1^{in}$ and exits via $\g_2^{out}$. Figure \ref{fig:TriplePointMove_Case1_ExampleRegions} illustrates two regions contributing to the entry $r_{12}$ of $R$ for this local model.

\begin{center}
	\begin{figure}[H]
		\centering
		\includegraphics[scale=0.6]{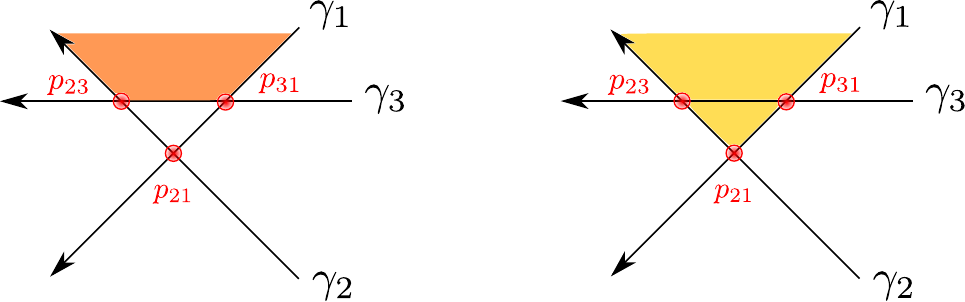}
		\caption{(Left) A region, in orange, that contributes $p_{23}p_{31}$ to the $r_{12}$ entry of the matrix $R$, as it is of the form $\alpha(p_{23}p_{31})\beta$ and it enters via $\gamma_1^{in}$ and exits via $\g_2^{out}$. (Right) Another region, in yellow, that contributes $p_{21}$ to the same entry $r_{12}$.}
		\label{fig:TriplePointMove_Case1_ExampleRegions}
	\end{figure}
\end{center}

Then the matrix $R$ reads
$$R=\begin{pmatrix}
0 & p_{21}+p_{23}p_{31} & p_{31}\\
p_{21} & 0 & p_{23}\\
p_{31} & p_{23} & 0
\end{pmatrix}.$$
Indeed, the second entry in the first row corresponds to regions that enter in the upper right through $\gamma_1^{in}$ and exist through $\gamma_2^{out}$ on the upper left. There are two such regions: using $p_{21}$ or using $p_{23}p_{31}$, see Figure \ref{fig:TriplePointMove_Case1_ExampleRegions}. Since both regions are oriented clockwise, there is an overall positive sign, and since none of the quadrants being used are shaded, the sign remains positive. The third entry of the first row is similar. The first entry $p_{21}$ and the third entry $p_{23}$ on the second row are actually $p_{21}=-(-p_{21})$ and $p_{23}=-(-p_{23})$: both of the regions they record are oriented counter-clockwise, which gives an overall minus sign, and in addition both of them use a unique quadrant which happens to be shaded. This introduces another minus sign and therefore the entry has two minus signs, thus a positive coefficient in the end. The first entry on the third row is also $p_{31}=-(-p_{31})$, in that sense, whereas the second entry on the third row has two positive signs directly. From now onwards we will compute with both types of signs in mind (orientation sign and shaded quadrants signs), without further specifying when a positive sign is actually an even number of negative signs. The local quiver, recording the intersections that occur only in the local model, is depicted in Figure \ref{fig:TriplePointMove_Case1_Before} (right).\\

After the triple point move we have the local model in Figure \ref{fig:TriplePointMove_Case1_After} (left) and its corresponding local quiver in Figure \ref{fig:TriplePointMove_Case1_After} (right). Note that the quivers in Figure \ref{fig:TriplePointMove_Case1_Before} (right) and Figure \ref{fig:TriplePointMove_Case1_After} (right), before and after, are identical: thus it is clear that $Q(\SC)$ is invariant under this particular triple point move.

\begin{center}
	\begin{figure}[H]
		\centering
		\includegraphics[scale=0.6]{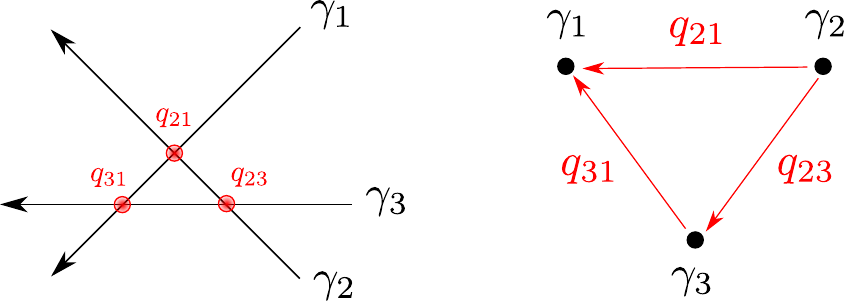}
		\caption{(Left) Local model for Case I after triple move. (Right) Local quiver.}
		\label{fig:TriplePointMove_Case1_After}
	\end{figure}
\end{center}

The matrix of regions $R'$ after the triple move is computed as with $R$ above. It reads:
$$R'=\begin{pmatrix}
0 & q_{21} & q_{31}\\
q_{21}-q_{23}q_{31} & 0 & q_{23}\\
q_{31} & q_{23} & 0
\end{pmatrix}$$

Note that the signs in the entry $q_{21}-q_{23}q_{31}$ are indeed correct: the region associated to $q_{21}$ is oriented counter-clockwise and it uses a shaded quadrant, thus it is positive, and the region associated to $q_{23}q_{31}$ is oriented counter-clockwise and it uses two shaded quadrants, thus it is negative. Let us now compare $R$ and $R'$, which keep track of the regions in the potential $W(\SC)$ before and after a triple point move.

For that, consider the automorphism $\phi\in\Aut(\C\langle Q(\SC)\rangle)$ of the path algebra which is the identity on all arrows $p_{ij}$ except for $(i,j)=(2,1)$ and sends the arrow $p_{21}$ to
$$p_{21}\mapsto p_{21}-p_{23}p_{31}.$$
After applying this automorphism, $R$ becomes
$$\phi(R)=\begin{pmatrix}
0 & (p_{21}-p_{23}p_{31})+p_{23}p_{31} & p_{31}\\
(p_{21}-p_{23}p_{31}) & 0 & p_{23}\\
p_{31} & p_{23} & 0
\end{pmatrix}=
\begin{pmatrix}
0 & p_{21} & p_{31}\\
p_{21}-p_{23}p_{31} & 0 & p_{23}\\
p_{31} & p_{23} & 0
\end{pmatrix}.$$

The matrix $R'$ and $\phi(R)$ are now related by the trivial relabeling $p_{ij}\longmapsto q_{ij}$. Note that the automorphism $\phi\in\Aut(\C\langle Q(\SC)\rangle)$ is defined on the entirety of the path algebra of $Q(\SC)$, not just the variables $p_{ij}$ for the local model. Furthermore, the matrices $R'$ and $\phi(R)$ match, up to trivial relabeling, and their entries are all computing regions with all possible boundary conditions in the local model, with the specific endpoints $\gamma_i^{in}$ and $\gamma_j^{out}$ being recorded in different components. Therefore this equivalence given by $p_{21}\mapsto p_{21}-p_{23}p_{31}$, and otherwise the identity, extends to an equivalence between the global potentials $W(\SC)$ and $W(\SC')$, where $\SC'$ is the given configuration $\SC$ after applying a triple point move. As a consequence, since the quiver $Q(\SC)=Q(\SC')$ remains invariant, the quivers with potential $(Q(\SC),W(\SC))$ and $(Q(\SC'),W(\SC'))$ are right-equivalent. Thus, the right-equivalence class of the quiver with potential $(Q(\SC),W(\SC))$ is invariant under this triple point move.\\

Second, we now consider the other local model for a triple move, as depicted in Figure \ref{fig:TriplePointMove_Case2_Before} (left). The local quiver is drawn in Figure \ref{fig:TriplePointMove_Case2_Before} (right).

\begin{center}
	\begin{figure}[H]
		\centering
		\includegraphics[scale=0.6]{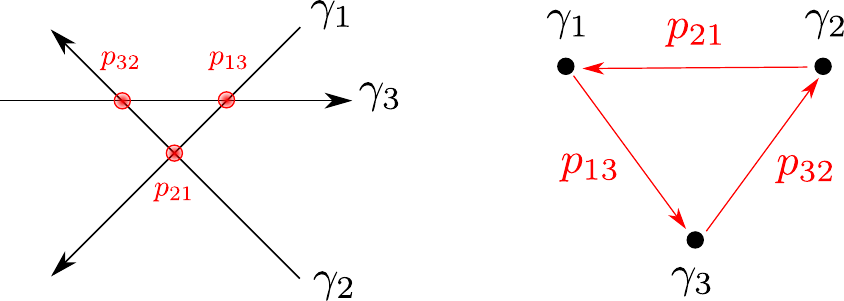}
		\caption{(Left) Local model for Case II before triple move. (Right) Local quiver.}
		\label{fig:TriplePointMove_Case2_Before}
	\end{figure}
\end{center}

In the same notation as above, the matrix $R$ before the triple point move is
$$R=\begin{pmatrix}
0 & p_{21} & p_{13}\\
p_{21} & 0 & p_{32}\\
p_{13} & p_{32} & 0
\end{pmatrix},$$
except that in this case we also have the closed triangle region associated to the monomial $p_{13}p_{32}p_{21}$, entirely contained in this local model. Notice that this is a clockwise oriented triangle with none of its (interior) quadrants being shaded: the sign is therefore positive.

\begin{center}
	\begin{figure}[H]
		\centering
		\includegraphics[scale=0.6]{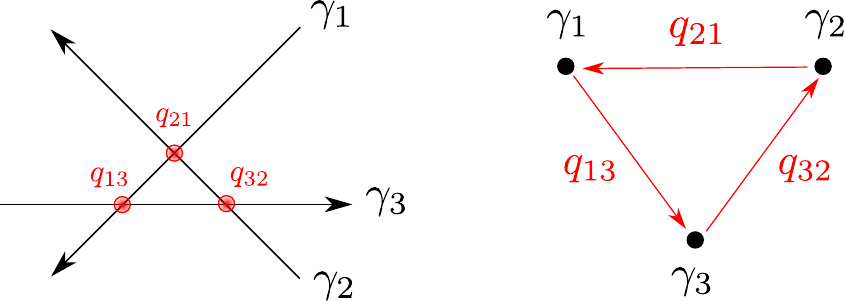}
		\caption{(Left) Local model for Case II after triple move. (Right) Local quiver.}
		\label{fig:TriplePointMove_Case2_After}
	\end{figure}
\end{center}

After the triple point move we have the local model depicted in Figure \ref{fig:TriplePointMove_Case2_After} (left), whose quiver is illustrated in Figure \ref{fig:TriplePointMove_Case2_After} (right). Therefore, the quiver $Q(\SC)$ also remains invariant under this second type of triple point move. In the notation as above, the matrix $R'$ after the triple point move is
$$R'=\begin{pmatrix}
0 & q_{21} & q_{13}\\
q_{21} & 0 & q_{32}\\
q_{13} & q_{32} & 0
\end{pmatrix},$$
and we also have the triangle region associated to the monomial $q_{13}q_{32}q_{21}$, again entirely enclosed in this local model. The sign of this triangle is indeed positive: it is oriented counter-clockwise and its three (interior) quadrants are shaded. This accounts for a total of four negative signs, and therefore a resulting positive sign. In this second type of triple move the comparison before and after is given by the relabeling $p_{ij}\mapsto q_{ij}$, which indeed maps $R$ to $R'$ and the triangle $p_{13}p_{32}p_{21}$ to the triangle $q_{13}q_{32}q_{21}$. Therefore the quiver $Q(\SC)$ and the potential $W(\SC)$ are both identical before and after this triple point move. This concludes the result.
\end{proof}


\subsubsection{A property of local bigons.}\label{sssec:local_property_bigons} An $\ell$-gon with $\ell=2$ will be referred to as a bigon $B\sse\Sigma$. Note that a bigon $B\sse\Sigma$ must be oriented, either clockwise or counter-clockwise.\footnote{If $\g_i,\g_j\in\SC$ bound an ``unoriented bigon'', then $Q(\SC)$ has two arrows from $\g_i$ to $\g_j$, or viceversa. We do not consider this a bigon. It does not yield a 2-cycle $Q(\SC)$ and (thus) it does not contribute a quadratic monomial to $W(\SC)$.} By definition, a bigon  $B\sse\Sigma$ is said to be local if $\mbox{int}(B)\cap\g_i=\emptyset$ for all $\g_i\in\SC$.

Given a local bigon $B\sse\Sigma$ as in Figure \ref{fig:LocalBigon}, we refer to the region $\rho(B)$ drawn in red, resp.~the region $\la(B)$ drawn in blue, as its right region, resp.~as its left region.

\begin{center}
	\begin{figure}[H]
		\centering
		\includegraphics[scale=0.8]{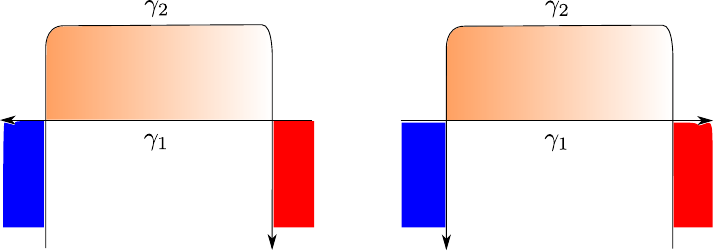}
		\caption{The regions $\rho(B)$ (red) and $\la(B)$ (blue). The bigons depicted in orange.}
		\label{fig:LocalBigon}
	\end{figure}
\end{center}

\color{black}
\begin{assumption}\label{assumption:bigons}
Let $\SC$ be a curve configuration. Every local bigon $B\sse\Sigma$ is assumed to satisfy that there is no polygon bounded by $\SC$ which contains both $\la(B)$ and $\rho(B)$.\hfill$\Box$
\end{assumption}

In other words, given a local bigon, we assume that there is no other polygon that has vertices two of its vertices be the vertices of the local bigon. \color{black} It will be proven in Section \ref{ssec:assumption_bigons}, specifically Proposition \ref{prop:local_bigons}, that this assumption is satisfied for the configurations of curves that we shall use, i.e.~for those configurations constructed in Section \ref{sec:QP_PlabicFence}. For now we work under Assumption \ref{assumption:bigons}: we suppose it holds for all local bigons discussed subsequently.
\color{black}
\begin{remark}
Proposition \ref{prop:local_bigons} shows that non-degeneracy of the quiver with potential guarantees that Assumption \ref{assumption:bigons} holds. Non-degeneracy is a robust condition: if a potential is non-degenerate, then any of its mutations and any restriction to a
full subquiver is also non-degenerate. That said, there are specific examples one can build where Assumption \ref{assumption:bigons} fails. For instance, the cylindrical closure of the 2-stranded braid word $\sigma_1^2$ (with two crossings) in a cylinder $\Sigma=S^1\times\R$ gives two curves $\g_1$ and $\g_2$ where Assumption \ref{assumption:bigons} does not hold.\hfill$\Box$
\end{remark}
\color{black}

\subsubsection{Behavior under bigon moves.} By definition, the local bigon moves are the local moves depicted in Figure \ref{fig:Moves2}. This move applies to local bigons, i.e.~bigons $B\sse\Sigma$ such that $\mbox{int}(B)\cap\g_i=\emptyset$ for all $\g_i\in\SC$. This is the reason for referring to it as a local bigon move, instead of just a bigon move.

\begin{center}
	\begin{figure}[H]
		\centering
		\includegraphics[scale=0.6]{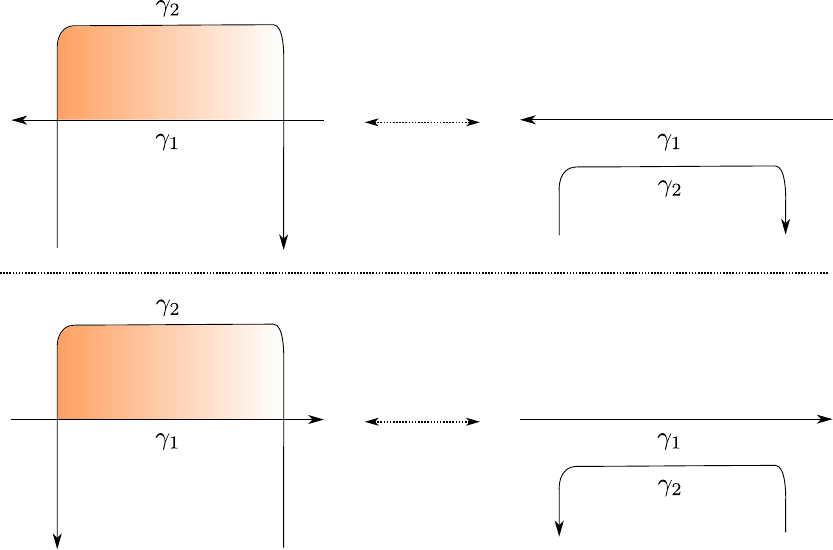}
		\caption{The two local bigon moves.}
		\label{fig:Moves2}
	\end{figure}
\end{center}

In order to understand how $(Q(\SC),W(\SC))$ changes under local bigon moves, we introduce the following:

\begin{definition}\label{def:localreduction}
Let $(Q,W)$ be a QP and $a,b$ be two arrows in $Q$ such that $ab$ is a 2-cycle and $ab$ appears as a quadratic monomial in $W$. The $ab$-reduction of $(Q,W)$, or the local reduction $(Q,W)$ at $ab$, is the QP $(Q',W')$ obtained as follows:
\begin{itemize}
    \item[(i)] The quiver $Q'$ coincides with $Q$ except that both arrows $a$ and $b$ have been erased.\\

    \item[(ii)] The potential $W'$ is constructed as follows. Suppose that there exist polynomials $U,V$, neither of them containing $a$ or $b$, such that
    $$W = (a-U)(b-V) + W',$$ where $W'$ does not contain $a$ or $b$, and the equality is up to cyclic permutation of each monomial. Then $W':=W-(a-U)(b-V)$.
\end{itemize}
If such polynomials $U,V$ do not exist, then the $ab$-reduction of $(Q,W)$ is said not to exist.\hfill$\Box$
\end{definition}

Definition \ref{def:localreduction} and the use of the word reduction for such an operation is a direct influence of \cite[Section 4]{DWZ}. Note that an oriented bigon $B\sse\Sigma$ uniquely determines its two intersection points, and it is uniquely determined by them if we know they bound a  bigon. Equivalently, these are two arrows $a,b$ in $Q(\SC)$ such that $ab$ is a 2-cycle and $ab$ is a monomial appearing in $W(\SC)$. In these cases, where $ab$ is the 2-cycle corresponding to an oriented bigon $B$, we also refer to an $ab$-reduction as a $B$-reduction.

\begin{lemma}\label{lem:bigon}
Let $(Q(\SC),W(\SC))$ be the curve QP associated to a curve configuration $\SC$, $B\sse\Sigma$ be a local bigon and $\SC'$ the configuration $\SC$ after a local bigon move at $B$. Then the $B$-reduction of $(Q(\SC),W(\SC))$ exists and it equals $(Q(\SC'),W(\SC'))$.
\end{lemma}

\begin{proof} Let us consider a bigon $B\sse\Sigma$ bounded by two curves $\g_1$ and $\g_2$. There are two cases, depending on whether the bigon $B$ is oriented clockwise or counter-clockwise. The two cases are almost identical and thus we focus on that of a clockwise oriented bigon, as depicted in Figure \ref{fig:Bigon} (left). The local quiver is drawn in Figure \ref{fig:Bigon} (right) and contains the 2-cycle $p_{12}p_{21}$. Since the bigon is oriented clockwise and none of the two quadrants of the bigon is shaded, its contribution to the potential $W(\SC)$ is the monomial $p_{12}p_{21}$.\footnote{In the case of a counter-clockwise oriented bigon, the bigon would contribute to the potential with $-p_{12}p_{21}$. There would be a minus sign because the bigon would be oriented counter-clockwise and use exactly two shaded quadrants.}

\begin{center}
	\begin{figure}[H]
		\centering
		\includegraphics[scale=0.6]{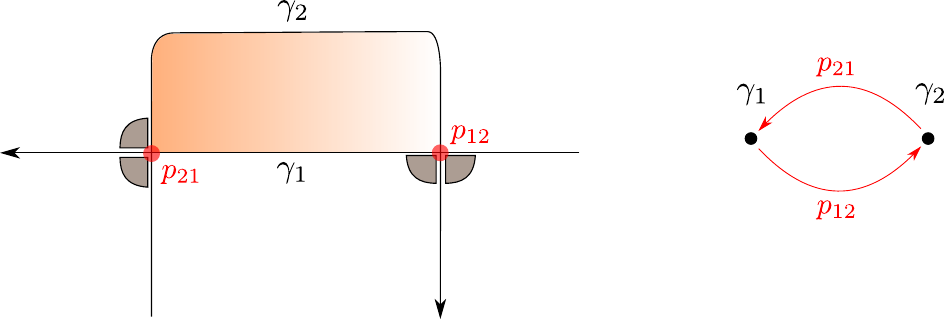}
		\caption{(Left) Local model for a bigon before bigon move. The intersection points are highlighted (red) and we have shaded the quadrants. (Right) Local quiver.}
		\label{fig:Bigon}
	\end{figure}
\end{center}

First, let us prove that the $B$-reduction of $(Q(\SC),W(\SC))$ exists.

Let $S_{21}$ be the set of polygons that contain the south-west corner of $p_{21}$, i.e. the set of polygons that come in from $\gamma_2^{in}$, turn left at $p_{21}$ and exit via $\gamma_1^{out}$.\footnote{There are no polygons using the north-west or south-east quadrants at $p_{21}$ because of orientations. Similarly, there are no polygons using the north-east and south-west quadrants at $p_{12}$.} A polygon in $S_{21}$ has the opposite orientation from the bigon $B$ and contains $p_{21}$ in its associated monomial in $W(\SC)$. Therefore, the contribution from polygons in $S_{21}$ to the potential $W(\mathcal{C})$ has the form $-Up_{21}$ for some polynomial $U$ that does not contain $p_{21}$. Similarly, let $S_{12}$ be the set of polygons that contain the south-east corner of $p_{12}$, i.e. the set of polygons that come from $\gamma_1^{in}$, turn down at $p_{12}$ and exit via $\gamma_2^{out}$. The contribution from the polygons in $S_{12}$ to the potential $W(\mathcal{C})$ has the form $-Vp_{12}$ for some polynomial $V$ that does not contain $p_{12}$.
\color{black}
By Assumption \ref{assumption:bigons}, the polynomial $U$ does not contain $p_{12}$ and $V$ does not contain $p_{21}$.

\color{black}Let us write the potential as
$$W(\cC) = p_{12}p_{21} - Up_{21} - Vp_{12} + \tilde{W}.$$
for some $\tilde{W}$. Now, monomials in $W(\SC)$ are precisely given by boundaries of embedded regions, and thus their associated cycles in the quiver are irreducible, i.e.~not the composition of two cycles. Therefore, the  construction above is such that $\tilde{W}$ does not contain $p_{12}$ nor $p_{21}$. Following the formulation in Definition \ref{def:localreduction}, we rewrite
$$W(\cC) = (p_{12}-U)(p_{21}-V) - UV + \tilde{W},$$

where $\tilde{W}$ does not contain $p_{12}$ or $p_{21}$. Therefore, we can select $W':= -UV + \tilde{W}$ in Definition \ref{def:localreduction}. Thus we conclude that $(Q(\SC),W(\cC))$ is indeed $B$-reducible.\\


Second, let us now show that the $B$-reduction of $(Q(\SC),W(\SC))$ equals $(Q(\SC'),W(\SC'))$. After the bigon move at $B$ we have the local configuration in Figure \ref{fig:Bigon_After} (left). The local quiver $Q'(\SC)$ becomes two vertices with no arrows between them, as there are no intersections in this local piece; it is depicted in Figure \ref{fig:Bigon_After} (right). For the same reason, the potential $W(\SC')$ after the bigon move has no contributions coming from this local configuration.

\begin{center}
	\begin{figure}[H]
		\centering
		\includegraphics[scale=0.6]{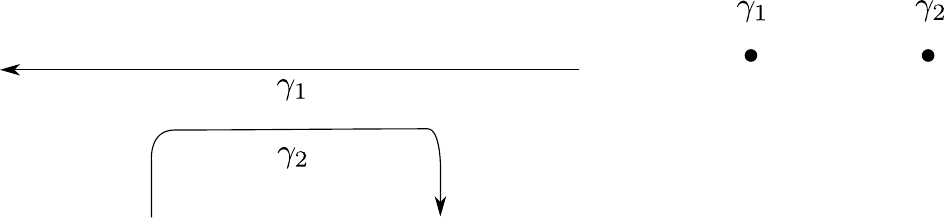}
		\caption{(Left) Local model after bigon move. (Right) Local quiver.}
		\label{fig:Bigon_After}
	\end{figure}
\end{center}

By comparing the two quivers directly, $Q(\SC')$ is obtained from $Q(\SC)$ by exactly removing the arrows $p_{12}$ and $p_{21}$. We claim that the $B$-reduction of the curve potential $W(\cC)$ equals the curve potential after the bigon removal, namely $W' = W(\cC')$, where $W'=-UV + \tilde{W}$ as above. Indeed, since the bigon region disappears, $p_{12}p_{21}$ is not in $W(\SC')$. Also, any region in $S_{21}$ itself disappears under the bigon move, and thus the polynomial term $-Up_{12}$ no longer appears in $W(\SC')$. The argument for $-Vp_{21}$ is identical. This justifies that $p_{12}p_{21}$, $-Up_{12}$ and $-Vp_{21}$ must be subtracted from $W(\SC)$ to obtain the potential $W(\SC')$. It suffices to prove that the term $UV$ must be added (with signs) to $W(\SC')$. Indeed, when we remove the local bigon, a polygon in $S_{21}$ and a polygon in $S_{12}$ will be connected along the strip between $\gamma_1$ and $\gamma_2$, creating new polygons that contribute $-UV$. This leads to the addition of $-UV$ in the potential $W(\SC')$. In conclusion, the curve QP $(Q(\SC'),W(\SC'))$ is the $B$-reduction of $(Q(\SC),W(\SC))$, as required.\\

\color{black}
Finally, the argument in the counterclockwise case is analogous. The only change comes from the signs of the polygons at the corners. In the same notation as above, the potential then reads
$$W(\SC) = p_{12}p_{21} + Up_{21} + Vp_{12} + \bar{W}$$
instead of $W(\cC) = p_{12}p_{21} - Up_{21} - Vp_{12} + \tilde{W}$. The rest of the argument remains same because we can absorb the signs into the $U$ and $V$ polynomials, by relabelling them $\wt U:=-U$ and $\wt V:=-V$.
\color{black}
\end{proof}


\color{black} 

\subsubsection{Bigon removal and the reduced part of $(Q(\SC),W(\SC))$.}\label{sssec:HassScott} We introduce the following definition:

\begin{definition}
A configuration of curves $\SC$ such that there is no bigon bounded by $\SC$ is said to be reduced. A configuration which is not reduced is said to be non-reduced.\hfill$\Box$
\end{definition}

\noindent Consider a non-reduced configuration $\SC_0$ with a collection of bigons $\{B_1,\ldots,B_m\}$. Note that these bigons $B_i$ might not be local, i.e.~the interior of each $B_i$ might intersect curves in $\SC_0$ in a non-empty set. The first step is to modify $\SC_0$ into a reduced configuration. For that, we must systematically remove bigons.

The argument for the removal of (not necessarily local) bigons dates back to E.~Steinitz in 1916, cf.~\cite[Sections 67/68]{BigonRemoval_76}. Since then, this argument has been reproduced in different parts of the literature. Specific instances are \cite[Lemma 2.2]{BigonRemoval_22}, Lemma 2 in \cite[Section 13.1]{BigonRemoval_67} and discussion thereafter, proofs of Lemmas 1.4 and 1.6 in \cite[Section 1]{BigonRemoval_94} and discussion preceding them, the first two lemmas of \cite[Section 2]{BigonRemoval_96}, the proof of \cite[Lemma 3.2]{BigonRemoval_09} and discussion thereafter, or the discussions in \cite[Section 1]{BigonRemoval_17} and \cite[Section 2]{BigonRemoval19} or references therein. Bigons are referred to as {\it 2-gons} in \cite{BigonRemoval_94}, {\it lentilles} in \cite{BigonRemoval_96}, {\it lenses} in \cite{BigonRemoval_67} and {\it spindel} in \cite[Section 67]{BigonRemoval_76}.

A piece of notation: by definition, a bigon $B\sse\Sigma$ of $\SC$ is said to be minimal if it does not contain another bigon of $\SC$, i.e.~there is no bigon $B'\sse\Sigma$ with $B'\sse B$ (and $B\neq B'$). Minimal bigons are called {\it indecomposable lenses} in \cite[Section 3.1]{BigonRemoval_67}, {\it irreduzible spindel} in \cite[Section 67]{BigonRemoval_76}, and {\it minimal} or {\it innermost} in \cite[Section 2.1]{BigonRemoval_22} and \cite[Section 1.1]{BigonRemoval_17}. For specificity, we choose the statement in \cite[Lemma 2.2]{BigonRemoval_22}:\\

\begin{lemma}[\cite{BigonRemoval_22}]\label{lem:bigon_removal} For any minimal bigon $B\sse\Sigma$ of $\SC$, there exists a sequence of triple point moves from $\SC$ to a new configuration $\SC'$ such that the image of the bigon $B\sse\Sigma$ in $\SC'$ is a local bigon, i.e.~any pieces of curves inside of $B$ can be removed by triple point moves.
\end{lemma}

\begin{center}
	\begin{figure}[H]
		\centering
		\includegraphics[scale=0.6]{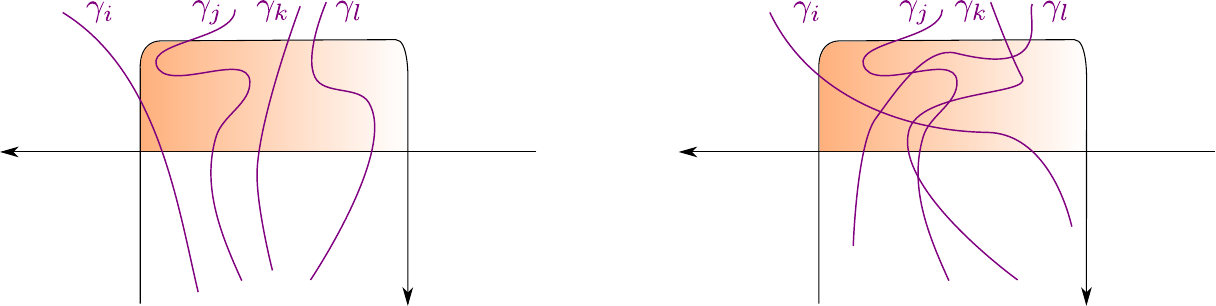}
		\caption{(Left) A bigon in the first case in the proof of Lemma \ref{lem:bigon_removal}. (Right) A non-local minimal bigon.}
		\label{fig:BigonRemoval}
	\end{figure}
\end{center}

The statement \cite[Lemma 2.2]{BigonRemoval_22} is referred to as {\it Steinitz bigon removal algorithm} in \cite{BigonRemoval_22}. At the request of the referee, we include a sketch of the proof:

\begin{proof}[Proof of Lemma \ref{lem:bigon_removal}]
Suppose first that $B$ contains no intersections of any $\gamma_i$ and $\gamma_j$ inside, i.e. any curve $\gamma_i$ intersecting the interior $\mbox{int}(B)$ does so in a connected interval  embedded in $B$ which does not intersect any other curves in $\mbox{int}(B)$. By minimality, this interval has endpoints on the two different sides of the bigon. Figure \ref{fig:BigonRemoval} (left) gives an example of such a case. In this case, it suffices to apply exactly one triple point move per intersecting curve to make the bigon local. Indeed, choose one of the two vertices of the bigon and perform a triple point move with the unique interval inside of $B$ which is a side of the unique local triangle in $B$ containing that vertex. This first step removes one such intersecting interval. Iterating, we can use triple point moves to remove each intersecting strand until the bigon is local.

Consider now the general situation, as depicted in Figure \ref{fig:BigonRemoval} (right), where there might be crossings of $\gamma_i,\gamma_j\in\SC$ inside of $B$. We reduce to the previous case as follows. First, one shows that any non-local minimal bigon $B$ must have a local triangle incident to one of its sides. This is \cite[Lemma 2.1]{BigonRemoval_22}, or \cite[Lemma 13.1.2]{BigonRemoval_67}, or also \cite[Lemma 1.4]{BigonRemoval_94}, for instance. Then one implements a triple point move to remove that triangle and the total number of crossings (of $\gamma_i,\gamma_j\in\SC$) inside of $B$ strictly decreases by one. Iterating this argument, the number of crossings is reduced to zero, which is the case discussed above.
\end{proof}

Lemma \ref{lem:bigon_removal} then implies:

\begin{thm}\label{thm:HassScott} Let $\SC_0$ be a non-reduced configuration with a collection of bigons $\{B_1,\ldots,B_m\}$. Then, for any $i\in[m]$ with $B_i$ minimal, there exists a sequence of triple point moves and one local bigon move on $\SC_0$ that yields a new configuration $\SC_1$ such that the collection of bigons of $\SC_1$ is $\{B_1,\ldots,B_m\}\setminus\{B_i\}$.
\end{thm}

\begin{proof}
Consider an $i\in [m]$ with $B_i$ minimal. By Lemma \ref{lem:bigon_removal}, there exists a sequence of triple point moves that makes $B_i$ a local bigon. Note that such sequence of triple point moves does not create or remove any additional bigons. Once the bigon $B_i$ has been made local, a unique local bigon move suffices to obtained the required configuration $\SC_1$.
\end{proof}

Note that the proof of Theorem \ref{thm:HassScott} is local on a given bigon $B_i$, in that it only modifies the configuration $\SC_0$ in an arbitrarily small neighborhood of $B_i$. It is worth remarking that Theorem \ref{thm:HassScott} is able to remove a bigon without introducing further (unnecessary) intersections: just with triple points moves any bigon (possibly non-local) becomes a local bigon.

In summary, given a non-reduced configuration  $\SC\sse\Sigma$, Theorem \ref{thm:HassScott} implies that there exists a sequence of moves, both possibly containing and intertwining triple point moves and local bigon moves, that when applied to $\SC$ yields a reduced configuration. Indeed, there always exists a minimal bigon in a non-reduced configuration. After applying Theorem \ref{thm:HassScott}, this bigon can be removed, yielding a configuration with strictly less bigons. Iterating this procedure leads to a reduced configuration.

\color{black}


\subsection{Curve QPs under $\g$-exchange} We introduce the following operation on a configuration $\SC$.


\subsubsection{Definition of $\g$-exchange of $\SC$}\label{sssec:definition_exchange} Let $\gamma\in\SC$ be a curve. By definition, the configuration of curves $\mu_{\g}(\SC)$ is given by the configuration of curves
$$\mu_{\g}(\SC):=\{\mu_\g(\g_1),\ldots,\mu_\g(\g_b)\},$$
where the curves $\mu_\g(\g_i)$ are obtained as follows, $i\in[b]$.

Consider a neighborhood $U\sse\Sigma$ of $\gamma$ such that any curves in $\SC\setminus\{\g\}$ intersect $U$ as depicted in Figure \ref{fig:CExchange_Before}. That is, up to planar isotopy,  each curve $\g_i\in\SC\setminus\{\g\}$ intersects $U$ at a collection of intervals $\{I_i^p\}$, $p\in[q_i]$ for some $q_i\in\N$, where each such interval only intersects $\g$ once and intervals do not intersect each other; note that this collection $\{I_i^p\}$ might be empty for some curves $\g_i\in\SC$. The neighborhood $U$ must be a cylinder, since it is the neighborhood of an embedded connected curve in an oriented $\Sigma$.

The curves $\mu_\g(\g_i)$ in $\mu_\g(\SC)$ are constructed as follows. We apply a positive Dehn twist of the cylinder $U$ along the simple embedded curve $\g$ to all the segments in $U$ that intersect $\g$ positively, depicted in blue in Figure \ref{fig:CExchange_Before}, and the identity map to all the segments in $U$ that intersect $\g$ negatively, depicted in green in Figure \ref{fig:CExchange_Before}. Note that a Dehn twist often refers to a mapping class, i.e.~an element of $\pi_0(\mbox{Diff}^c(U))$. In this construction we explicit mean a representative of that mapping class: at this stage we make one such choice of representative and continue.

Since both a Dehn twist and the identity are (represented by) compactly supported diffeomorphisms, each resulting segment $f(I_i^p)\sse U$, $f$ a Dehn twist or the identity, can be glued to the corresponding curve $\g_i\cap(\Sigma\setminus U)$. By definition, the result of applying such operation to $\g_i$ is the curve $\mu_\g(\g_i)$ if $\g_i\neq\g$. We define $\mu_\g(\g):=-\g$. See Figure \ref{fig:CExchange_After} for the result of applying $\mu_\g$ to Figure \ref{fig:CExchange_Before}.

\begin{remark}\label{rmk:exchange_choices}
The choices in the above construction will not affect any aspects of our results. For instance, the choice of neighborhood $U$ with the required properties only modifies the resulting configuration $\mu_\g(\SC)$ by a global isotopy of $\Sigma$. Similarly, the choice of representative of the Dehn twist class in $\pi_0(\mbox{Diff}^c(U))$ results in the same configuration $\mu_\g(\SC)$ up to global isotopy.\hfill$\Box$
\end{remark}

By Remark \ref{rmk:exchange_choices} and the fact that we only consider configurations $\SC$ up to a global diffeomorphism of $\Sigma$, the resulting configuration $\mu_{\g}(\SC)$ is well-defined. We therefore define:

\begin{definition}\label{def:exchange}
The configuration of curves
$\mu_{\g}(\SC)$ is said to be the $\g$-exchange of $\SC$.\hfill$\Box$
\end{definition}
Note that applying a $\g$-exchange twice along the same curve, first to $\g$ and then to $-\g$, leads to the same configuration; same up to an overall compactly supported diffeomorphism of $\Sigma$. Indeed, consecutively applying a $\g$-exchange at the same vertex yields the configuration $\mu_{-\g}\mu_\g(\SC)=\tau_\g(\SC)$, given by applying a (representative of a) Dehn twist $\tau_\g\in\mbox{Diff}^c(\Sigma)$ to {\it all} the curves in $\SC$. In that sense, a $\g$-exchange is an involution. From the perspective of the quiver with potential $(Q(\SC),W(\SC))$, applying a Dehn twist $\tau_\g$ preserves the quiver with potential: the quivers with potential of $\SC$ and $\tau_\g(\SC)$ are identical, $(Q(\SC),W(\SC))=(Q(\tau_\g(\SC)),W(\tau_\g(\SC)))$, not just right-equivalent.

Two comments on Definition \ref{def:exchange}:

\begin{itemize}
    \item[(i)] Since we only apply a Dehn twist to {\it some} collection of $\g_i\in\SC$, it is possible (and it occurs) that the smooth representatives we have constructed for the curves in $\mu_\g(\SC)$ bound bigons. That is, $\mu_\g(\SC)$ might be non-reduced even if $\SC$ is reduced.\\

    \item[(ii)] If a configuration $\SC$ of simple embedded curves is non-reduced and $\g,\g'$ are curves bounding a bigon, then $\mu_\g(\SC)$ contains the immersed curve $\mu_\g(\g')$. 
\end{itemize}

These two comments highlight an important principle: if we want to perform a sequence of $\g$-exchanges, we must guarantee that $\mu_\g(\SC)$ is reduced so that we can iterate the procedure, as $\g$-exchanges are not defined for immersed curves $\g$. The saving grace is that we might be able to perform triple moves and bigon moves to $\mu_\g(\SC)$ so that a configuration becomes reduced, following Theorem \ref{thm:HassScott}. Our focus thus now shifts towards understanding how polygons bounded by $\SC$ change under a $\g$-exchange. In other words, understanding how the right-equivalence class of the curve QP $(Q(\SC),W(\SC))$ changes under a $\g$-exchange.

\begin{remark}[Algebraic topological $\g$-exchange]\label{rmk:topological_exchange} There is a variation on Definition \ref{def:exchange} which is also natural. Namely, we define $\mu_\g(\g_i)=\tau_\gamma(\g_i)$ if the algebraic intersection number $\langle\gamma_i,\gamma\rangle$ is strictly positive (or $\g_i=\g$) and $\mu_\g(\g_i)=\g_i$ otherwise. Here $\tau_\gamma$ denotes a positive Dehn twist along $\g$ and we use the skew-symmetric intersection pairing in $H_1(\Sigma,\Z)$. For the configurations $\SC$ that we will study, which are always reduced, have algebraic intersection numbers equal to geometric intersection numbers and all $\g_i\in\SC$ be primitive, this alternative definition coincides with Definition \ref{def:exchange}. In general, it does not coincide: consider a $\gamma$-exchange along a null-homologous curve $\g$ which intersects another $\g_i$ exactly at a positive and negative intersection points, bounding a bigon. Definition \ref{def:exchange} would produce an immersed representative of the homology class of $\g_i$, and the definition in this remark would keep $\g_i$ identically the same.\hfill$\Box$
\end{remark}

\begin{remark} Definition \ref{def:exchange} or its homological variation (in Remark \ref{rmk:topological_exchange}) have appeared in previous works. Specifically, in \cite[Section 3]{KingQiu_TwistedSurfaces} and \cite[Definition 2.3]{STW}. In the context of the former reference, it is introduced as an analogue
of simple tilting on hearts for $t$-structures in triangulated categories, following the homological construction of forward and backward tilts in \cite{HRS96}. The context for the later reference is almost identical to ours and Definition \ref{def:exchange} is essentially \cite[Definition 2.3]{STW}.\hfill$\Box$
\end{remark}


\subsubsection{QP-mutation according to Derksen-Weyman-Zelevinsky}\label{sssec:DWZ} In order to understand how the right-equivalence class of the curve QP $(Q(\SC),W(\SC))$ changes under a $\g$-exchange, we recall the notion of QP-mutation introduced in \cite[Section 5]{DWZ}. Let $Q$ be a quiver with set of vertices $Q_0$, set of arrows $Q_1$ and let us denote $h(a)\in Q_0$, resp.~$t(a)\in Q_0$, the head, resp.~the tail, of an arrow $a\in Q_1$.

Let $(Q,W)$ be a quiver with potential. Consider a vertex $v_k\in Q_0$ and assume that $v_k$ does not belong to an oriented 2-cycle. Suppose also that no monomial in $W$ starts or ends with $v_k$, i.e.~ we write the monomials so that $v_k$ appears neither at the start nor the end of a monomial in $W$, which can always be achieved after cyclically reordering. The former hypothesis will always be satisfied in the applications of this manuscript and, as just said, the latter hypothesis can be always guaranteed after cyclically changing some monomials in the potential.

\begin{definition}[\cite{DWZ}] The non-reduced QP-mutation of $(Q,W)$ at $v_k$, satisfying the above hypotheses, is the QP $(\mu_k(Q),\mu_k(W))$ defined as follows.

The quiver $\mu_k(Q)$ has the same set of vertices $\mu_k(Q)_0=Q_0$ as $Q$ and its set of arrows $\mu_k(Q)_1$ is obtained from $Q_1$ according to the following procedure:

\begin{enumerate}
    \item All arrows in $Q_1$ not incident to $v_k$ also belong to $\mu_k(Q)_1$.\\

    \item For each pair $a,b\in Q_1$ of incoming arrow $a$ and outgoing arrow $b$ at $v_k$, create the composite
arrow $[ba]\in\mu_k(Q)_1$.\\

    \item Replace each incoming arrow $a\in Q_1$ (resp.~each outgoing arrow $b\in Q_1$) at $v_k$ by a
corresponding arrow $a^\ast\in\mu_k(Q)_1$ (resp.~$b^\ast\in\mu_k(Q)_1$) now oriented in the opposite way.
\end{enumerate}

The potential $\mu_k(W)$ is defined as
$$\mu_k(W):=[W]+\Delta_k,\quad \Delta_k:=\sum_{a,b\in Q_1,t(a)=h(b)=k}[ba]b^\ast a^\ast$$
where $[W]$ is obtained by substituting the composite arrow $[ab]$ for each factor $ab$ with $t(a) =
h(b) = k$ of any cyclic path occurring in the expansion of $W$ that contains $ab$.\hfill$\Box$
\end{definition}

By definition, the QP $(\mu_k(Q),\mu_k(W))$ is said to be obtained by non-reduced QP-mutation of $(Q,W)$ at the vertex $v_k$. It is shown in \cite[Theorem 5.2]{DWZ} that the right-equivalence class of $(\mu_k(Q),\mu_k(W))$ depends only on the right-equivalence class of $(Q,W)$.

\begin{remark}
The mutation $\mu_k(Q)$ of a quiver $Q$, without the potentials $W$ and $\mu_k(W)$, was previously defined in \cite[Definition 4.2]{FominZelevinsky_ClusterI}.\hfill$\Box$
\end{remark}

The result of a QP-mutation of $(Q,W)$ with no 2-cycles might result in a QP $(\mu_k(Q),\mu_k(W))$ with 2-cycles. In \cite{DWZ} the notions of {\it reduced} and {\it trivial} QPs are introduced, as follows. A QP $(Q,W)$ is said to be reduced if the degree-2 homogeneous part $W^{(2)}$ of $W$ is trivial, i.e.~$W^{(2)}=0$. That is, $(Q,W)$ is reduced if $W$ contains no quadratic monomial terms, i.e.~no terms of the form $ab$, $a,b\in Q_1$. A QP $(Q,W)$ is said to be trivial if $W$ is entirely quadratic, i.e.~ $W\in \C\langle Q\rangle^{(2)}$ belongs to the degree-2 homogenous part of the path algebra $\C\langle Q\rangle$ and the Jacobian algebra of $(Q,W)$ is isomorphic to $\C$.

\begin{remark} By \cite[Prop.~ 4.4.]{DWZ}, there is a more pragmatic criterion to detect triviality: $(Q,P)$ is trivial if and only if the set of arrows
$Q_1$ consists of $2N$ distinct arrows $a_1, b_1,\ldots,a_N , b_N$ such that each $a_kb_k$ is a cyclic
2-path, and there is a change of arrows $\varphi$ such that $\varphi(W)$ is
cyclically equivalent to $a_1b_1 +\ldots+ a_N b_N$.\hfill$\Box$
\end{remark}

\begin{remark}\label{rmk:reduced_2cycles} Note that the quiver $Q$ is not enough to determine the reduced and trivial parts of a QP $(Q,W)$. For instance, the quiver $Q$ consisting of two vertices $Q_0=\{v_1,v_2\}$ and two arrows $Q_1=\{a,b\}$ with $h(a)=t(b)=v_1$, $t(a)=h(b)=v_2$ is trivial if the potential is chosen to be $W=ab$, and it is reduced if $W=0$ is chosen to vanish.\hfill$\Box$
\end{remark}

The following structural result is established in \cite[Theorem 4.6]{DWZ}:

\begin{thm}[\cite{DWZ}]\label{thm:DWZ_splitting} For every $QP$ $(Q, W)$ with trivial arrow span
$Q_{triv}$ and reduced arrow span $Q_{red}$, there exist a trivial QP $(Q_{triv},W_{triv})$ and
a reduced QP $(Q_{red}, W_{red})$ such that $(Q,W)$ is right-equivalent to the direct sum $(Q_{triv},W_{triv})\oplus(Q_{red}, W_{red})$. Also, the right-equivalence class of each of the
QPs $(Q_{triv},W_{triv})$ and $(Q_{red}, W_{red})$ is determined by the right-equivalence class of
$(Q,W)$.
\end{thm}

This allows us to finally define QP-mutation:

\begin{definition}[\cite{DWZ}]\label{def:QP_mutation} The QP-mutation of $(Q,W)$ at $v_k$ is the QP $(\mu_k(Q)_{red},\mu_k(W)_{red})$ given by the reduced part of the non-reduced mutation $(\mu_k(Q),\mu_k(W))$.\hfill$\Box$
\end{definition}

This is an involutive operation if performed at the same vertex $v_k$ consecutively, i.e.~performing QP-mutation of $(Q,W)$ at $v_k$ twice consecutively leads to $(Q,W)$ again. Note also that Definition \ref{def:localreduction} of local reduction is one step towards extracting the reduced part of a QP.


\subsubsection{Reduced part for curve QPs}\label{sssec:reducedpart_curveQP} Let $\SC$ be a curve configuration and $(Q(\SC),W(\SC))$ its associated curve QP. Theorem \ref{thm:HassScott} above geometrically explains Theorem \ref{thm:DWZ_splitting} in the case of curve QPs. Indeed, consider a non-reduced configuration $\SC_0$ with a non-empty collection of bigons $\{B_1,\ldots,B_m\}$. Then its associated curve QP $(Q(\SC),W(\SC))$ is not reduced: by definition, it contains one 2-cycle in $Q(\SC)$ for each bigon, and such 2-cycles each appear as quadratic monomial in $W(\SC)$.

\begin{definition}\label{def:reduced_configuration} Let $\SC$ be a configuration with a collection of bigons $\{B_1,\ldots,B_m\}$. Any reduced curve configuration $\SC_{red}$ obtained by iteratively applying Theorem \ref{thm:HassScott} $m$ times to $\SC$, where at each time exactly one bigon is eliminated, is said to be a reduction of $\SC$.\hfill$\Box$
\end{definition}

Namely, Theorem \ref{thm:HassScott} allows us to remove one bigon at a time, by applying a sequence of triple point moves and then a local bigon move. By Proposition \ref{lem:triplepoint}, the sequence of triple point moves does not change the right-equivalence class of $(Q(\SC),W(\SC))$. Each time that we apply Theorem \ref{thm:HassScott} we need exactly one local bigon move. By Lemma \ref{lem:bigon}, the QP $(Q(\SC),W(\SC))$ then undergoes a $B$-reduction at a bigon $B$. By iteratively applying Theorem \ref{thm:HassScott}, we obtain the following:

\begin{lemma}\label{lem:reduced_parts} Let $\SC$ be a configuration, $(Q(\SC),W(\SC))$ its associated curve QP, $(Q(\SC)_{red},W(\SC)_{red}))$ its reduced QP part, and $\SC_{red}$ a reduction of the configuration $\SC$. Then

\begin{itemize}
    \item[(i)] $(Q(\SC_{red}),W(\SC_{red}))=(Q(\SC)_{red},W(\SC)_{red}))$, up to right equivalence.\\

    \item[(ii)] If $(Q(\SC),W(\SC))$ is non-degenerate, then $Q(\SC_{red})$ has no 2-cycles.
\end{itemize}

\end{lemma}

The proof of Lemma \ref{lem:reduced_parts} uses the notion of a non-degenerate QP, discussed in Section \ref{ssec:nondeg_def} below. This proof is thus postponed until Section \ref{ssec:proof_reduction}. Note that $Q(\SC_{red})$ might a priori have 2-cycles, even if there are no bigons bounded by $\SC_{red}$, cf.~Remark \ref{rmk:reduced_2cycles}. The non-degeneracy of a quiver with potential precisely rules out this type of situation, where a 2-cycle in $Q$ is not kept track by the potential $W$.


\subsubsection{Curve QP under $\g$-exchanges change via QP-mutations} We conclude this section by relating the $\g$-exchanges in Definition \ref{def:exchange} to the QP-mutation in Definition \ref{def:QP_mutation}.

\begin{prop}\label{prop:quivermutation}
Let $(Q(\SC),W(\SC))$ be a curve QP associated to $\SC$ and $\g\in \SC$. Then, possibly after applying a sequence of triple point moves and bigon moves to $\mu_\g(\SC)$, the QP $(Q(\mu_\g(\SC)),W(\mu_\g(\SC)))$ is right-equivalent to the QP-mutation of the QP $(Q(\SC),W(\SC))$ at the vertex $\gamma$.
\end{prop}

\begin{proof}
Given Definition \ref{def:QP_mutation}, there are three pieces to justify:
\begin{enumerate}
    \item The change in the quiver $Q(\SC)$ under $\g$-exchange.
    
    \item The change in the potential $W(\SC)$ under $\g$-exchange.

    \item The reduced part of the non-reduced QP-mutation is indeed $(Q(\mu_\g(\SC)_{red}),W(\mu_\g(\SC)_{red}))$.
\end{enumerate}
Parts (1) and (2) will be argued directly, as they indeed need a new computation. In fact, we will show that $(Q(\mu_\g(\SC)),W(\mu_\g(\SC)))$ is the non-reduced QP-mutation $(\mu_\g(Q(\SC)),\mu_\g(W(\SC)))$ of $(Q(\SC),W(\SC))$ at the vertex associated to $\g$. Part (3) follows directly from Lemma \ref{lem:reduced_parts} applied to the configuration $\mu_\g(\SC)$, once $(Q(\mu_\g(\SC)),W(\mu_\g(\SC)))=(\mu_\g(Q(\SC)),\mu_\g(W(\SC)))$ is proven.\\

First, we focus on the case of $\g$ and just two intersecting intervals $\tau_+$ and $\tau_-$, which intersect $\g$ in two points $p_+$ and $p_-$ with opposite signs. The core computations in the general case essentially reduce to this situation. We have depicted this configuration in Figure \ref{fig:CExchange_TwoBefore} (left), where the corresponding local quiver is drawn beneath the configuration. The local quiver is the linear $A_3$-quiver, with a unique arrow $p_+$ from (the vertex associated to) $\tau_+$ to $\g$, and a unique arrow $p_-$ from $\g$ to $\tau_-$. In Figure \ref{fig:CExchange_TwoBefore} (and the upcoming figures in this proof) we always identify the right hand side of the figure with the left hand side via the identity map: these are all configurations drawn in the resulting cylinder; indeed, $\g$ is a circle and it is depicted as a flat horizontal segment with its right endpoint being identified with its left endpoint.

\begin{center}
	\begin{figure}[H]
		\centering
		\includegraphics[scale=0.65]{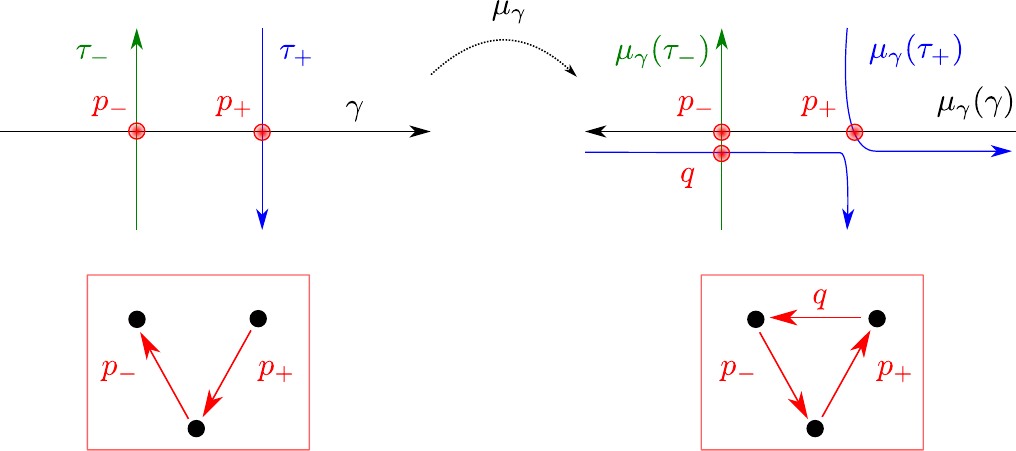}
		\caption{The effect of $\g$-exchange in the case of exactly one positive and one negative intersections. The configuration before the $\g$-exchange is depicted on the left, and the configuration after the $\g$-exchange is depicted on the right. The local quivers recording the intersection patterns are drawn under the configurations.}
		\label{fig:CExchange_TwoBefore}
	\end{figure}
\end{center}

After performing a $\g$-exchange, according to Definition \ref{def:exchange}, the resulting configuration is as depicted in Figure \ref{fig:CExchange_TwoBefore} (right). The local quiver, which is drawn right beneath the configuration, is the quiver obtained by quiver mutation at (the vertex corresponding to $\g$), according to Definition \ref{def:QP_mutation}. Indeed, the previous two intersection points $p_\pm$ persist in the configuration but since the orientation of $\g$ is opposite, their associated arrows go from $\tau_-$ to $\g$, for $p_-$, and from $\g$ to $\tau_+$, for $p_+$. As illustrated in Figure \ref{fig:CExchange_TwoBefore}, a third intersection point $q$ is created after the $\g$-exchange: it is an intersection between $\tau_-$ and $\tau_+$. This yields a new arrow $q$ in the quiver which is precisely the composite arrow $q=[p_-p_+]$. In conclusion, for these local configurations, we have verified that $Q(\mu_g(\SC))$ equals $\mu_g(Q(\SC))$.

Let us now study the change of the potential $W(\SC)$ under $\g$-exchange, also in these particular configurations just with $\g$, $\tau_+$ and $\tau_-$. For that we need to understand how polygons bounded by $\SC$ change under $\g$-exchange. The key image is Figure \ref{fig:CExchange_TwoRegions}, that we now explain.

\begin{center}
	\begin{figure}[H]
		\centering
		\includegraphics[scale=0.6]{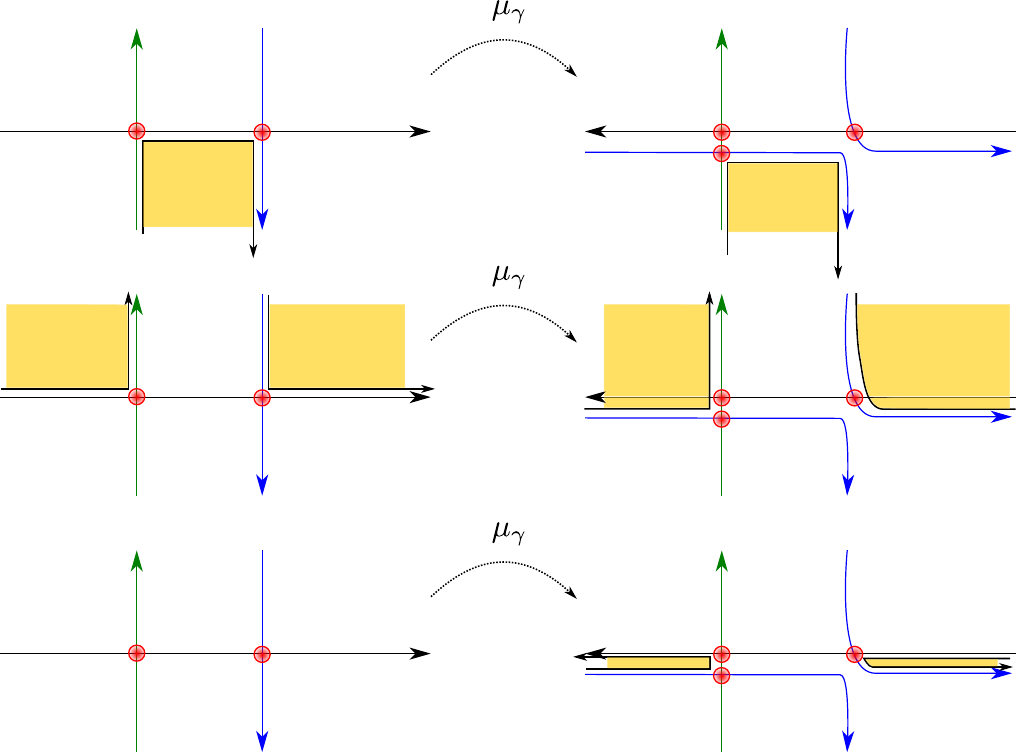}
		\caption{(Left) Potential pieces of polygons in $\SC$ that have $p_-$ and $p_+$ as vertices: the regions are highlighted in yellow. The diagram in the third row of this left column represents an empty region. (Right) Potential pieces of polygons in $\SC$ that have $p_-,p_+$ and $q$ as vertices, now after a $\g$-exchange.}\label{fig:CExchange_TwoRegions}
	\end{figure}
\end{center}

Before the $\g$-exchange, the potential $W(\SC)$ records polygons in $\SC$. Consider the intersection of any such polygon recorded by $W(\SC)$ with the region $U$ and assume again that we obtain the configuration with just $\g$, $\tau_+$ and $\tau_-$, intersecting as above. It suffices to focus on polygons that use $p_-$ and $p_+$. In this configuration, there are exactly two, drawn in the first and second rows of Figure \ref{fig:CExchange_TwoRegions} (left). The potential $W(\SC)$ records such polygons with monomials (on the arrows of $\SC$) that contain $p_-p_+$.

After the $\g$-exchange, as drawn in Figure \ref{fig:CExchange_TwoRegions} (right), those polygons still exist but now must use the new intersection point $q$.\footnote{From the perspective of $H_1(U;\Z)$, this intersection point appears because of the Picard-Lefschetz formula in its most elementary setting: explaining how homology classes  change under Dehn twists.} Indeed, the boundary conditions for the polygons are given (entering in $\tau_\pm$ and exiting in $\tau_\mp$) and there is a unique manner in which those region can exist in the configurations of Figure \ref{fig:CExchange_TwoRegions} (right). This is precisely the term $[W(\SC)]$ in Definition \ref{def:QP_mutation} of the mutation of the potential: each time we see $p_-p_+$ in a monomial in $W(\SC)$ we must substitute it by the composite arrow $q=[p_-p_+]$. There is an additional contribution to $W(\mu_\g(\SC))$: there is a triangle, as drawn in the third row of Figure \ref{fig:CExchange_TwoRegions}. Since its vertices are $q,p_+$ and $p_-$, it contributes the monomial $qp_+p_-$ to $W(\mu_\g(\SC))$, which is precisely the $\Delta_\g$ term in Definition \ref{def:QP_mutation}. Therefore, we have
$$W(\mu_\g(\SC))=[W(\SC)]+qp_+p_-=[W(\SC)]+\Delta_\g=\mu_\g(W(\SC)).$$

Thus far we have argued  $(Q(\mu_\g(\SC)),W(\mu_\g(\SC)))=(\mu_\g(Q(\SC)),\mu_\g(W(\SC)))$ for the configuration with $\g$, $\tau_+$ and $\tau_-$ as in Figure \ref{fig:CExchange_TwoBefore}. The general case is concluded as follows.

Consider a neighborhood $U$ of $\g$ such that\footnote{Such a neighborhood always exists because $\SC$ has finitely many curves and they all intersect transversely.} the configuration $\SC$ intersected with $U$ is as depicted in Figure \ref{fig:CExchange_Before}. That is, only curves $\g_i\in\SC$ that intersect $\g$ do intersect $U$ and the intersections are (up to planar isotopy) straight vertical segments that either point upwards or downwards and only intersect $\g$ once (and these segments do not intersect each other in $U$).

\begin{center}
	\begin{figure}[H]
		\centering
		\includegraphics[scale=0.6]{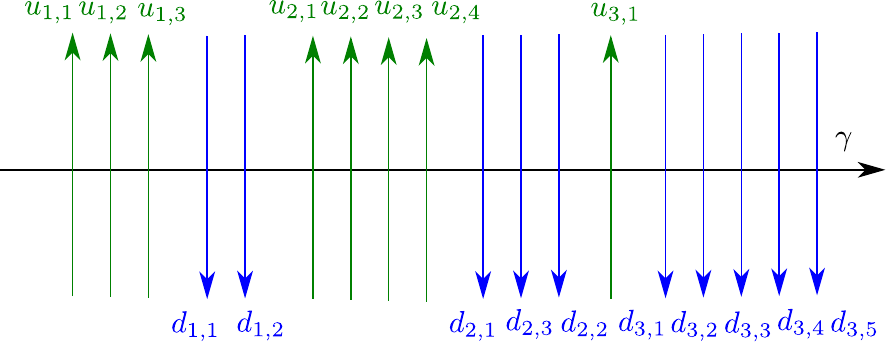}
		\caption{The general configuration $\SC$ in a neighborhood of $\g$ before a $\g$-exchange.}\label{fig:CExchange_Before}
	\end{figure}
\end{center}

Choose an arbitrary but fixed point $g\in\g$. As we scan the arrows starting at $g$ and move in the direction given by the orientation of $\g$, there will be:

\begin{itemize}
    \item[-] A number of blocks $U_1,\ldots,U_v$, $v\in\N$, of arrows pointing upward, each with $u_i:=|U_i|$ arrows, $i\in[v]$. The curves in $U_i$ will be labeled by $u_{i,j}$, $i\in[v]$, $j\in[u_i]$, where the index $j$ is ordered left to right as we transverse $\g$ according to its orientation.\\

    \item[-] A number of blocks $D_1,\ldots,D_e$, $e\in\N$, of arrows pointing upward, each with cardinality $d_i:=|D_i|$, $i\in[e]$. The curves in $D_i$ will be labeled by $d_{i,j}$, $i\in[e]$, $j\in[d_i]$, where the index $j$ is ordered left to right as we transverse $\g$ according to its orientation.
\end{itemize}

Figure \ref{fig:CExchange_After} illustrates this notation in a specific example. Now, for every pair of curves $u_{i,j}$ and $d_{l,k}$, the configuration of three curves $\gamma$, $u_{i,j}$ and $d_{l,k}$ is precisely the local configuration we studied above, as in Figure \ref{fig:CExchange_TwoBefore} (left). Since a $\g$-exchange fixes all the $u_{i,j}$ and applies a Dehn twist to all the $d_{l,k}$, the behaviour of {\it any} such pair under $\g$-exchange is identical and it coincides with the one studied above, see Figure \ref{fig:CExchange_TwoBefore} (right). As emphasized in Subsection \ref{sssec:definition_exchange}, a choice of representative for the (mapping class of the) Dehn twist must be made: we specifically choose a compactly supported diffeomorphism $f\in \mbox{Diff}^c(U)$ which acts on the set of curves $d_{l,k}$ as drawn\footnote{We drew the blue curves $d_{l,k}$ as piecewise smooth curves for convenience: the curves are smooth, and we smooth this PL depiction by an arbitrarily smoothing of the corners, so no further intersections are created. } in Figure \ref{fig:CExchange_After}. Namely, each curve $f(d_{l,k})$ intersects each $u_{i,j}$ and $\gamma$ exactly once.\footnote{Such a choice exists: the standard model for the Dehn twist defined via the geodesic flow in the punctured disk bundle (extended by the antipodal map to the zero section) has this property.}

\begin{center}
	\begin{figure}[H]
		\centering
		\includegraphics[scale=0.6]{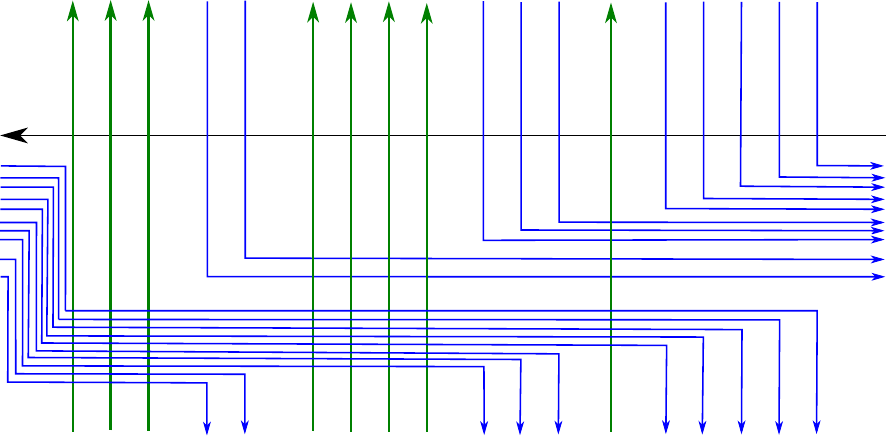}
		\caption{The general configuration $\SC$ in a neighborhood of $\g$ after a $\g$-exchange.}\label{fig:CExchange_After}
	\end{figure}
\end{center}

Since the only intersections before the $\g$-exchange are between $u_{i,j}$ and $\g$, or $d_{l,k}$ and $\g$, the new intersections afterwards will be precisely as recorded by the local configuration in Figure \ref{fig:CExchange_TwoBefore} (right). Note that this applies to every pair of $u_{i,j}$, $d_{l,k}$ and hence for every pair of arrows in-and-out of $\g$ in $Q(\SC)$ we will add the composite arrow (and flip that pair of arrows). Thus the quiver changes according to a mutation at $\g$ also in this general case. Similarly, since the only regions in Figure \ref{fig:CExchange_Before} using intersection points with $\g$ are precisely those entering a $u_{i,j}$ and exiting a $d_{l,k}$, or vice-versa, the change of polygons in $\SC$ will be recorded by tracking all the polygons that appear in the local configuration in Figure \ref{fig:CExchange_TwoBefore} (right) for {\it all} pairs of $u_{i,j}$ and $d_{l,k}$. We have already argued that the potential in each local configuration precisely changes as a (non-reduced) QP-mutation, and therefore it is also the case for the potential in the general configuration.
\end{proof}


\section{Non-degeneracy of curve QPs for plabic fences}\label{sec:QP_PlabicFence}

In this section we construct curve configurations $\SC$ from a certain type of plabic graph and show rigidity for their associated curve QP $(Q(\SC),W(\SC))$. Let us start with the plabic graphs that we need:

\begin{definition}\label{def:plabic_fence}
An embedded planar bicolored graph $\bG\sse\R^2$ is said to be a \emph{plabic fence} if it satisfies the following conditions.
\begin{itemize}\setlength\itemsep{0.5em}
    \item[(i)] The vertices of $\bG\sse\R^2$ belong to the standard integral lattice $\Z^2\sse\R^2$, and they are colored in either black or white.
    \item[(ii)] The edges of $\bG\sse\R^2$ belong to the standard integral grid $(\Z\times\R)\cup (\R\times\Z)\sse\R^2$. Edges that are contained in $\Z\times\R$ are said to be \emph{vertical}, and edges that are contained in $\R\times\Z$ are said to be \emph{horizontal}. 
    \item[(iii)] A maximal connected union of horizontal edges is called a \emph{horizontal line}. All horizontal lines must start, on the left, at univalent white vertices with the same $x$-coordinate and must end, on the right, at univalent white vertices with the same $x$-coordinate.
    \item[(iv)] Each vertical edge must end at trivalent vertices of opposite colors, with white on top and black on bottom, and the end points of a vertical edge must be contained in the interior of a horizontal line. In addition, no two vertical edges are contained in the same (vertical) line.
\end{itemize}
\end{definition}

\begin{center}
	\begin{figure}[H]
		\centering
		\includegraphics[scale=0.75]{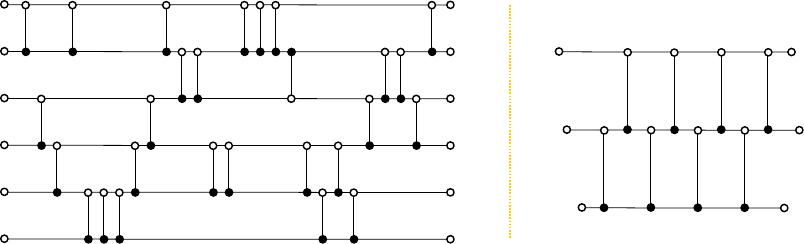}
		\caption{Two plabic fences. The fence on the right encodes $\beta=(\sigma_1\sigma_2)^4$.}\label{fig:PlabicFence_Example}
	\end{figure}
\end{center}

These are a special type of plabic graphs, studied in \cite[Section 12]{FPST}, \cite[Section 2.5]{CasalsWeng22}, \cite[Section 5]{STWZ}, and also \cite{GSW,ShenWeng}, the latter in the context of triangulations of flag configuration diagrams. For a visual instance, two plabic fences $\bG$ are drawn in Figure \ref{fig:PlabicFence_Example}.

\begin{remark}
The definition of plabic fences in \cite{FPST,CasalsWeng22} is more general but, for our purposes, it is without loss of generality that we can work with Definition \ref{def:plabic_fence}. This is due to the fact that cyclic rotation is a quasi-cluster transformation, as shown in \cite{CasalsWeng22,CGGLSS}; see also \cite{ShenWeng}.\hfill$\Box$
\end{remark}

\color{black}
We use the following bijection between plabic fences with $n$ horizontal edges and positive braid words in $n$-strands. For a plabic fence $\bG$ we construct a braid word $\beta(\bG)$ iteratively by scanning the plabic fence left-to-right: starting with the empty word $\beta(\bG)$, when we encounter a vertical edge between the $k$ and $(k+1)$st horizontal edges, counting from the bottom, we add the positive Artin generator $\sigma_k$ to the right of $\beta(\bG)$. Note that we obtain a positive braid word through this process, as all Artin generators being used in this construction are positive. Given a positive braid word $\beta$, the plabic fence that gives the braid word $\beta$ through this process will be denoted by $\bG(\beta)$. 

\color{black}


\subsection{Curve configuration $\SC(\bG)$ of a plabic fence $\bG$}\label{sssec:configuration_plabicfence} A plabic fence gives rise to a curve configuration $\SC(\bG)$ as follows. Consider the conjugate surface $\Sigma(\bG)$ of \cite[Section 1.1.1]{GonKen}; see also \cite[Section 2]{Goncharov_IdealWebs}, \cite[Section 4]{STWZ} or \cite[Section 2.1]{CasalsLi22} for definition and details on conjugate surfaces. For the purposes of this manuscript, this is a (ribbon) surface obtained from a plabic graph by using the three local models in Figure \ref{fig:ConjugateSurface_Models}.\footnote{The boundary of the surface is given by the alternating strand diagram of $\bG$, see \cite{Postnikov,Goncharov_IdealWebs}.} It retracts to $\bG$ and thus $H_1(\Sigma(\bG),\Z)\cong H_1(\bG,\Z)$, the latter being a free $\Z$-module of rank equal to the number of faces of $\bG$.\footnote{A face is a bounded connected component of the complement of the plabic fence in $\R^2$.}

\begin{center}
	\begin{figure}[H]
		\centering
		\includegraphics[scale=0.75]{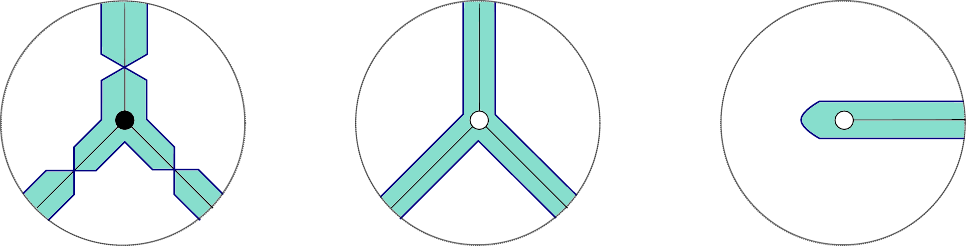}
		\caption{The three local models needed to draw a (projection of a) conjugate surface associated to a plabic fence $\bG$. The boundary of the surface is in dark blue and the surface itself in light blue.}\label{fig:ConjugateSurface_Models}
	\end{figure}
\end{center}

We allow the equivalence in Figure \ref{fig:ConjugateSurface_Models3} for conjugate surfaces. Combinatorially, this move is allowed in the literature so as to keep the underlying graph bicolored and with all edges having one black end and one white end, by introducing a white vertex in the middle region of Figure \ref{fig:ConjugateSurface_Models3} (left). The fact that this move does not affect the symplectic geometry in Section \ref{sec:Symplectic}, is verified rather simply\footnote{The move removes/creates a trivial $2$-cycle that would also feature in the potential, therefore the reduced part of the QP quiver of the associated curve configuration will remain invariant.}, e.g.~ it is proven in \cite[Section 2.1.3]{CasalsLi22}. We always consider the diagram for the conjugate surface $\Sigma(\bG)$ after we have applied this move (left to right in Figure \ref{fig:ConjugateSurface_Models3}) to every local double-crossing configuration as in Figure \ref{fig:ConjugateSurface_Models3} (left). That is, we remove any instances of Figure \ref{fig:ConjugateSurface_Models3} (left) in the diagram of $\Sigma(\bG)$, substituting them by Figure \ref{fig:ConjugateSurface_Models3} (right).

\begin{center}
	\begin{figure}[H]
		\centering
		\includegraphics[scale=0.85]{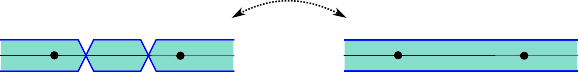}
		\caption{An equivalence of conjugate surfaces. This corresponds to a non-dangerous tangency in the alternating strand diagram.}\label{fig:ConjugateSurface_Models3}
	\end{figure}
\end{center}

Therefore, near a face $F\sse\bG$, the conjugate surface can be drawn as in Figure \ref{fig:PlabicFence_AddingCrossing}; the face $F$ is the central face, immediately to the left of the green vertical edge. Note that this is the general pattern near a face $F$: bounded between two vertical edges at the same level, with a series of vertical edges arriving with a black vertex from the level right above, and a series of vertical edges arriving with a white vertex from the level right below. We use the word {\it level} to denote the horizontal space (which might contain vertical edges) between two consecutive horizontal lines of the plabic fence. In particular, once we turn the plabic fence into a positive braid, the (height of the) level of a vertical edge corresponds
to the index of the corresponding Artin generator. The levels are indexed bottom to top, with the bottom level being the first level.

\begin{center}
	\begin{figure}[H]
		\centering
		\includegraphics[scale=0.9]{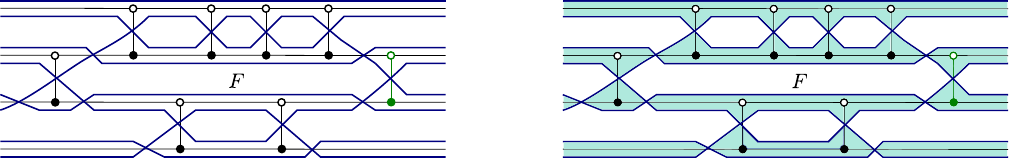}
		\caption{(Left) A plabic fence in black with the boundary of the conjugate surface drawn in blue. (Right) The conjugate surface associated to the plabic fence, with its interior drawn in turquoise.}\label{fig:PlabicFence_AddingCrossing}
	\end{figure}
\end{center}

The curve configuration $\SC(\bG)$ from $\bG$ is built as follows:

\begin{definition}\label{def:configuration_from_fence}
Let $\bG$ be a plabic fence and $\Sigma(\bG)$ its conjugate surface. The curve configuration $\SC(\bG)$ is the configuration of embedded closed curves in $\Sigma(\bG)$ constructed as follows:

\begin{itemize}
    \item[(i)] There is a curve $\g_F\in\SC(\bG)$ for every face $F\sse\bG$.

    \item[(ii)] Each curve $\g_F\sse\Sigma(\bG)$ is obtained by applying the local models in Figure \ref{fig:ConjugateSurface_Models2} near each vertex of $\bG$ and connecting the resulting segments by following the boundary of (the diagram of) $\Sigma(\bG)$ in a planar parallel manner.
\end{itemize}
The configuration $\SC(\bG)$ of curves in $\Sigma(\bG)$ is said to be the (curve) configuration associated to $\bG$.\hfill$\Box$
\end{definition}

\begin{center}
	\begin{figure}[H]
		\centering
		\includegraphics[scale=0.75]{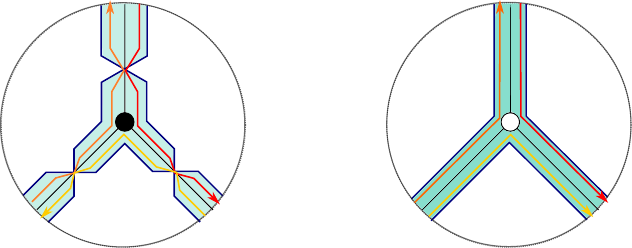}
		\caption{The two local models needed to associate a closed curve $\g_F$ in $\Sigma(\bG)$ to every face of $\bG$. In the left picture, two curves in $\SC(\bG)$ intersect transversely in the conjugate surface (inside the twisted band). For instance, the orange and red curve intersect transversely: the twist in the band might make it appear otherwise but the intersection is transverse in the conjugate surface. The convention for the twist is fixed in this local picture and so are the intersection signs.}\label{fig:ConjugateSurface_Models2}
	\end{figure}
\end{center}

Figure \ref{fig:PlabicFence_AddingCrossing2} (left) depicts a plabic fence with a face $F$ and Figure \ref{fig:PlabicFence_AddingCrossing2} (right) depicts its associated closed embedded curve $\g_F$ in $\Sigma(\bG)$. Note again that this is the generic form of any plabic fence near a face $F$, possibly with different numbers of vertical edges at the levels right above and below, which would in any case not change $\g_F$.


\begin{center}
	\begin{figure}[H]
		\centering
		\includegraphics[scale=0.9]{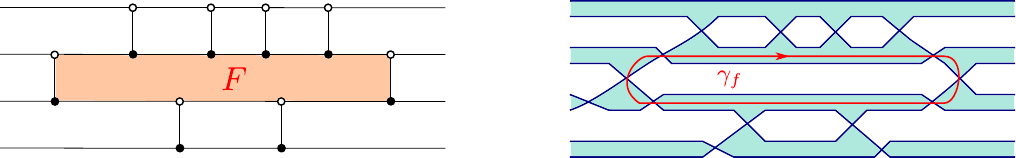}
		\caption{(Left) A plabic fence $\bG$ in black with a face $F$ highlighted. (Right) The curve $\g_F$ in $\Sigma(\bG)$ associated to the face $F$.}\label{fig:PlabicFence_AddingCrossing2}
	\end{figure}
\end{center}

Note that we can (and do) draw the curves in $\SC(\bG)$ such that two curves $\g_{F_1}$ and $\g_{F_2}$ in $\SC(\bG)$ only intersect at the precise twist point of the ribbon diagram for $\Sigma(\bG)$. That is, the only intersections of curves in $\SC(\bG)$ that occur are of the form depicted in Figure \ref{fig:PlabicFence_AddingCrossing4}.

\begin{center}
	\begin{figure}[H]
		\centering
		\includegraphics[scale=1.1]{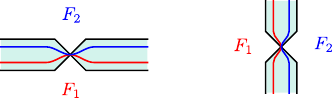}
		\caption{The two local models near an intersection point of two curves in $\SC(\bG)$.}\label{fig:PlabicFence_AddingCrossing4}
	\end{figure}
\end{center}

In the next subsection we introduce a QP $(Q(\bG),W(\bG))$ that will be useful to understand the properties of the curve QP associated $\SC(\bG)$.

\subsection{The QP associated to a plabic fence}\label{ssec:QP_plabicfence} There are several descriptions of quivers $Q(\bG)$ associated to a plabic fence $\bG$, see \cite{BLL18,CasalsWeng22,FPST,GonKen,GSW,Postnikov} for some; they are all equivalent. For specificity, we explicitly describe the quiver $Q(\bG)$ in the upcoming Defininion \ref{def:quiver_plabicfence}.

A piece of notation: suppose that $e\sse\bG$ is a vertical edge at level $k$, we denote by $F_e$ the face of $\bG$ that has $e$ as its right vertical edge. This face $F_e$ is unique or it does not exist. By definition, a black pente-row (resp.~white pente-row) is a consecutive collection of black vertices in the same horizontal line of $\bG$ such that:

\begin{itemize}
    \item[-] There must be two white (resp.~black) vertices bounding it: one white vertex at its left, and one white vertex at its right.\footnote{This is in the spirit of pente capture configurations, following the board game Pente.}\\

    \item[-] Each connected component of $\R^2\setminus\bG$ whose closure contains any segment between two vertices of the black vertices above or between a black vertex and one of the two white vertices above must be a  face.
\end{itemize}

The total number of black (resp.~white) vertices in a black (resp.~white) pente-row is said to be its length. See Figure \ref{fig:PlabicFence_Row1} for a length four black pente-row and a length four white pente-row, where we have marked the connected components of $\R^2\setminus\bG$ that must be faces with a blue dot. The rightmost face in a black (resp.~white) pente-row, which has a left corner at the rightmost black (resp.~white) vertex, is said to be its right face.


\begin{center}
	\begin{figure}[H]
		\centering
		\includegraphics[scale=0.9]{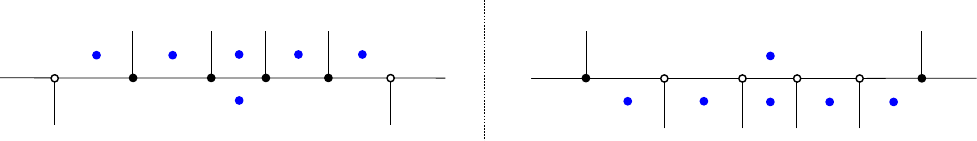}
		\caption{(Left) A black pente-row. (Right) A white pente-row.}\label{fig:PlabicFence_Row1}
	\end{figure}
\end{center}

\begin{definition}\label{def:quiver_plabicfence} Let $\bG$ be the plabic fence with $n$ horizontal edges, its associated QP $(Q(\bG),W(\bG))$ is defined as follows. The quiver $Q(\bG)$ has vertex set the set of faces of $\bG$. The arrow set of $Q(\bG)$ is inductively described as follows, scanning $\bG$ left to right:

\begin{itemize}
    \item[(i)] If $\bG$ is the empty plabic fence, then the arrow set of $Q(\bG)$ is empty.\\

    \item[(ii)] Choose a vertical edge $e\sse\bG$ at level $k$ and assume that the arrow set of $Q(\bG_{<e})$ is $A_{<e}$, where $\bG_{<e}$ is the plabic subfence of $\bG$ consisting of those vertical edges to the left of $e$. If $F_e$ does not exist, the arrow set of $Q(\bG_{\leq e})$, where $\bG_{\leq e}=\bG_{<e}\cup\{e\}$, is defined to be $A_{<e}$.\\
    
    \noindent If $F_e$ exists, the arrow set of $Q(\bG_{\leq e})$ is defined to be $A_{<e}$ union the following possible arrows:\\

    \begin{itemize}
        \item[(a)] Let $d$ be the left vertical edge of $F_e$. If $F_d$ exists, then we add an arrow $[de]$ from $F_d$ to $F_e$.\\

       \item[(b)] Let $d^\uparrow$ be the first vertical edge in $\bG$ at level $(k+1)$ to the right of $d$. If $d^\uparrow$ and $F_{d^\uparrow}$ exist, then we add an arrow $[ed^\uparrow]$ from $F_e$ to $F_{d^\uparrow}$. See Figure \ref{fig:PlabicFence_Quiver} (left) for such an arrow, marked with a pink $(Z)$ pattern.\\

       \item[(c)] Let $d^\downarrow$ be the first vertical edge in $\bG$ at level $(k-1)$ to the right of $d$. If $d^\downarrow$ and $F_{d^\downarrow}$ exist, then we add an arrow $[ed^\downarrow]$ from $F_e$ to $F_{d^\downarrow}$. See Figure \ref{fig:PlabicFence_Quiver} (left) for such an arrow, marked with a red $(S)$ pattern.\\
    \end{itemize}

 \noindent If the hypotheses in these cases are not met, no arrows are added: e.g.~if $F_e$ or $F_d$ do not exist in case (a), we do not add any arrows at that stage.\\
\end{itemize}

\begin{center}
	\begin{figure}[H]
		\centering
		\includegraphics[scale=0.9]{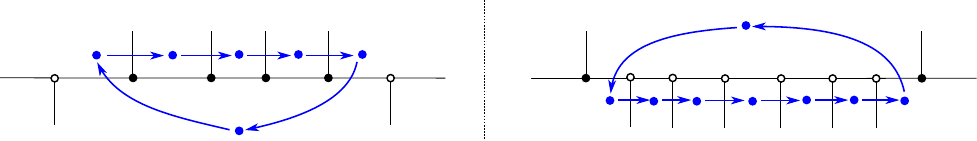}
		\caption{(Left) A black pente-row with its associated cycle around it. (Right) A white pente-row with its associated cycle around it.}\label{fig:PlabicFence_Row1_Example}
	\end{figure}
\end{center}

Note that $Q(\bG)$ around a black pente-row of $\bG$ has a (planar) clockwise cycle. The length of this cycle equals the number of black vertices in the black pente-row plus two. See Figure \ref{fig:PlabicFence_Row1_Example} (left) for an instance of such a quiver cycle. Similarly, the quiver $Q(\bG)$ around a white pente-row has a (planar) counter-clockwise cycle. The length of such a cycle equals the number of white vertices in the white pente-row plus two. See Figure \ref{fig:PlabicFence_Row1_Example} (right) for an example of such a quiver cycle. See Figure \ref{fig:PlabicFence_Quiver} (right) for an instance of $Q(\bG)$.

\begin{center}
	\begin{figure}[H]
		\centering
		\includegraphics[scale=0.9]{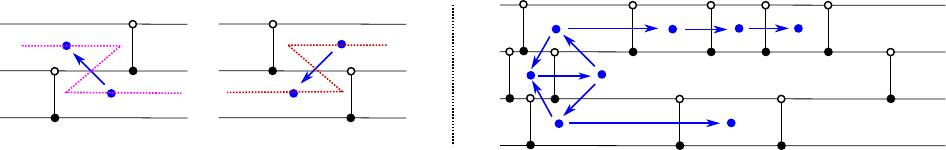}
		\caption{(Left) The two local patterns for arrows in $Q(\bG)$ that are not horizontal. (Right) An example of $Q(\bG)$.}\label{fig:PlabicFence_Quiver}
	\end{figure}
\end{center}

The potential $W(\bG)$ is similarly described inductively. If $\bG$ is the empty plabic fence, then $W(\bG)=0$. For each vertical edge $e$ added to the right of $\bG_{<e}$ as above, the potential is defined as
$$W(\bG_{\leq e})=W(\bG_{<e})+P_{black}(e)-P_{white}(e),$$
where the monomials in $P_{black}(e)$ and $P_{white}(e)$ are described as follows. By definition, $P_{black}(e)$, resp.~$P_{white}(e)$, is the (cyclic) monomial in the arrows of $Q(\bG)$ encoding the planar cycle in $Q(\bG)$ associated to the unique black (resp.~white) pente-row with right face equal $F_e$, if such pente-row exists and else it equals zero. Note that monomials in $P_{white}(e)$ have an overall minus sign in front due to the counter-clockwise orientation of their cycles.\hfill$\Box$
\end{definition}

\begin{remark}
Note that a pente-row of length one with right face $F_e$ gives rise to a triangle cycle in $Q(\bG)$ and thus a cubic monomial in the potential. Thus this construction generalizes known quivers with potentials from bipartite graphs, cf.~\cite[Section 5.1.2]{Goncharov_IdealWebs}. For instance, the case of $\beta=w_0$, the longest element in the symmetric group $S_n$, recovers the potential associated to the $n$-triangulation of a triangle, see \cite[Section 1.1]{Goncharov_IdealWebs} or \cite[Section 3.1]{CasalsZaslow}. \hfill$\Box$
\end{remark}

\begin{remark} For a non-inductive definition of $Q(
\bG)$, cf.~\cite[Section 1]{GSW}. The inductive nature of Definition \ref{def:quiver_plabicfence} is useful in our inductive proof of Propositions \ref{lem:QP_coincide} and \ref{thm:QPnondeg}. Also, there is no need to keep track of frozen vertices in Definition \ref{def:quiver_plabicfence}. There are natural generalizations to iced quivers, e.g..~see \cite{Pressland20}, \cite{GSW} or \cite{CasalsWeng22}.\hfill$\Box$
\end{remark}

The main reason to introduce $(Q(\bG),W(\bG))$ is to be able to describe the QP associated to $\SC(\bG)$ combinatorially in terms of $\bG$. Indeed, we have:

\begin{prop}\label{lem:QP_coincide}
Let $\bG$ be a plabic fence and $\SC=\SC(\bG)$ its associated curve configuration. Then the curve QP $(Q(\SC),W(\SC))$ equals the QP $(Q(\bG),W(\bG))$.
\end{prop}

\begin{proof}
Let us scan the plabic fence $\bG$ left to right and compare $Q(\bG)$ and $Q(\SC)$ as we do so. Let $e\sse\bG$ be a vertical edge at level $k$ and $F=F_e$ its associated face.\footnote{If $F_e$ does not exist, there is no curve added nor the quiver gains vertices, thus both coincide.} Since $Q(\bG)$ and the curves in $\SC$ associated to faces at level $k$ are local in the union of levels $(k-1),k$ and $(k+1)$st, it suffices to consider the piece of plabic fence in the neighborhood of $F$ consisting of all faces that share a vertical edge (there is at most one such face) or a piece of a horizontal edge with $F$. An example of a neighborhood of $F$ is drawn in Figure \ref{fig:PlabicFence_AddingCrossing3}, with the relevant set of curves in $\SC$ on its left and the quiver on its right. Note that we are scanning left to right and assuming that the vertical edge $e$ is the rightmost edge of $\bG$ at this stage. For the plabic fence in Figure \ref{fig:PlabicFence_AddingCrossing3} (left), this is the edge depicted in green in Figure \ref{fig:PlabicFence_AddingCrossing}.

\begin{center}
	\begin{figure}[H]
		\centering
		\includegraphics[scale=0.9]{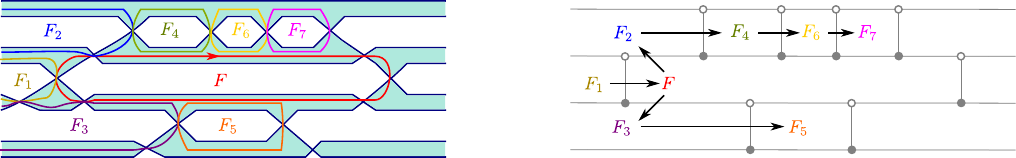}
		\caption{(Left) A neighborhood of a face $F$ with the conjugate surface depicted and all the curves $\g_{F_i}$ for the faces $F_i$ of $\bG$ adjacent to $F$. (Right).}\label{fig:PlabicFence_AddingCrossing3}
	\end{figure}
\end{center}

The curve $\g_F$ associated to the new vertex of $Q(\bG_{\leq e})$ not in  $Q(\bG_{<e})$ intersects the other curves $\g_{F_i}$, $F_i$ faces in $\bG_{<e}$, when there is a ribbon twist. (Recall Figure \ref{fig:PlabicFence_AddingCrossing4}.) The curve $\g_F$ only traverses twists that respectively connect to:

\begin{itemize}
    \item[-] The face $F_d$, where $d$ is the left vertical edge of $F_e$, if it exists.

    \item[-] The face $F_{d^\uparrow}$, as in Definition \ref{def:quiver_plabicfence}, if it exists.

    \item[-] The face $F_{d^\downarrow}$, as in Definition \ref{def:quiver_plabicfence}, if it exists.

    \item[-] The region to the right of $e$ that, by hypothesis at this stage ($e$ being the rightmost edge we have scanned) is unbounded and thus not a face.
\end{itemize}

The three intersections $\g_{d}\cap\g_{e}$, $\g_{d^\uparrow}\cap\g_{e}$ and $\g_{d^\downarrow}\cap\g_{e}$ are precisely recorded by the arrows $[de]$, $[ed^\uparrow]$ and $[de],[ed^\downarrow]$, if they respectively exist. The direction of the arrow is given by the intersection sign, which is positive for $\g_{d}\cap\g_{e}$ and negative for $\g_{d^\uparrow}\cap\g_{e}$ and $\g_{d^\downarrow}\cap\g_{e}$. Therefore $Q(\bG)$ and $Q(\SC)$ coincide.

In order to compare $W(\bG)$ and $W(\SC)$ it suffices to note that a polygon bounded by $\SC$ in the conjugate surface $\Sigma(\bG)$ must have vertices be as in Figure \ref{fig:PlabicFence_AddingCrossing5}. \color{black} Indeed, consider for instance Figure \ref{fig:PlabicFence_AddingCrossing4} (left), which depicts a piece of the conjugate surface (a twisted band) with two pieces of curves from $\SC$ in it. The curves from $\SC$ are drawn in blue and red in Figure \ref{fig:PlabicFence_AddingCrossing4}, the conjugate surface is highlighted in a lighter blue. These two red and blue pieces of curves intersect at a unique point. Near that intersection point, a polygon which has a vertex at that unique point must equal the yellow region depicted in Figure \ref{fig:PlabicFence_AddingCrossing5} (upper left) or Figure \ref{fig:PlabicFence_AddingCrossing5} (upper right). That is, it will either be to the left of the intersection point (and bounded by the red and blue curves) in the former case, or to its right, in the latter case. The case of Figure \ref{fig:PlabicFence_AddingCrossing4} (left), where the twisted band is drawn vertical, is analogous, and leads to the polygons in Figure \ref{fig:PlabicFence_AddingCrossing5} (lower left) and Figure \ref{fig:PlabicFence_AddingCrossing5} (lower right).\\


\color{black}
Therefore, polygons bounded by $\SC$ must correspond to pente-rows. Indeed, a pente-row gives a unique polygon bounded by $\SC$ by drawing $\Sigma(\bG)$ near the pente-row and cutting at the ribbon twists around the pente-row; each ribbon twist corresponds to a vertex of the polygon. Such cuts give an embedded planar region and there is a unique polygon embedded in it with vertices given by the (locations where we cut the) ribbon twists. Note that we cut as many vertical ribbon twists as the length of the pente-row and we always cut two horizontal ribbon twists, i.e.~the polygon has as many sides as the length of the pente-row plus two. The polygon is oriented clockwise, resp.~counter-clockwise, if the row is black, resp.~white.

\begin{center}
	\begin{figure}[H]
		\centering
		\includegraphics[scale=1.1]{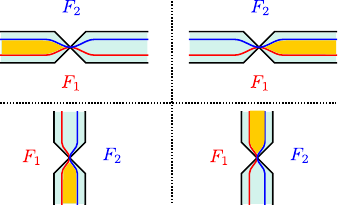}
		\caption{\color{black} In yellow, pieces of regions in the conjugate surface that have the potential to give a polygon bounded by $\SC$. As before, the conjugate surface is drawn in light blue (with a twisted band) and the two curves from $\SC$ are drawn in blue and red. In the figures of the first row, a yellow polygon takes either the left side (depicted on the upper left figure) or the right side (depicted on the upper right figure) of the unique intersection point in the ribbon twist. In the figures of the second row, a yellow polygon takes either the lower side (depicted on the lower left figure) or the upper side (depicted on the lower right figure) of the unique intersection point in the ribbon twist.\color{black}}
\label{fig:PlabicFence_AddingCrossing5}
	\end{figure}
\end{center}

Conversely, the conjugate surface $\Sigma(\bG)$ has a ribbon twist in any segment (horizontal or vertical) between a black and a white vertex, and no twist if a segment is between two vertices of the same color. Embeddedness of the polygon in $\Sigma(\bG)$ implies that it must lie within a region bounded by twists (and no twists on the interior of that region). Since the ends of a polygon must be as in Figure \ref{fig:PlabicFence_AddingCrossing5}, the only possible polygons bounded by $\SC$ must have the form of those around a pente-row. In conclusion, polygons bounded by $\SC$ in $\Sigma(\bG)$ must correspond to pente-rows and the potentials $W(\bG)$ and $W(\SC)$ coincide as we scan $\bG$ left to right.\end{proof}

We conclude this subsection with a combinatorial property satisfied by the quivers $Q(\bG)$. This property will be used in the proof of Proposition \ref{thm:QPnondeg}. The following proof was explained to us by D.~Weng, to whom we are grateful:

\begin{lemma}\label{lem:source}
Let $\bG$ be a plabic fence and $Q(\bG)$ its associated quiver. Consider the rightmost vertical edge $e\in\bG$ and the vertex $F_e\in Q(\bG)_0$ associated to the face $F_e$ immediately to the left of $e$, if it exists. Then there exists a sequence of vertices $(v_1,\ldots,v_k)$ in $v_i\in Q(\bG_{<e})_0$, $i\in[k]$ and $k\in\N$ depending on $e$, such that $\mu_k\cdots\mu_1(Q(\bG))$ has $F_e$ be a source vertex.
\end{lemma}

\begin{proof}
Let us consider the set of faces $F_1,\ldots,F_k,F_{k+1}$ at the same level as the edge $e$. We order them left to right as we scan $\bG$, thus $F_1$ is the leftmost face of $\bG$ at that level, $F_2$ is the face adjacent to $F_1$ exactly to its right, and so on until $F_{k+1}=F_e$. The claim is:

{\underline {\it Assertion}}. $(v_1,\ldots,v_k)=(F_1,\ldots,F_k)$ is a sequence that turns $F_{k+1}=F_e$ into a source.\\

Before we prove the assertion, three comments. First, we recall that \cite{ShenWeng}, or \cite[Section 5]{GSW}, explain how to associate a quiver to a more general type of plabic fence than that in Definition \ref{def:plabic_fence}.\footnote{These are instances of the general theory of quivers for plabic graphs, see \cite{Postnikov}.} Namely, to a plabic fence that also allows vertical edges to have a black vertex at the top and a white vertex at the bottom. See also \cite[Section 2]{CasalsWeng22} for such objects, their quivers and their properties. These are essentially all rephrasings of the theory of double wiring diagrams, cf.~\cite[Section 2.4]{FWZ}, now in the plabic graph terminology and generalized to the non-reduced case. Just for this proof, we refer to these more general plabic fences also as plabic fences.

Now, the reflection move in \cite[Section 2.3]{ShenWeng}, see also \cite[Section 5.2]{CasalsWeng22} or \cite[Section 4]{GSW}, is as follows. A reflection move $r_k$ at level $k$ of a plabic fence $\bG$ is the operation that inputs $\bG$ and outputs a plabic fence $r_k(\bG)$ which coincides with $\bG$ everywhere except that the leftmost vertical edge at level $k$ has been flipped.\footnote{Here flipping means: if it had a black vertex at the top, now it has a black vertex at the bottom, and vice versa.} A reflection move does {\it not} change the quiver $Q(\bG)$, see \cite{ShenWeng}.

Second, for each face $F$, let us denote by $\dd_-F$, resp.~$\dd_+F$ its left vertical edge, resp.~right vertical edge. Suppose that we have a square face in $\bG$ which has a $\dd_-F$ with a black vertex at the top and $\dd_+F$ with a black vertex at the bottom. Then the square move, cf.~\cite[Section 2.5]{FWZ}, is the operation that exchanges $\dd_-F$ and $\dd_+F$, i.e.~the new plabic fence $\mu_F(\bG)$ coincides with $\bG$ everywhere except that $\dd_-F$ now has a black vertex at the bottom and $\dd_+F$ now has a black vertex at the top. A square move induces a mutation of $Q(\bG)$ at the vertex $F$, i.e.~$Q(\mu_F(\bG))=\mu_F(Q(\bG))$.

\begin{center}
	\begin{figure}[H]
		\centering
		\includegraphics[scale=0.75]{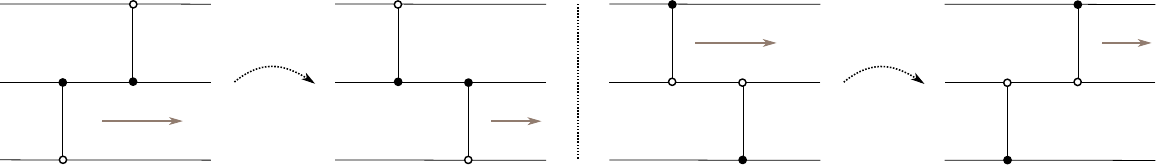}
		\caption{Two moves on plabic fences that do not change the quiver. They allows us to move a certain type of vertical edge to the right in the presence of certain vertical edges above and below it.}\label{fig:PlabicFence_Example2}
	\end{figure}
\end{center}

Third, the two moves in Figure \ref{fig:PlabicFence_Example2} allow us to slide right a vertical edge with black at the top through vertical edges with black at the bottom to its upper right or lower right. Let us now prove the assertion above, which will conclude the proof for the lemma.

{\it Proof of Assertion}: Let us consider the initial plabic fence $\bG$ and apply a reflection $r_k(\bG)$, where $k$ is the level of $e\in\bG$. We still have $Q(r_k(\bG))=Q(\bG)$. This reflection creates a vertical edge $\delta$ with black on top at the leftmost part of level $k$ of $r_k(\bG)$; all the remaining vertical edges in $r_k(\bG)$ have black at the bottom. Then we slide $\delta$ to the right through level $k$ by either applying the sliding moves in Figure \ref{fig:PlabicFence_Example2}, or using square moves. The former has no effect on the quiver, the latter induce quiver mutations.

Let us slide $\delta$ to the right until $\delta$ is exactly to the left of $e$, so that (the new) $F_e$ has $\dd_-F_e=\delta$ and $\dd_+F_e=e$. Denote by $\bG_{\delta,e}$ this plabic fence. Then $\delta$ will have slid through faces $F_1,\ldots,F_k$ when going from $\bG$ to $\bG_{\delta,e}$. Thus at this stage we have performed a sequence of mutations at those faces, in this same order, and $Q(\bG_{\delta,e})=\mu_{F_k}\cdots\mu_{F_1}(Q(\bG))$. To conclude, it suffices to note that the quiver $Q(\bG_{\delta,e})$ associated to the plabic fence $Q(\bG_{\delta,e})$ has a unique arrow out of the vertex associated to $F_e$, and therefore $F_e$ is a source, as required.
\end{proof}

Lemma \ref{lem:source} states that we can find a sequence of mutations for $Q(\bG)$ so that the vertex corresponding to the face immediately left of rightmost vertical edge of $\bG$ becomes a source vertex, and such that this sequence of mutations never includes a mutation at that particular vertex that we want to turn into a source vertex.

\begin{remark}
This is not needed for our applications, but note that the proof of Lemma \ref{lem:source} explicitly presents a mutation sequence that turns $F_e$ into a source.\hfill$\Box$
\end{remark}


\subsection{Non-degeneracy of $(Q(\SC(\bG)),W(\SC(\bG)))$}\label{ssec:nondeg_def} The notion of non-degeneracy of a QP was introduced in \cite[Section 7]{DWZ}. It reads as follows:

\begin{definition}\label{def:non_degenerate}
Let $(Q,W)$ be a QP and $k_1,\ldots, k_l\in Q_0$ be a finite sequence of vertices, no two consecutive ones being equal. By definition, $(Q,W)$ is $(k_1,\ldots, k_l)$-non-degenerate if all
the QPs $(Q,W)$, $\mu_{k_1}(Q,W),\mu_{k_2}\mu_{k_1}(Q,W),\ldots,\mu_{k_l}\cdots\mu_{k_2}\mu_{k_1}(Q,W)$ are 2-acyclic. By definition $(Q,P)$ is non-degenerate if it is $(k_1,\ldots, k_l)$-non-degenerate
for every such sequence of vertices.\hfill$\Box$
\end{definition}

Thus, in a non-degenerate $(Q,W)$ we can choose an arbitrary sequence of vertices in $Q$ and mutate the QP $(Q,W)$ along that sequence of vertices. That is, it is possible to perform QP-mutations to $(Q,W)$ just by naming the quiver mutations of $Q$. There are some QPs that are non-degenerate and some that are degenerate, see \cite{DWZ}, especially Sections 7 and 8 therein, and, for instance, \cite[Section 2.2]{chang_zhang_2022}, where quivers associated to the top positroid variety of a Grassmannians are shown to be non-degenerate.

Before establishing non-degeneracy of $(Q(\SC),W(\SC))$, we recall that \cite[Definition 6.10]{DWZ} defines a QP $(Q,W)$ to be rigid if the trace space of its Jacobian algebra is equal to the ground ring $\C$. Rather than the definition itself, the three following properties about rigid QPs are most relevant to us:

\begin{itemize}
    \item[-] Rigidity is preserved under QP mutation. That is, if $(Q,W)$ is rigid, then $\mu_v(Q,W)$ is rigid for any vertex $v\in Q_0$. This is established in \cite[Corollary 6.11]{DWZ}.\\

    \item[-] If $(Q,W)$ is rigid and $Q'$ is obtained by adding one vertex $v$ to $Q$, so that $Q'_0=Q_0\cup\{v\}$, such that $v$ is a source (or a sink) in $Q'$, then $(Q',W')$ with $W'=W$ is also rigid.\footnote{In the quiver literature, this operation is part of what is known as a triangular extension of $Q(\bG_{<e})$ and a point, cf.~\cite[Definition 3.8]{Amiot09}.} This is proven in \cite[Remark 4.4]{Ladkani13} or cf.~\cite[Section 8]{DWZ}.\\
    
    \item[-] A rigid QP $(Q,W)$ is non-degenerate. This is proven in \cite[Corollary 8.2]{DWZ}.
\end{itemize}

Let us now show that QP $(Q(\SC),W(\SC))$ associated to curve configurations $\SC$ such that $\SC=\SC(\bG)$ for a plabic fence are indeed non-degenerate.

\begin{prop}\label{thm:QPnondeg}
Let $\bG$ be a plabic fence and $\SC=\SC(\bG)$ its associated curve configuration. Then the QP $(Q(\SC),W(\SC))$ is rigid, and thus non-degenerate.
\end{prop}

\begin{proof}
By Lemma \ref{lem:QP_coincide}, it suffices to argue that $(Q(\bG),W(\bG))$ as in Definition \ref{def:quiver_plabicfence} is rigid. As usual, we prove this by scanning the plabic fence $\bG$ left to right. First, the empty quiver $(Q(\bG),W(\bG))$ is rigid and non-degenerate. Second, we assume that we have scanned until the vertical edge $e$ and we have rigidity for $W(\bG_{<e})$. Let us consider two cases:

\begin{itemize}
    \item[-] If $[de]$ does not exist, then $F_e$ is a source vertex and $Q(\bG_{\leq e})$ is obtained from $Q(\bG_{< e})$ by adding a source, thus no new cycles appear and the potential is still rigid by the second property above. Similarly, if neither $[ed^\uparrow]$ and $[ed^\downarrow]$ exist, then $F_e$ is a sink vertex and the same argument applies.\\

    \item[-] If either $[ed^\uparrow]$ or $[ed^\downarrow]$ exist and so does $[de]$, then $F_e$ is neither a sink nor a source.  Nevertheless, we can apply Lemma \ref{lem:source} to turn $F_e$ into a source by mutating at vertices in $Q(\bG_{<e})$. Indeed, the sequence of mutations $\mu_v$ from Lemma \ref{lem:source} is at an ordered collection of vertices $v:=(v_1,\ldots,v_k)$ which are already vertices in $Q(\bG_{<e})$, i.e.~ it does not mutate at $F_e$. By hypothesis, $W(\bG_{<e})$ is rigid and, as stated above, rigidity is preserved by QP mutation. Therefore $\mu_v(W(\bG_{<e}))$ is also rigid. Since restriction commutes with mutation by \cite[Lemma 2.5]{FominIgusaLee21} and the mutations occur at vertices of $Q(\bG_{<e})$, the mutated potential $\mu_v(W(\bG_{\leq e}))$ restricted to the subquiver $Q(\bG_{<e})$ is also rigid, as it equals $\mu_v(W(\bG_{<e}))$. Since the quiver $\mu_v(Q(\bG_{\leq e}))$ is given by adding a source to $\mu_v(Q(\bG_{<e}))$, the resulting potential $\mu_v(W(\bG_{\leq e}))$, now without restricting, is still rigid. The original potential $W(\bG_{\leq e})$ is QP mutation-equivalent to this resulting potential, and thus also rigid.
\end{itemize}
In summary, in either case, the QP $(Q(\bG),W(\bG))$ stays rigid as we scan $\bG$ left to right. Therefore $(Q(\SC(\bG)),W(\SC(\bG)))$ is non-degenerate.
\end{proof}

Note that there are many QP $(Q,W)$ that might not be rigid (or non-degenerate). The class of quivers with potentials $(Q(\bG),W(\bG))$ that we introduced and analyzed above is particular enough that rigidity can be argued directly as in Proposition \ref{thm:QPnondeg}. This is similar to the fact that the quivers in the Kontsevich-Soibelman class $\mathcal{P}$\footnote{This is the class generated by the one vertex quiver by triangular extensions and mutations.} are rigid and admit a unique non-degenerate potential up to right-equivalence. In fact, though it will not be needed for our application, it can also be proven that $W(\bG)$ is the unique non-degenerate potential for $Q(\bG)$ up to right-equivalence.

\begin{remark} It is likely that the more general class of curve configurations $\SC(\mathfrak{w})$ that we associated to the weaves $\mathfrak{w}$ in \cite{CGGLSS} also have $W(\SC(\mathfrak{w}))$ be non-degenerate; see \cite[Section 7.3]{CGGLSS}. In that case, the arguments in Section \ref{sec:Symplectic} would also prove Theorem \ref{thm:main} for $(-1)$-closures.\hfill$\Box$
\end{remark}

This concludes our construction and study of the QPs $(Q(\SC),W(\SC))$. We have established in Section \ref{sec:curves} the invariance of their right-equivalence classes under triple point moves and local bigon moves, proven that they undergo a QP-mutation when a $\g$-exchange is applied to $\SC$, and now shown rigidity for those configurations associated to plabic fences. There are two technical pieces that still need justification: arguing that Assumption \ref{assumption:bigons} holds and proving Lemma \ref{lem:reduced_parts}. We conclude this section by presenting such proofs.

\color{black}
\subsection{The assumption on bigons is satisfied} Recall that given a local bigon $B\sse\Sigma$ bounded by a curve configuration $\SC$, $\rho(B)$ denotes the region to its right and $\la(B)$ the region to its left. See Figure \ref{fig:LocalBigon} in Subsection \ref{sssec:local_property_bigons}. In this subsection we show that Assumption \ref{assumption:bigons} is satisfed. We recall the statement of the assumption:

{\bf Assumption 1}. Every local bigon $B\sse\Sigma$ is assumed to satisfy that there is no polygon bounded by $\SC$ which contains both $\la(B)$ and $\rho(B)$.\hfill$\Box$

\label{ssec:assumption_bigons} By definition, a cycle $v_1\ldots v_n$ in a QP $(Q,W)$ is said to be empty if $v_1\ldots v_n$, or any of its cyclically equivalent forms, does not appear as a monomial in $W$. The following lemma shows that we have no empty cycles for QP associated to curve configurations:

\begin{lemma}\label{lem:empty_cycles_degeneracy}
Let $(Q,W)$ be a QP and assume that it contains an empty cycle, i.e.~a cycle of arrows in $Q$ whose corresponding monomial does not appear in the potential $W$. Then $(Q,W)$ is degenerate.
\end{lemma}

\begin{proof}
Let us show this by induction on $n\in\N$, where $n$ is the number of arrows of an empty $n$-cycle $v_1\ldots v_n$. The base case is $n=2$: suppose there is a 2-cycle $v_1v_2$ in $Q$ such that $v_1v_2$ is not a monomial in $W$. By Definition \ref{def:non_degenerate}, without mutating at all, the quiver in the reduced part of $(Q,W)$ still has the 2-cycle $v_1v_2$ and thus $(Q,W)$ is degenerate.\\

The induction step is as follows. Consider an empty $n$-cycle $v_1\ldots v_n$ for $(Q,W)$ and assume that the existence of an empty $(n-1)$-cycle in a QP implies its degeneracy. Mutation of $(Q,W)$ at the vertex $v=h(v_1)=t(v_2)$, creates the $(n-1)$-cycle $[v_1v_2]v_3\ldots v_n$. By Definition \ref{def:QP_mutation}, the potential $\mu_v(W)$ for the mutated QP $\mu_v(Q,W)$ is obtained by substituting $v_1v_2$ by $[v_1v_2]$ in $W$ and adding the cubic term $[v_1v_2]v_1^\ast v_2^\ast$. Therefore, $\mu_v(W)$ contains the monomial $[v_1v_2]v_3\ldots v_n$ if and only if $W$ contains the monomial $v_1\ldots v_n$. By assumption, $W$ does not contain $v_1\ldots v_n$ and thus $[v_1v_2]v_3\ldots v_n$ is an empty $(n-1)$-cycle in $\mu_v(Q,W)$. By induction, $\mu_v(Q,W)$ is degenerate and thus so is $(Q,W)$.
\end{proof}

Let us now use Lemma \ref{lem:empty_cycles_degeneracy} to prove that Assumption \ref{assumption:bigons} is satisfied for our non-degenerate configurations. The precise statement reads as follows:

\begin{prop}\label{prop:local_bigons}
Let $\SC$ be a curve configuration and $B\sse\Sigma$ a local bigon bounded by $\SC$. Suppose that $(Q(\SC),W(\SC))$ is a non-degenerate potential. Then there is no polygon bounded by $\SC$ which contains both $\la(B)$ and $\rho(B)$.
\end{prop}

\begin{proof}
By contradiction, suppose that there exists a local bigon $B$ such that $\rho(B)$ and $\la(B)$ are both contained in the same polygon. We will now argue that either:

\begin{enumerate}
    \item The potential $(Q,W)$ is degenerate, or
    \item The homology classes in $H_1(\Sigma;\Z)$ of the curves in $\SC$ do not span $H_1(\Sigma;\Z)$, because there is a sub-collection of them that give linearly dependent homology classes.
\end{enumerate}

\begin{center}
	\begin{figure}[H]
		\centering
		\includegraphics[scale=0.75]{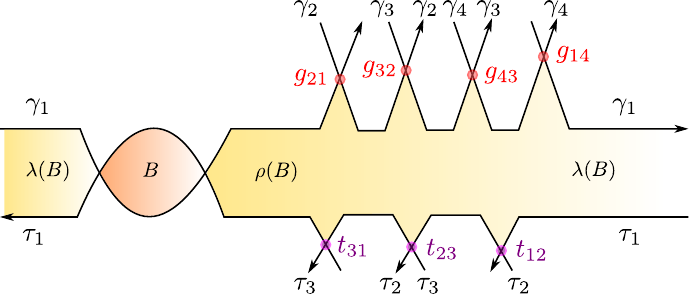}
		\caption{A configuration $\SC$ near bigon $B$, bounded by $\g_1,\tau_1$, with $\la(B)=\rho(B)$.}
		\label{fig:Bigon_LeftRight}
	\end{figure}
\end{center}

Let us consider a neighborhood of such local bigon $B$. It is as depicted in Figure \ref{fig:Bigon_LeftRight}, where the polygon $P_{\la,\rho}$ containing $\la(B)$ and $\rho(B)$ is drawn in yellow and $B$ is drawn in orange. To be precise, this yellow polygon might have curves from $\SC$ going through it, as illustrated in Figure \ref{fig:Bigon_LeftRight15}. These curves through the polygon (depicted in blue and dashed in Figure \ref{fig:Bigon_LeftRight15}) have no effect and we will just not drawn them in the figures for the rest of the argument. Now, following Figure \ref{fig:Bigon_LeftRight}, the yellow polygon $P_{\la,\rho}$ is bounded by:

\begin{enumerate}
    \item A collection of curves $\g_2,\ldots,\g_s\in\SC$, $s\in\N$, with intersection pattern exactly given by $\g_i$ intersects $\g_{i+1}$ negatively, at a point $g_{i+1,i}$ and $\g_g$ intersects $\g_1$ negatively at $g_{1,s}$.\\

    \item A collection of curves $\tau_2,\ldots,\tau_l\in\SC$, $l\in\N$, with intersection pattern exactly given by $\tau_i$ intersects $\tau_{i+1}$ positively, at a point $t_{i,i+1}$ and $\tau_l$ intersects $\g_1$ negatively at $t_{l,1}$.
\end{enumerate}

\begin{center}
	\begin{figure}[H]
		\centering
		\includegraphics[scale=0.75]{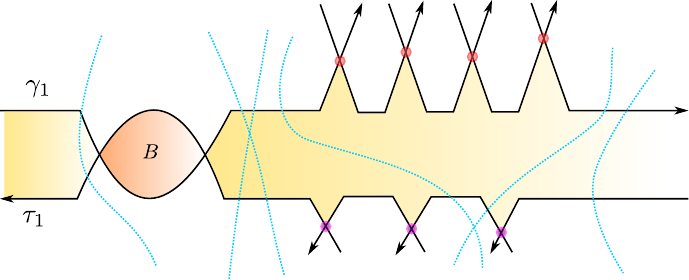}
		\caption{The configuration $\SC$ near the local bigon $B$ in Figure \ref{fig:Bigon_LeftRight} with curves from $\SC$ passing through. These curves passing through are drawn in blue and dashed.}
		\label{fig:Bigon_LeftRight15}
	\end{figure}
\end{center}

Figure \ref{fig:Bigon_LeftRight} depicts a case with $s=4$ and $l=3$. Note that some of these curves $\g_i$ and $\tau_j$ might be equal to each other. For instance a curve $\g_i$, resp.~some curve $\tau_i$, might be equal to some of the other curves $\g_j$, resp.~the other curves $\tau_j$, or a curve $\g_i$ might be equal to a curve $\tau_j$. The argument that now follows works in any of these cases. Note nevertheless that, due to the fact that $P_{\la,\rho}$ is oriented, it is not possible for $\g_i$ to equal $-\g_j$ or $-\tau_j$ for any $j$, and similarly for the $\tau_i$ curves.

\begin{center}
	\begin{figure}[H]
		\centering
		\includegraphics[scale=0.75]{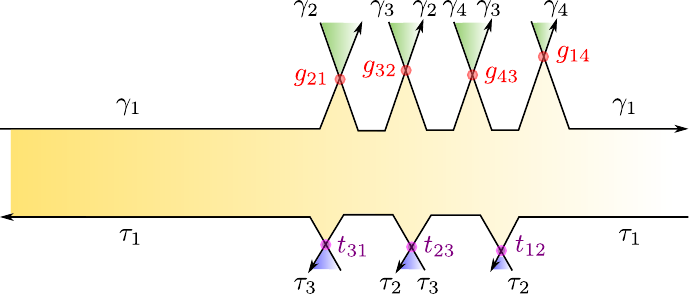}
		\caption{A curve configuration $\SC$ after performing a local bigon move to Figure \ref{fig:Bigon_LeftRight}. The resulting region $\la(B)=\rho(B)$ is not simply-connected.}
		\label{fig:Bigon_LeftRight2}
	\end{figure}
\end{center}

Now consider the resolution of the bigon $B$ with a local bigon move. This yields the local curve configuration depicted in Figure \ref{fig:Bigon_LeftRight2}. The arrows $g_{2,1},g_{3,2},\ldots,g_{i+1,i},\ldots,g_{s,s-1},g_{1,s}$ in the quiver $Q(\cC)$ form a cycle $G$ in $Q(\SC)$. Indeed, the arrow $g_{ij}$ goes from the vertex associated to $\gamma_i$ to that associated to $\gamma_j$ and thus the sequence of arrows $g_{2,1},g_{3,2},\ldots,g_{i+1,i},\ldots,g_{s,s-1},g_{1,s}$ in the quiver $Q(\cC)$ starts with an arrow from $\g_2$ to $\g_1$, then from $\g_3$ to $\g_2$ until eventually having an arrow from $\g_s$ to $\g_{s-1}$ and closing up the cycle with an arrow from $\g_{1}$ to $\g_s$. In particular, this cycle passes through the vertices $\gamma_1, \gamma_s,\dotsb, \gamma_2$ in $Q(\cC)$, in this cyclic order. Similarly, the arrows $t_{1,2}\ldots t_{i,i+1}\ldots t_{l,1}$ in the quiver $Q(\SC)$ form a cycle $T$ in $Q(\SC)$. Note that $G$ might not be irreducible if there are some coincides of the form $\g_i=\g_j$, and similarly for $T$. Here a cycle in a quiver is said to be irreducible if it passes through each of its vertices exactly once. In either case, we can express each of $G$ and $T$ as a sequence of irreducible cycles, which will be referred to as the irreducible components of $G$ and $T$.

\begin{center}
	\begin{figure}[H]
		\centering
		\includegraphics[scale=0.85]{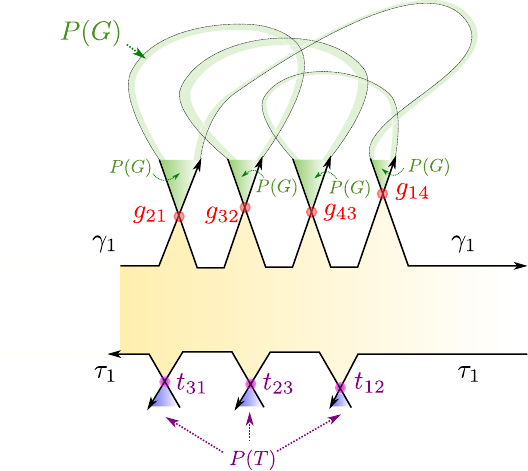}
		\caption{A depiction of the polygons $P(G)$, in green, and $P(T)$, in purple. The polygon $P(G)\sse\Sigma$ is the region bounded by the curves $\gamma_i$ above the intersection points $g_{i+1,i}$. The polygon $P(T)\sse\Sigma$ is the region bounded by the curves $\tau_j$ below the intersection points $t_{i+1,i}$. The figure schematically illustrates that the region $P(G)$ might go deep into the surface $\Sigma$, as might the curves $\gamma_i$. The apparent intersection points of those dotted parts of the $\gamma_i$ are not necessarily intersection points: they are just meant to indicate that the curves $\gamma_i$ might venture elsewhere in $\Sigma$ and the narrow green region is meant to indicate that $P(G)$ also goes along. The situation is analogous for $P(T)$.}
		\label{fig:Bigon_LeftRight3}
	\end{figure}
\end{center}

Consider the cycles $G$ and $T$ in the quiver. If an irreducible component of $G$ or $T$ is empty, then Lemma \ref{lem:empty_cycles_degeneracy} implies that $(Q(\SC),W(\SC))$ is degenerate. This contradicts the assumption. Therefore, all irreducible components of $G$ and $T$ must be non-empty. That means that there exists monomials in the potential given by the irreducible components of the cycle $G$, and similarly for $T$. In particular, there must exist a (possibly disconnected) polygon $P(G)\sse\Sigma$ with vertices $g_{2,1},g_{3,2},\ldots,g_{i+1,i},\ldots,g_{s,s-1},g_{1,s}$ and bounded by $\g_1,\ldots,\g_s$. The number of connected components of $P(G)$ is the number of irreducible cycles of $G$: just for this argument, we allow for polygons to be disconnected so as to ease notation. This polygon $P(G)$ is schematically depicted in green in Figures \ref{fig:Bigon_LeftRight2} and \ref{fig:Bigon_LeftRight3}, the latter figure specifically contains the notation $P(G)$. Similarly, there exists a polygon $P(T)\sse\Sigma$ with vertices $t_{1,2}\ldots t_{i,i+1}\ldots t_{l,1}$ and bounded by $\tau_1,\ldots,\tau_l$. This polygon $P(T)$ is schematically depicted in purple in Figures \ref{fig:Bigon_LeftRight2} and \ref{fig:Bigon_LeftRight3}

Note that from the specific configuration we are studying, as in Figures \ref{fig:Bigon_LeftRight2} and \ref{fig:Bigon_LeftRight3}, the polygon $P(G)$, locally near its vertices, must be right above the intersection points $g_{i+1,i}$. Similarly, the polygon $P(T)$, locally near its vertices, must be right below the intersection points $t_{j+1,j}$. A piece of notation: for an oriented crossing, we refer to its only oriented resolution as the $\infty$-resolution, in other words, the $\infty$-resolution is given by Figure \ref{fig:Bigon_LeftRight5}. For example, for the $g_{21}$ crossing in Figure \ref{fig:Bigon_LeftRight2}, its $\infty$-resolution is such that the strand coming from the north-west continues to the north-east strand and the strand coming from the south-west continues south-east. 

\begin{center}
	\begin{figure}[H]
		\centering
		\includegraphics[scale=0.75]{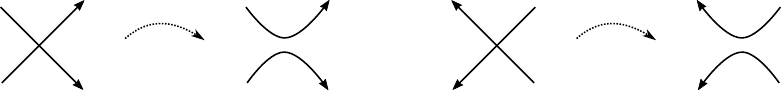}
		\caption{Two instances of an $\infty$-resolution at an intersection point.}
		\label{fig:Bigon_LeftRight5}
	\end{figure}
\end{center}

Consider now the smooth oriented representative $\overline{\g}$ of $[\g_1+\g_2+\ldots+\g_t]$ in $H_1(\Sigma;\Z)$ given by the oriented $\infty$-resolutions of the crossings $g_{21},g_{32},\ldots,g_{t,t-1},g_{1,t}$. That is, $\overline{\g}$ consists of two types of connected components: one connected component is the curve $\gamma_1'$ depicted in Figure \ref{fig:Bigon_LeftRight4} and the other type of connected components form the smooth boundary $\dd^{sm} P(G)$ of the polygon $P(G)$, after smoothing $P(G)$ at its vertices (using the $\infty$-resolution). Because the $\infty$-resolutions are oriented, all the resulting connected components are naturally oriented.

\begin{center}
	\begin{figure}[H]
		\centering
		\includegraphics[scale=0.75]{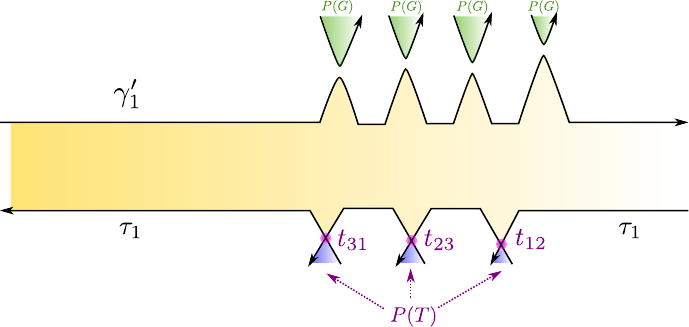}
		\caption{The situation in Figure \ref{fig:Bigon_LeftRight3} after applying an $\infty$-resolution to each of the crossings $g_{i+1,i}$. We emphasize that the polygon $P(G)$ and its smoothing are open embedded pieces of $\Sigma$. In particular, the boundary of its smoothing is a null-homologous curve in $\Sigma$.}
		\label{fig:Bigon_LeftRight4}
	\end{figure}
\end{center}

Similarly, let $\overline{\tau}$ be the smooth oriented representative of $[\tau_1+\tau_2+\ldots+\tau_l]$ in $H_1(\Sigma;\Z)$ given by the oriented $\infty$-resolutions of the crossings $t_{12},t_{23},\ldots,t_{l-1,l},t_{l,1}$. That is, $\overline{\tau}$ consists of two types of smooth connected components: one connected component is the curve $\tau_1'$ (analogous to $\g'_1$ above), and the other connected components come from the smooth boundary $\dd^{sm} P(T)$ of the polygon $P(T)$, after smoothing $P(T)$ at its vertices (again using the $\infty$-resolution).

Now, all the components of $\dd^{sm} P(G)$ and $\dd^{sm} P(T)$ are null-homologous, as $P(G)$ and $P(T)$ are embedded regions in $\Sigma$ bounding them. Independently, the yellow strip region in Figure \ref{fig:Bigon_LeftRight4} shows that $\g_1'$ and $-\tau_1'$ bound an embedded strip. Therefore $[\g_1']=-[\tau_1']$. As a consequence,
$$[\g_1+\g_2+\ldots+\g_t]=[\overline{\gamma}]=[\gamma_1']=-[\tau_1']=-[\overline{\tau}]=-[\tau_1+\tau_2+\ldots+\tau_l].$$
Therefore $[\g_1+\g_2+\ldots+\g_t+\tau_1+\tau_2+\ldots+\tau_l]=0$. The left hand side is a positive linear combination of $\g_i$ and $\tau_j$: as said above, some $\g_i$ might coincide with $\g_j$ or $\tau_j$, but not with $-\g_j$ or $-\tau_j$; similarly for the $\tau_i$ curves. As a consequence, the left hand side is a non-zero homology class and $[\g_1+\g_2+\ldots+\g_t+\tau_1+\tau_2+\ldots+\tau_l]=0$ is a non-trivial relation in homology. Thus the set of curves in $\SC$ does not span $H_1(\Sigma;\Z)$, as there are exactly $b_1(\Sigma)$ of them and there is a non-empty subset of linearly dependent classes.
\color{black}
\end{proof}

Proposition \ref{thm:QPnondeg} and Proposition \ref{prop:local_bigons} show that Assumption \ref{assumption:bigons} is always satisfied for the curve configurations $\SC(\bG)$ associated to plabic fences $\bG$.

\subsection{Proof of Lemma \ref{lem:reduced_parts}}\label{ssec:proof_reduction}
For Part (i), if $\SC$ bounds bigons, we obtain $\SC_{red}$ by iteratively applying Theorem \ref{thm:HassScott} to remove them. Lemmas \ref{lem:triplepoint} and \ref{lem:bigon} imply that $(Q(\SC),W(\SC))$ undergoes a sequence of right-equivalences and local reductions. Note that Lemma \ref{lem:bigon} implies that these reductions exist. By Definition \ref{def:localreduction}, each local reduction can be understood as splitting off a trivial direct summand from $(Q(\SC),W(\SC))$. Indeed, an $ab$-reduction of $(Q,W)$ yields a decomposition $(Q_{ab},W_{ab})\oplus(Q',W')$, where $Q_{ab}$ has just two arrows $a,b$ and $W_{ab}=(a-U)(b-V)$ for some polynomials $U, V$ determined by polygon counting in $\SC$. By considering the automorphism $\overline{a}:=a-U, \overline{b}:=b-V$ in the path algebra, this potential is equivalent to the trivial potential $W_{\overline{ a}\overline{ b}}=\overline{ a}\overline{ b}$. Thus $(Q_{ab},W_{ab})$ is trivial and $(Q',W')$ is the $ab$-reduction of $(Q,W)$. Therefore, after Theorem \ref{thm:HassScott} is iteratively applied until we obtain $\SC_{red}$, $(Q,W)$ undergoes a sequence of local reductions until it becomes right-equivalent to $(Q(\SC_{red}),W(\SC_{red}))$: in each such local reduction a trivial summand splits off and we are left with $(Q(\SC_{red}),W(\SC_{red}))$. Since $\SC_{red}$ has no bigons, the quadratic part of $W(\SC_{red})$ vanishes and thus $(Q(\SC_{red}),W(\SC_{red}))$ is reduced. Therefore, it is the reduced part of $(Q(\SC),W(\SC))$, up to right-equivalence and Lemma \ref{lem:reduced_parts}.(i) follows.

For Part (ii), Lemma \ref{lem:empty_cycles_degeneracy}, or directly Definition \ref{def:non_degenerate}, implies that $Q(\SC)$ has no empty 2-cycles, since $(Q(\SC),W(\SC))$ is non-degenerate. Therefore, if $\SC$ bounds no bigons, then $Q(\SC)$ has no 2-cycles and thus neither does its reduced part. 
\hfill$\Box$


\section{A Lagrangian filling for every cluster seed}\label{sec:Symplectic}

The goal of this section is to prove Theorem \ref{thm:main}. We use the results developed in Sections \ref{sec:curves} and \ref{sec:QP_PlabicFence}.


\subsection{Preliminaries}\label{ssec:prelim_symp} Let us consider a positive braid word $\beta$ on $n$-strands and its associated Legendrian link $\La_\beta\sse(\R^3,\xi_{st})$. By definition, $\La_\beta\sse(\R^3,\xi_{st})$ is the Legendrian link given by the front rainbow closure of $\beta$. Figure \ref{fig:FrontProjection1} depicts such a rainbow closure, i.e.~a front for $\La_\beta$ in the $(x,z)$-plane $\R^2$. The box with the label $\beta$ contains exactly the crossings of $\beta$. The ideal contact boundary $(T^\infty\R^2,\la_{st})$ of the cotangent bundle $(T^*\R^2,\la_{st})$ is contactomorphic to the 1-jet space $(J^1S^1,\xi_{st})$, where the Legendrian zero section $S^1\sse J^1S^1$ is the fiber of the projection $T^\infty\R^2\lr\R^2$ onto the base $\R^2$. Let us consider a Legendrian embedding $\iota_0:S^1\lr(\R^3,\xi_{st})$ of the (unique) max-tb Legendrian unknot $\La_0\sse(\R^3,\xi_{st})$.

\begin{center}
	\begin{figure}[H]
		\centering
		\includegraphics[scale=0.6]{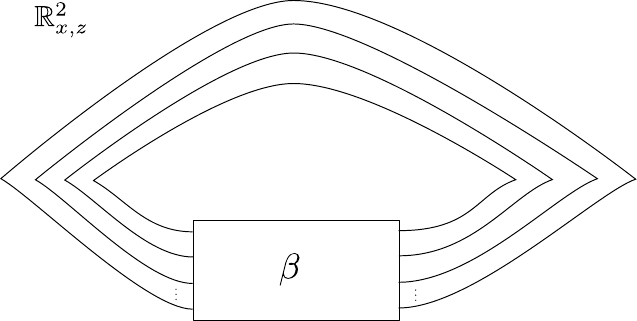}
		\caption{The rainbow closure front projection of the positive braid word $\beta$.}
		\label{fig:FrontProjection1}
	\end{figure}
\end{center}

By the Weinstein neighborhood theorem \cite[Section 2.5]{Geiges08}, any Legendrian link $\La\sse(J^1S^1,\xi_{st})$ can be satellited along the Legendrian embedding $\iota_0$, as there is a neighborhood of the max-tb Legendrian unknot $\La_0\sse(\R^3,\xi_{st})$ contactomorphic to $(J^1S^1,\xi_{st})$, where the contactomorphism extends $\iota_0$. We denote the resulting Legendrian link in $(\R^3,\xi_{st})$ by $\iota_0(\La)$. If we consider the Legendrian $\La(\alpha)\sse(J^1S^1,\xi_{st})$ given by the braid diagram of a positive braid word $\alpha$ in $n$-strands, then $\iota_0(\La(\beta\Delta^2))\cong \La_\beta$ are Legendrian isotopic in $(\R^3,\xi_{st})$, where $\Delta$ is the half-twist on $n$ strands. See \cite[Section 2.2]{CasalsNg} for further details. Note that the class of Legendrian links $\La_\beta$ includes the max-tb representatives of all algebraic links and also Legendrian representatives of infinitely many satellite and hyperbolic links, cf.~\cite[Section 6]{CasalsHonghao}.

Let $L\sse(\R^4,\la_{st})$ be an exact oriented Lagrangian filling.\footnote{Of some Legendrian link $\dd L$ in the ideal contact boundary of $(\R^4,\la_{st})$.} Consider an arbitrary but fixed convex neighborhood $\Op(L)$ which is symplectomorphic to a convex neighborhood of the zero section in $(T^*L,\la_{st})$ and the hypersurface $\dd\Op(L)\sse(\R^4,\la_{st})$ is a contact hypersurface contactomorphic to the ideal contact boundary $(T^\infty L,\la_{st})$ of $(T^*L,\la_{st})$. An embedded (co)oriented connected curve $\g\sse L$ lifts to a Legendrian knot $\La_\g\sse\dd\Op(L)$, as it defines a front under the Legendrian projection $(T^\infty L,\la_{st})\lr L$. Since there is a canonical correspondence between oriented and co-oriented curves in an oriented surface, we always discuss oriented curves rather than co-oriented curves. Following \cite{CasalsWeng22}, we introduce the following:

\begin{definition}[$\L$-compressing systems]\label{def:compressible}
Let $L\sse(\R^4,\la_{st})$ be an exact oriented Lagrangian filling and $\g\sse L$ an embedded oriented curve. By definition, $\g$ is said to be $\bL$-compressible if there exists a properly embedded Lagrangian 2-disk $D\sse (T^*\R^2\setminus \Op(L))$ such that $\dd\overline{D}\cap \dd\Op(L)=\La_\g\sse\R^4$ and the union of $\overline{D}\cup \nu_\gamma$ is a smooth Lagrangian disk, where $\nu_\gamma\sse\Op(L)$ is the Lagrangian conormal cone of $\gamma$. This Lagrangian disk is said to be an $\mathbb{L}$-compressing disk for $\g$.

A collection $\Gamma=\{\gamma_1,\ldots,\gamma_b\}$ of such curves in $L$, with a choice of $\mathbb{L}$-compressing disks $\sD=\{D_1,\ldots,D_b\}$ for each curve, is said to be an $\L$-compressing system for $L$ if $D_i\cap D_j=\emptyset$ for all $i,j\in[b]$ and the (homology classes of the) curves in $\Gamma$ form a basis of $H_1(L;\Z)$.\hfill$\Box$
\end{definition}

We ease notation and refer to the collection $\sD$ as the $\bL$-compressing system for $L$. A piece of notation:



\subsection{Curve configurations and $\bL$-compressing systems} Let $L\sse(\R^4,\la_{st})$ be a Lagrangian filling and $\Gamma$ a $\bL$-compressing system.

\begin{definition}
By definition, the curve configuration $\SC(\Gamma)$ associated to $\Gamma$ is the configuration of oriented closed embedded curves $\Gamma$ in $L$. If $\sD$ is the collection of $\bL$-compressing disks associated to $\Gamma$, we also write $\SC(\sD)$ for $\SC(\Gamma)$.\hfill$\Box$
\end{definition}

Note that $\SC(\Gamma)$ is considered as a collection of smooth oriented curves in a smooth surface, with no need to record the Lagrangian condition on $L$ and the disks in $\sD$. The notation $\SC(\Gamma)$, instead of just $\Gamma$, is in order to emphasize the smooth embedded curves, rather than the symplectic topological aspects of $\Gamma$.

\begin{remark}
Suppose that a sequence of triple point moves and local bigon moves is applied to an $\bL$-compressing system $\SC(\Gamma)$. This yields a configuration $\SC'$ in $L$. Front homotopies, which include triple point moves and local bigon moves, lift to Legendrian isotopies in the ideal contact boundary. The trace of a Legendrian isotopy yields an invertible Lagrangian concordance in the symplectization. By concatenating the disks associated to $\Gamma$ with this Lagrangian concordance, we obtain an $\bL$-compressing system $\Gamma'$ for $L$ such that $\SC'=\SC(\Gamma')$. We will consider two such configurations $\SC,\SC'$ equivalent and two such $\bL$-compressing systems $\Gamma,\Gamma'$ equivalent.\hfill$\Box$
\end{remark}


\subsection{$\bL$-compressing systems for $\La_\beta$}\label{ssec:compressing_systems}

Consider the plabic fence $\bG(\beta)$ associated to $\beta$, as introduced in Section \ref{sec:QP_PlabicFence}. Then we have the following facts:

\begin{enumerate}
    \item The alternating strand diagram of $\bG(\beta)$ is a front for $\La(\beta\Delta^2)\sse(J^1S^2,\xi_{st})$. This is proven in \cite[Section 2]{CasalsWeng22} and see also \cite[Section 5.1]{STWZ}. Thus, after including $(J^1S^1,\xi_{st})$ into $(\R^3,\xi_{st})$ as a neighborhood of the max-tb Legendrian unknot, it is Legendrian isotopic to the Legendrian link $\La_\beta\sse(\R^3,\xi_{st})$.\\

    \item From the conjugate surface $\Sigma(\bG(\beta))$ one can construct an embedded exact Lagrangian filling $L_\beta$ of $\La_\beta$, defined up to compactly supported Hamiltonian isotopy in $(T^*\R^2,\la_{st})$. By construction, $L_\beta$ is smoothly isotopic to the (smooth) surface $\Sigma(\bG(\beta))$. The construction of $L_\beta$ is done in \cite[Section 4.2]{STWZ}, and \cite[Prop.~2.4]{CasalsLi22} shows that the resulting Hamiltonian isotopy class is independent of the choices in the construction. In particular, it gives an oriented embedded exact Lagrangian filling $L_\beta\sse(\R^4,\la_{st})$ in the symplectization of $(\R^3,\xi_{st})$, after we have identified the standard cotangent bundle $(T^*\R^2,\omega_{st})$ with the symplectic Darboux $(\R^4,\omega_{st})$.\\

    \item The conjugate surface $\Sigma(\bG(\beta))$ also gives a $\bL$-compressing system $\Gamma(\beta)$ for $L$. The existence of such $\bL$-compressing system is proven in \cite[Section 3]{CasalsWeng22}, see also \cite[Section 2]{CasalsLagSkel} and \cite[Section 4.2]{CasalsLi22}. By construction, the configuration of curves in $\SC(\Gamma(\beta))$ coincides with the configuration of curves in $\SC(\bG(\beta))$. The Lagrangian $\L$-compressing disks, showing that this is indeed an $\L$-compressing system, are constructed in \cite[Section 3]{CasalsWeng22}.\footnote{Intuitively, the corresponding $\bL$-compressing Lagrangian disks are the faces of $\bG$, which are disjoint Lagrangian pieces (disks) of the Lagrangian zero section $\R^2$ in $T^*\R^2$. See Figure \ref{fig:PlabicFence_AddingCrossing3} (left).} Let $\sD_\beta$ denote its associated collection of $\bL$-compressing Lagrangian disks.
\end{enumerate}

The relevant fact about the $\bL$-compressing system $\Gamma(\beta)$ is that it can be used to produce new Lagrangian fillings from $L_\beta$.

\begin{remark}
For context, let $\beta$ be a positive braid word and $\La_\beta\sse(\R^3,\xi_{st})$ its associated Legendrian link. The union $\L_\beta\sse(\R^4,\la_{st})$ of the embedded exact Lagrangian filling $L_\beta$ and all the closures of the Lagrangian disks of the collection $\mathscr{D}_\beta$ (extended by the Lagrangian conormal cones of curves in $\Gamma(\beta)$) is an arboreal Lagrangian skeleton for the Weinstein pair given by $(\R^4,\la_{st})$ and a Weinstein ribbon of $\La_\beta$, see \cite[Section 2]{Eliashberg18_Revisited}.\hfill$\Box$
\end{remark}

\subsection{Lagrangian disk surgery} Lagrangian disk surgery was introduced in \cite[Section 2.3]{MLYau17}, following closely the Lagrange surgery defined in \cite{Polterovich_Surgery}. A more recent account reviewing Lagrangian surgery is \cite[Section 6.2]{CasalsMurphyPresas}, and see also \cite[Proposition 5.15]{STWZ}, \cite[Theorem 1.5]{STW} and \cite[Section 4.5]{PascaleffTonkonog20}. In our context, it is used as follows. Consider a Legendrian link $\La\sse(\R^3,\xi_{st})$, seen as $(S^3,\xi_{st})$ minus a point, and an embedded exact Lagrangian filling $L\sse(\R^4,\la_{st})$ in the standard Darboux 4-ball symplectic filling $(S^3,\xi_{st})$. Suppose that there exists a properly embedded Lagrangian disk $D\sse\R^4\setminus L$ such that $\dd\overline{D}\sse \mbox{int}(L)$ is a smooth embedded connected curve, where $\mbox{int}(L)$ in the interior of $L$.\footnote{In the discussion of Subsection \ref{ssec:compressing_systems}, these Lagrangian disks are obtained by considering the union of the disks $D_\g\in\sD$ in an $\L$-compressing system $\sD$ and concatenating them with (a piece of) the Lagrangian conormal cone in $T^*L$ of the corresponding (co)oriented curve $\g$.} Then Lagrangian disk surgery is an operation that inputs the pair $(L,D)$ and outputs another pair $(L',D')$, with the same properties: $L'$ is an embedded exact Lagrangian filling of $\La$ and $D'$ is an embedded Lagrangian disk in the complement of $L'$ with embedded boundary on $L'$. In this process, it is crucial that $D$ is an embedded Lagrangian disk, and not just immersed. Two facts about Lagrangian disk surgery are:

\begin{itemize}
    \item[-] The Lagrangians $L$ and $L'$ are not necessarily exact Lagrangian isotopic. That is, there might not be a compactly supported Hamiltonian isotopy from $L$ to $L'$.\\

    \item[-] The Lagrangians $L$ and $L'$ are smoothly isotopic, relative to their boundaries. The Lagrangian disk surgery of $L'$ along $D'$ yields $(L,D)$ back, up to compactly supported Hamiltonian isotopy.
\end{itemize}

The first item above precisely indicates that we can potentially produce a new Lagrangian filling by using a given Lagrangian filling and a Lagrangian disk as above. See \cite{CasalsWeng22,Polterovich_Surgery,MLYau17} and references therein for these facts and more details.

\begin{remark}\label{rmk:immersed} Lagrangian disk surgery is {\it not} known to exist if the boundary $\dd\overline{D}\sse L$ is an immersed curve, rather than embedded. Similarly, the disk $D$ must be embedded. In general, it is not just a lack of available constructions, as \cite[Section 4.10]{CasalsWeng22} presents examples of immersed disks that one cannot perform Lagrangian disk surgery to, due to the existence of frozen vertices coming from absolute 1-cycles in $L$.\hfill$\Box$
\end{remark}


\subsection{Effect of Lagrangian surgery on curve configurations}\label{ssec:effectLag} Let $L\sse(\R^4,\la_{st})$ be an exact Lagrangian filling and $\Gamma$ an $\bL$-compressible system for $L$. Consider a disk $D\in\sD$. Lagrangian disk surgery on $D$ leads to another exact Lagrangian filling $\mu_D(L)\sse(\R^4,\la_{st})$ endowed with a curve configuration $\mu_D(\Gamma)$ and a collection of Lagrangian disks $\mu_D(\sD)$ bounding the curves in $\mu_D(\Gamma)$. There is a natural bijection between the disks in $\sD$ and those in $\mu_D(\sD)$ and a diffeomorphism between $L$ and $\mu_D(L)$, as stated above. Now, the configuration $\mu_D(\Gamma)$ might not be an $\bL$-compressible system because it might contain immersed curves.

In general, these new curve configurations $\mu_D(\Gamma)$ obtained by Lagrangian disk surgery on a disk associated to $\Gamma$ can be understood via the following:

\begin{lemma}\label{lem:configuration_Lagrangiansurgery} Let $L\sse(\R^4,\la_{st})$ be an exact Lagrangian filling, $\Gamma$ an $\bL$-compressible system for $L$ with $\L$-compressing disks $\sD$, and $\SC(\Gamma)$ its associated configuration in $L$. Consider a disk $D\in\sD$ with boundary the lift of $\g\in\SC(\Gamma)$. Then there is a natural identification between the configuration $\mu_D(\Gamma)$ and the $\g$-exchange of $\Gamma$.
\end{lemma}

\begin{proof}
Lagrangian disk surgery occurs in a neighborhood of $D$ in $(\R^4,\la_{st})$. It is shown in \cite[Section 4.8]{CasalsZaslow} that it can be locally modeled by the weave mutation in Figure \ref{fig:MutationWeave}. For this proof, we assume familiarity with \cite[Section 2]{CasalsZaslow}, or \cite[Section 3]{CasalsWeng22}. It suffices to understand how Lagrangian disk surgery along the disk $D$ bounding $\g$ affects the boundary of the other disks in $\sD$. There are two cases: positive and negative intersections, represented by the segments $\tau_+$ and $\tau_-$ in Figure \ref{fig:CExchange_TwoBefore} (left). 

\begin{center}
	\begin{figure}[H]
		\centering
		\includegraphics[scale=0.7]{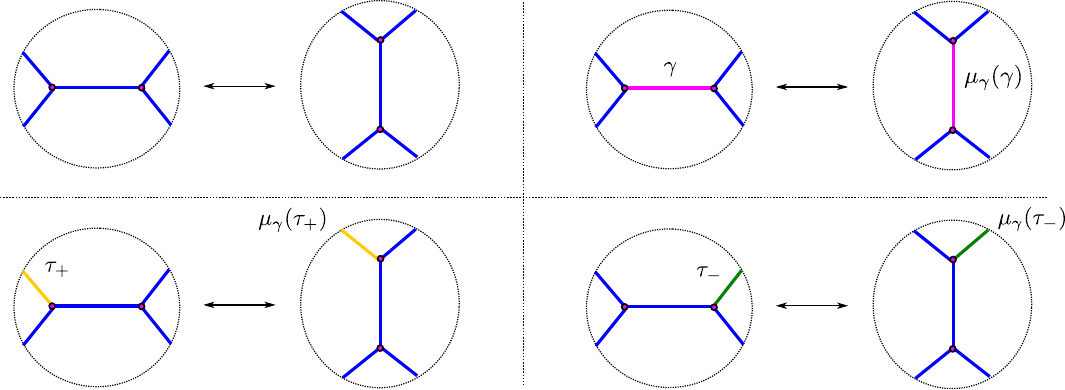}
		\caption{(Upper left) Lagrangian disk surgery in the 2-weave model along a short $\sf I$-cycle $\g$. (Upper right) The short $\sf I$-cycle $\g$ depicted in pink, drawn before and after $\mu_\g$. (Lower left) A weave line in yellow smoothly representing a relative homology class $\tau_+$ with unique geometric intersection $+1$ with $\g$, drawn before and after $\mu_\g$. (Lower right) A weave line in green smoothly representing a relative homology class $\tau_-$ with a unique geometric intersection $-1$ with $\g$, also drawn before and after $\mu_\g$.}\label{fig:MutationWeave}
	\end{figure}
\end{center}

In the weave diagram, the segments $\tau_+$ and $\tau_-$ in Figure \ref{fig:CExchange_TwoBefore} (left) can be represented by the weave lines $\tau_+$ and $\tau_-$ in the second row of Figure \ref{fig:MutationWeave}.\footnote{The segment $\tau_-$ remains green, while $\tau_+$ is now drawn in yellow because the weave is typically drawn in blue.} In this situation, before the mutation, $\tau_+\cap\gamma=+1$, $\tau_-\cap\gamma=-1$ and $\tau_+\cap\tau_-=0$. Here it suffices to record homological intersections between these curves, as the smooth curves in this case are uniquely determined (up to isotopy) by their relative homology classes. The 2-weave in the left of Figure \ref{fig:MutationWeave} (upper left) represents a Lagrangian cylinder, and the left parts of the other three figures indicate how to draw $\tau_\pm$ and $\gamma$ in that cylinder.

It suffices to understand how $\tau_+,\tau_-$ and $\g$ in the weave change under weave mutation. This is drawn in the right parts of the upper right and the second row of Figure \ref{fig:MutationWeave}. The resulting curves $\mu_\g(\tau_+)$ and $\mu_\g(\tau_-)$ are shown, also in yellow and green respectively; the curve $\mu_\g(\g)$ is drawn in pink. By using the intersection numbers on weaves, cf.~\cite[Section 2]{CasalsZaslow} or \cite[Section 4.4]{CGGLSS}, we obtain that $\mu_\g(\tau_+)\cap\mu_\g(\g)=-1$, $\mu_\g(\g)\cap\mu_\g(\tau_-)=-1$ and $\mu_\g(\tau_+)\cap \mu_\g(\tau_-)=1$. Therefore the curves change exactly according to a $\g$-exchange.
\end{proof}

\begin{remark}
Lemma \ref{lem:configuration_Lagrangiansurgery} could be proven using other models, such as plabic fences and Lagrangian conjugate surfaces, cf.~\cite[Section 3]{CasalsLi22} or \cite[Section 5.2]{STWZ}. We also refer the reader to \cite[Section 2]{STW} for an explanation using a conical model and the discussion in \cite[Section 4]{PascaleffTonkonog20}.\hfill$\Box$
\end{remark}
\color{black}
Lemma \ref{lem:configuration_Lagrangiansurgery} clarifies the combinatorics of an $\L$-compressing system $\Gamma$ that lead to immersed curves for $\mu_D(\Gamma)$, as we momentarily explain.

\begin{remark}\label{rmk:quiver_mut} Before discussing that, let us point that we will now use $\mu_vQ$ to refer to mutation of a quiver $Q$ at a vertex $v$ without removing 2-cycles (so $Q$ is allowed to have 2-cycles and $v$ be a vertex of a 2-cycle) and leading to quivers with loops. In the original context of quiver mutations, cf.~\cite[Definition 4.2]{FominZelevinsky_ClusterI} or see \cite[Definition 2.1.2]{FWZ}, 2-cycles are removed by default and there are no loops. The 2-cycle removal is step (3) in \cite[Definition 2.1.2]{FWZ}. (In fact, some authors in the literature do not allow a quiver to have 2-cycles or loops.)

In the definition of QP-mutation, \cite[Definition 5.5]{DWZ}, one allows for 2-cycles but takes the reduced part as part of the definition of QP-mutation. (Intuitively, 2-cycles that are accounted for in the potential thus disappear after QP-mutation.) Note that in that context, see \cite[Definition 4.1]{DWZ}, the quiver is not allowed to have loops. For the upcoming discussion, we allow 2-cycles, loops and use the corresponding general form of mutation. Specifically, if $Q$ has a 2-cycle one of whose vertices is $v$, then we denote by $\mu_vQ$ the quiver mutation given by steps (1) and (2) of \cite[Definition 2.1.2]{FWZ} (but not step (3)), where loops might be created. That is, if we have two arrows $a,b\in Q_1$, $h(a)=t(b)=v$, in a 2-cycle, then mutation at $v$ leads to the composed arrow $[ab]\in (\mu_vQ)_1$, which is a loop in $\mu_vQ$.\hfill$\Box$
\end{remark}

Now,  Proposition \ref{prop:quivermutation} implies that $\g$-exchange leads to a quiver mutation, i.e.~ $\mu_\g Q(\Gamma)=Q(\mu_{D_\g}(\Gamma))$. We have the following two facts:

\begin{enumerate}
\item By construction, the curve configuration $\mu_D(\Gamma)$ has an immersed curve if and only if the quiver $Q(\mu_D(\Gamma))$ contains a loop.\\

\item  If a quiver $Q$ has no loops, then the quiver $\mu_vQ$ has a loop if and only if $v$ was part of a 2-cycle in $Q$. (This is by the definition of quiver mutation, cf.~Remark \ref{rmk:quiver_mut}.)
\end{enumerate}

Facts (1) and (2) above imply:

\begin{lemma}\label{lem:2cycleimmersed} Let $\Gamma$ be an $\L$-compressing system. Then $Q(\Gamma)$ has a $2$-cycle if and only if $\mu_D(\Gamma)$ has an immersed curve.
\end{lemma}

\begin{proof}
By fact (1), $\mu_{D_\g}(\Gamma)$ has an immersed curve if and only if the quiver $Q(\mu_{D_\g}(\Gamma))$ contains a loop. By fact (2) and Proposition \ref{prop:quivermutation}, $Q(\mu_{D_\g}(\Gamma))=\mu_\g Q(\Gamma)$ has a loop if and only if (the vertex associated to) $\g$ was part of a 2-cycle in $Q(\Gamma)$, and so $Q(\Gamma)$ has a 2-cycle.
\end{proof}

\begin{center}
	\begin{figure}[H]
		\centering
		\includegraphics[scale=0.5]{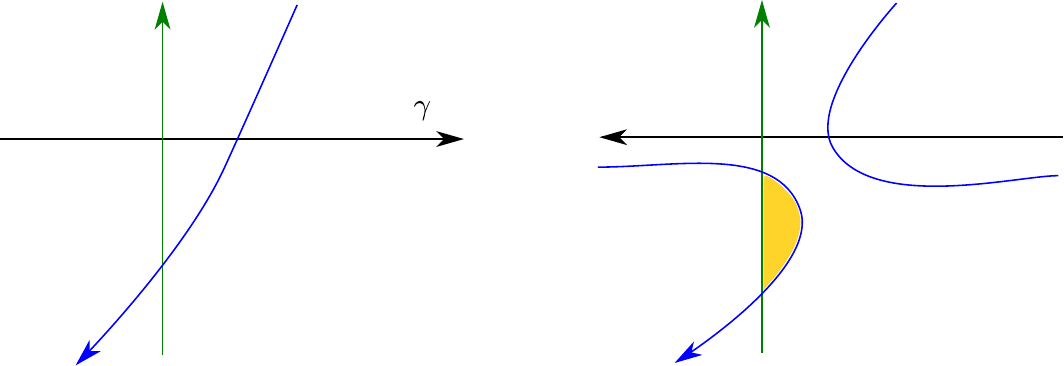}
		\caption{(Left) Three curves in a curve configuration $\SC$, depicted in blue, green and black. (Right) The result of applying a $\g$-exchange to $\SC$: note that, in this case, a (local) bigon is bounded by the image of the green and blue curves under the $\g$-exchange.}
		\label{fig:Example_Immersed}
	\end{figure}
\end{center}
\color{black}
\begin{example}\label{ex:Example_Immersed}
Figure \ref{fig:Example_Immersed} illustrates the effect of a $\g$-exchange applied to the curves depicted in green and blue in Figure \ref{fig:Example_Immersed} (left). The result is depicted in Figure \ref{fig:Example_Immersed} (right). Note that this $\g$-exchange will create a bigon, which is drawn in yellow in Figure \ref{fig:Example_Immersed} (right). The resulting configuration is non-reduced, bounding bigons. In this particular case, it is immediate how to reduce it, as we can apply a local bigon move to the yellow bigon. In general, Theorem \ref{thm:HassScott} is used.

Figure \ref{fig:Example_Immersed2} illustrates the effect of a $\g$-exchange (along the black curve $\g$) applied to the blue curve drawn in Figure \ref{fig:Example_Immersed2} (left). The result is depicted in Figure \ref{fig:Example_Immersed2} (right). Note that this $\g$-exchange will create an immersed (blue) curve even if the original blue curve was embedded. The cause for this immersed point (emphasized in orange) is the bigon bounded by the local blue curve and $\g$ in Figure \ref{fig:Example_Immersed2} (left).

As a consequence, it is not immediate that $\L$-compressing systems are readily carried through $\g$-exchanges. For instance, one of the curves might become immersed, as described in Figure \ref{fig:Example_Immersed2}. Another instance is that two curves that did not bound a bigon will bound a bigon after the $\g$-exchange, as described in Figure \ref{fig:Example_Immersed}. In that latter case, a second exchange at either of these two curves will result in the other curve turning into an immersed curve.\hfill$\Box$
\end{example}

\begin{center}
	\begin{figure}[H]
		\centering
		\includegraphics[scale=0.5]{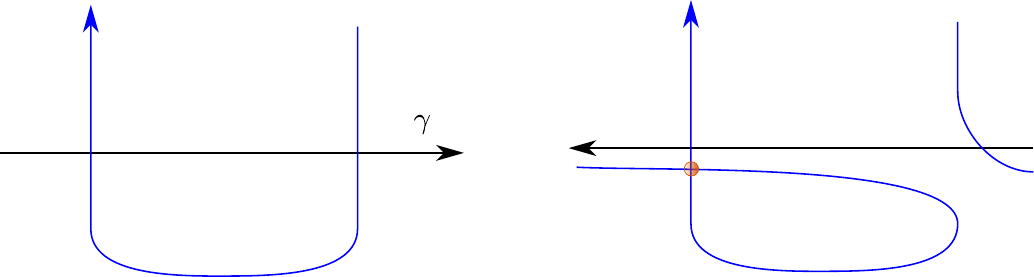}
		\caption{(Left) Two curves in a curve configuration $\SC$, depicted in blue and black. (Right) The result of applying a $\g$-exchange to $\SC$: note that, in this case, an immersed point exists in the image of blue curve under the $\g$-exchange. That is, even if the blue curve (on the left) was originally embedded, the result of a $\g$-exchange might turn it into an immersed curve.}
		\label{fig:Example_Immersed2}
	\end{figure}
\end{center}


\subsection{Iteration of Lagrangian disk surgeries and QP non-degeneracy}\label{ssec:iteration}

Let $L\sse(\R^4,\la_{st})$ be an embedded exact Lagrangian filling and $\Gamma$ an $\L$-compressing system for $L$. Suppose that the configuration $\SC(\Gamma)$ is reduced and $(Q(\SC(\Gamma)),W(\SC(\Gamma)))$ is non-degenerate. In particular, there are no 2-cycles in $Q(\SC(\Gamma))$.


\subsubsection{The problem that arises with iterations} Consider an $\L$-compressing disk $D\in\sD(\Gamma)$ with boundary $\g\in\SC(\Gamma)$ and perform Lagrangian disk surgery on $(L,D)$. This produces a new Lagrangian filling $(L',D')$ with an $\L$-compressing disk $D'$. The $\L$-compressing system $\Gamma$ for $L$ yields an $\L$-compressing system $\Gamma'$ for $L'$, whose configuration of curves $\SC(\Gamma')$ is as described by Lemma \ref{lem:configuration_Lagrangiansurgery}. We wish to be able to iterate this procedure arbitrarily: given any curve $\g_1\in\SC(\Gamma')$ with associated $\L$-compressing disk $D_1$, we want to be able to perform Lagrangian disk surgery to $L'$ along $D_1$ and obtain an $\L$-compressing system $\Gamma''$ for the new Lagrangian filling $L''=\mu_{D_1}(L')$. (And so on.)

The problem with iterating is that the $\g$-exchange of some curve configuration $\SC_{bad}(\Gamma)$, appearing at some point in the iteration, might be a ``curve configuration'' $\mu_\g(\SC_{bad}(\Gamma))$ with an immersed curve. See Example \ref{ex:Example_Immersed} above. Here we write ``curve configuration'' in quotations to mean a curve configuration as in Definition \ref{def:configuration} but where we relax the requirement that the curves are embedded to merely being immersed. The problem is that $\mu_\g(\SC_{bad}(\Gamma))$ will not admit an $\L$-compressing system and we cannot keep iterating arbitrarily. For instance, it would not be possible to perform a Lagrangian disk surgery along that immersed curve.

\begin{remark} Because of the properties of the cluster algebra constructed in \cite{GSW} (see also \cite{CasalsWeng22,CGGLSS}), and the arguments in this article (some already present in \cite{STW} indeed), the appearance of immersed curves is the only problem. To be precise, if we knew in advance that no immersed curves will appear in this iteration process, then the procedure presented above -- as already discussed in \cite{STW} -- and \cite[Theorem 1.1]{CasalsWeng22} would suffice to prove Theorem \ref{thm:main}. In that sense, immersed curves are the only problem.\hfill$\Box$
\end{remark}

By Lemma \ref{lem:2cycleimmersed}, the problem of an immersed curve appearing in some $\g$-exchange $\mu_\g(\SC(\Gamma))$ is equivalent to the quiver $Q(\SC(\Gamma))$ having a 2-cycle containing the vertex for $\g$. There are two situations in which $Q(\SC(\Gamma))$ might have a 2-cycle:
\begin{itemize}
    \item[(i)] It might be that $\SC(\Gamma)$ does not bound any bigons but there are curves $\g_i,\g_j\in\SC(\sD_1)$ with two points of geometric intersection, one positive and one negative, and these two intersection points do not bound a bigon. In this case the quiver $Q(\SC(\Gamma))$ has a 2-cycle but the potential $W(\SC(\Gamma))$ would not have a monomial accounting for it (as there is no bigon).\\

    \item[(ii)] It might be that $\SC(\Gamma)$ is a non-reduced configuration, bounding bigons. That is, there are curves $\g_i,\g_j\in\SC(\sD_1)$ with two points of geometric intersection, one positive and one negative, and these two intersection points bound a bigon. In this case the quiver $Q(\SC(\Gamma))$ has a 2-cycle and the potential $W(\SC(\Gamma))$ has a monomial accounting for it, as there is a bigon.
\end{itemize}

The main contribution of this paper is a solution to this problem which uses the curve potential constructed in Sections \ref{sec:curves} and \ref{sec:QP_PlabicFence}. By construction, $W(\SC(\beta))$ keeps track of polygons and, in particular, bigons. By Proposition \ref{thm:QPnondeg}, $(Q(\SC(\Gamma)),W(\SC(\Gamma)))$ is non-degenerate if $\Gamma$ is an $\L$-compressing system of the form $\Gamma=\Gamma(\beta)$. Non-degeneracy of the potential is then used to exclude possibility $(i)$ above from happening. That is, non-degeneracy ensures that whenever we have a 2-cycle between two vertices in the quiver we must also have a bigon bounded by the corresponding curves in the configuration. The reduction process, using Theorem \ref{thm:HassScott}, and the properties of the potential proven in Section \ref{sec:curves} are used to deal with possibility $(ii)$. Let us explain this in detail.
\color{black}

\color{black}


\subsubsection{$\L$-compressing systems and QPs under Lagrangian surgery} Consider the curve quiver with potential $(Q(\SC(\beta)),W(\SC(\beta)))$ constructed in Section \ref{sec:QP_PlabicFence}. We have the following two properties:

\begin{itemize}
     \item[(P1)] The QP $(Q(\SC(\beta)),W(\SC(\beta)))$ is reduced, as defined in Section \ref{sssec:DWZ}. Indeed, the construction of the QP $(Q(\bG),W(\bG))$ in Section \ref{ssec:QP_plabicfence} implies that is reduced for any plabic fence $\bG$. By Proposition \ref{lem:QP_coincide}, $(Q(\SC(\beta)),W(\SC(\beta)))$ coincides with $(Q(\bG(\beta)),W(\bG(\beta)))$, where $\bG(\beta)$ is the plabic fence associated to $\beta$. Thus $(Q(\SC(\beta)),W(\SC(\beta)))$ is reduced as well. \\

    \item[(P2)] The QP $(Q(\SC(\beta)),W(\SC(\beta)))$ is a curve QP, as in Definition \ref{def:quiverC}. By Proposition \ref{thm:QPnondeg}, it is a non-degenerate QP. In particular, there are never 2-cycles in any QP $(Q,W)$ which is related to  $(Q(\SC(\beta)),W(\SC(\beta)))$ by a sequence of QP mutations.
\end{itemize}

Let us use these properties, along with the results in Sections \ref{sec:curves} and \ref{sec:QP_PlabicFence}, and prove the following result.

\begin{prop}\label{prop:iterative_step}
Let $L\sse(\R^4,\la_{st})$ be an embedded exact Lagrangian filling and $\Gamma$ an $\L$-compressing system for $L$. Suppose that the configuration $\SC(\Gamma)$ is reduced and $(Q(\SC(\Gamma)),W(\SC(\Gamma)))$ is non-degenerate.

If $D\in\sD(\Gamma)$ is an $\L$-compressing disk with boundary $\g\in\SC(\Gamma)$ then:

\begin{itemize}
    \item[(i)] The Lagrangian filling $\mu_D(L)$ obtained from $L$ by Lagrangian disk surgery on $D$ admits an $\L$-compressing system $\mu_D(\Gamma)$ such that its associated QP $(Q(\mu_D(\Gamma)),W(\mu_D(\Gamma)))$ is reduced.\\

    \item[(ii)] The quiver with potential associated to $\mu_D(\Gamma)$ satisfies
    $$(Q(\SC(\mu_D(\Gamma))),W(\SC(\mu_D(\Gamma)))=(\mu_\g Q(\SC(\Gamma)),\mu_\g W(\SC(\Gamma)))$$
    and the quiver $Q(\SC(\mu_D(\Gamma)))$ contains no 2-cycles.
\end{itemize}
In addition, the curve configuration associated to $\mu_D(\Gamma)$ is a reduction of the $\g$-exchange of $\SC(\Gamma)$.

\end{prop}

\begin{proof} Let us prove $(i)$. First, the quiver $Q(\SC(\Gamma))$ has no 2-cycles. Indeed, $\SC(\Gamma)$ is reduced by hypothesis, and thus there are no 2-cycles coming from the trivial part of $(Q(\SC(\Gamma)),W(\SC(\Gamma)))$. By hypothesis as well, $(Q(\SC(\Gamma)),W(\SC(\Gamma)))$ is non-degenerate, and thus Definition \ref{def:non_degenerate} implies that there are no empty 2-cycles either. By Lemma \ref{lem:configuration_Lagrangiansurgery}, the configuration of curves $\SC(\Gamma)$ undergoes a $\g$-exchange under Lagrangian disk surgery. Since $Q(\SC(\Gamma))$ has no 2-cycles, all curves in the resulting configuration are embedded. None of them is immersed. Therefore, Lagrangian disk surgery along $D$ produces a new $\L$-compressing system $\Gamma'$, such that $\SC(\Gamma')$ is the $\g$-exchange of $\SC(\Gamma)$.

Second, at this stage, $\SC(\Gamma')$ might not be a priori reduced. By Theorem \ref{thm:HassScott}, there exists a reduction $\mu_D(\Gamma)$ of $\SC(\Gamma')$. Since  $(Q(\SC(\Gamma)),W(\SC(\Gamma)))$ is non-degenerate, Lemma \ref{lem:reduced_parts} applies and shows that, up to right equivalence,
$$(Q(\mu_D(\Gamma)),W(\mu_D(\Gamma)))=(Q(\Gamma)_{red},W(\Gamma)_{red}).$$
This implies that $(Q(\mu_D(\Gamma)),W(\mu_D(\Gamma)))$ is reduced. This implies $(i)$. By construction, $\mu_D(\Gamma)$ is a reduction of $\SC(\Gamma')$, which itself is the $\g$-exchange of $\SC(\Gamma)$. The last sentence of the statement is thus proven as well.

For item $(ii)$, Proposition \ref{prop:quivermutation} implies the equality between the QPs. By item $(i)$, the quiver with potential $(Q(\SC(\mu_D(\Gamma))),W(\SC(\mu_D(\Gamma)))$ is reduced and thus the non-degeneracy of $(Q(\SC(\Gamma)),W(\SC(\Gamma)))$ implies that  $Q(\SC(\mu_D(\Gamma)))$ has no 2-cycles.
\end{proof}

\subsection{The ring of regular functions $\C[X(\La_\beta,T)]$}\label{ssec:aug}  Let $\beta$ be a positive braid word on $n$-strands. Consider the Legendrian link $\La_{\beta}\sse(\R^3,\xi_{st})$ and a set of marked points $T\sse\La_\beta$, one per component. Let $A_{\La_\beta}$ be Legendrian contact dg-algebra of $(\La_\beta,T)$ with $\Z$-coefficients. We refer to \cite[Section 5.1]{CasalsNg} or the survey \cite{Etnyre_ng} for details on $A_{\La_\beta}$. Consider its augmentation variety $X(\La_\beta,T)$, which is the space of all dg-algebra maps from $A_{\La_\beta}$ to $\C$, the latter considered as a dg-algebra in grading $0$ and zero differential.
Following \cite[Section 5]{CasalsNg}, the space $X(\La_\beta,T)$, which is naturally an affine variety, can be explicitly described as follows. For $k \in [n-1]$, define the $n\times n$ matrix $P_k(a)$, as a function of an input $a$, as:
\[
(P_k(a))_{ij} = \begin{cases} 1 & i=j \text{ and } i\neq k,k+1 \\
1 & (i,j) = (k,k+1) \text{ or } (k+1,k) \\
a & i=j=k+1 \\
0 & \text{otherwise.}
\end{cases}
\]
Namely, $P_k(a)$ is the identity matrix except for the $2\times 2$ submatrix given by rows and columns $k$ and $k+1$, which is $\left( \begin{smallmatrix} 0 & 1 \\ 1 & a \end{smallmatrix} \right)$. Consider the full twist $\Delta^2$ in $n$-strands and set $\ell=\ell(\beta\Delta^2)$ for the length of $\beta\Delta^2$. Given a braid word $\beta\Delta^2= \sigma_{k_1}\cdots\sigma_{k_\ell}$, where $\sigma_i\in\Br_n$ are Artin generators in $n$-strands, let $c=|\pi_0(\La_{\beta\Delta^2})|$ and choose $i_1,\ldots,i_c\in\N$ such that the $i_j$-th strand in $\beta\Delta^2$ is in its $j$-th connected component, $j\in[c]$. Define
$D(\mathbf{t}_{\beta\Delta^2})$ to be the diagonal $n\times n$ matrix with $(k,k)$-entry equal to $t_{j}$ if $k=i_j$ and $1$ otherwise, let $\mathbf{1}$ be the $n\times n$ identity matrix and set
$$P_{\beta\Delta^2}(z_1,\ldots,z_\ell;t_1,\ldots,t_c):=P_{k_1}(z_1)P_{k_2}(z_2)\cdots P_{k_\ell}(z_\ell) D(\mathbf{t}_{\beta\Delta^2}).$$
Then, by \cite[Proposition 5.2]{CasalsNg}, $X(\La_\beta,T)$ is isomorphic to the affine variety:
$$X(\La_\beta,T)\cong\{(z_1,\ldots,z_\ell;t_1,\ldots,t_c):\mathbf{1} + P_{\beta\Delta^2}(z_1,\ldots,z_\ell;t_1,\ldots,t_c)=0\}\sse\C^{\ell}\times(\C^\times)^c.$$

The next result is proven in \cite[Theorem 1.1]{GSW}. It also follows from \cite[Theorem 1.1]{CasalsWeng22}, after noticing that $X(\La_\beta,T)$ is isomorphic to a decorated moduli space of sheaves in $\R^2$ singularly supported in a front for $\La_\beta$.

\begin{thm}[\cite{CasalsWeng22}]\label{thm:cluster} Let $\beta$ be a positive braid word on $n$-strands, $\La_{\beta}\sse(\R^3,\xi_{st})$ its associated  Legendrian link and a set of marked points $T\sse\La_\beta$, one per component. Then the coordinate ring of regular functions $\C[X(\La_\beta,T)]$ is a cluster algebra. In addition, it has the following properties:

\begin{itemize}
    \item[(a)] Let $(L_\beta,\Gamma_\beta)$ be the pair of an embedded exact Lagrangian filling $L_\beta$ of $\La_\beta$ and an $\L$-compressing system $\Gamma_\beta=\Gamma(\bG(\beta))$ associated to the plabic fence $\bG(\beta)$. Then there exists a canonical cluster seed $\mathfrak{c}(L_\beta,\Gamma_\beta)$ associated to $L_\beta$ and it has quiver $Q(\bG(\beta))$.\\
    
    \item[(b)] Let $(L,\Gamma)$ be an exact Lagrangian filling with an $\L$-compressing system obtained from the Lagrangian filling $(L_\beta,\Gamma_\beta)$ by a sequence of Lagrangian disk surgeries along an ordered collection of curves $\g_1,\ldots,\g_k\in\Gamma_\beta$.\footnote{Here it is a hypothesis that such sequence of Lagrangian disk surgeries is such that the initial $\L$-compressing system $\Gamma_\beta$ carries through a sequence of $\L$-compressing systems along the disk surgeries and the last system is $\Gamma$.} Then\\

    \begin{itemize}
        \item[(i)] There is a canonical cluster seed $\mathfrak{c}(L,\Gamma)$ associated to $(L,\Gamma)$ and the cluster variables in the cluster seed $\mathfrak{c}(L,\Gamma)$ are computed by microlocal merodromies along the Poincar\'e dual relative cycles of the curves in the $\L$-compressing system $\Gamma$. In particular, cluster variables are indexed by the curves in $\Gamma$.\\

        \item[(ii)] The cluster seed $\mathfrak{c}(L,\Gamma)$ is $\mu_{v_k}\cdots\mu_{v_1}(\mathfrak{c}(L_\beta,\Gamma_\beta))$, the cluster seed obtained by the sequence of cluster mutations along the vertices $v_i\in Q(\bG(\beta))_0$ associated to $\g_i\in\Gamma(\bG(\beta))$, $i\in[k]$.
    \end{itemize}
\end{itemize}
\end{thm}

Technically, \cite{CasalsWeng22} shows that $\C[X(\La_\beta,T)]$ is an upper cluster algebra with the above properties. That said, it is locally acyclic, see e.g.~\cite[Theorem 7.13]{CGGLSS}, and thus  the cluster algebra coincides with the upper cluster algebra, as proven in \cite{Muller14_AU}.

For Theorem \ref{thm:main}, it suffices to focus on the mutable part of the cluster algebra $\C[X(\La_\beta,T)]$: it is well-established that there are $b_1(L)$ mutable vertices in each seed of $\C[X(\La_\beta,T)]$, where $b_1(L)$ coincides for all embedded exact orientable Lagragian fillings $L$, as the frozen variables have only to do with the set of marked points $T$, cf.~\cite{CasalsWeng22,GSW}.

\begin{remark} Note that the class of Legendrian links $\La_\beta$ is such that $\C[X(\La_\beta,T)]$ has been proven to be a cluster algebra. This is unknown for a general Legendrian link $\La\sse(\R^3,\xi_{st})$. Among other reasons, this is due to the lack of definition of a cluster structure on a general derived stack, which is a matter of homotopical algebraic geometry. That said, the technique developed to prove Theorem \ref{thm:main} and Corollary \ref{cor:main} should likely apply to general Legendrian links $\La\sse(\R^3,\xi_{st})$ once it is understood how to make sense of cluster structures on $X(\La,T)$ and they are proven to exist.\hfill$\Box$
\end{remark}


\subsection{Proof of Theorem \ref{thm:main}}\label{ssec:proof_main} First, we choose the Lagrangian filling $L$ and the $\L$-compressing system $\Gamma$ in the statement to be $L:=L_\beta$ and $\Gamma:=\Gamma(\beta)$. This filling and $\L$-compressing system were both introduced in Subsection \ref{ssec:compressing_systems}. As stated there, the curve configurations $\SC(\Gamma(\beta))=\SC(\bG(\beta))$ coincide. Here $\bG(\beta)$ is the plabic fence associated to $\beta$, as described in Section \ref{sec:QP_PlabicFence}, and $\SC(\bG(\beta))$ is the associated configuration, as constructed in Section \ref{sssec:configuration_plabicfence}. We shorten notation to $\SC(\beta):=\SC(\bG(\beta))$.

Consider the cluster seed $\mathfrak{c}(L_\beta,\Gamma_\beta)$ in $\C[X(\La_\beta,T)]$ associated to $(L_\beta,\Gamma_\beta)$. By construction, the vertices of the quiver in $\mathfrak{c}(L_\beta,\Gamma_\beta)$ are given by curves in $\SC(\beta)$ and the arrows record signed intersections, cf.~ \cite{GSW} or \cite[Section 4]{CasalsWeng22}.  Thus the quiver in the seed $\mathfrak{c}(L_\beta,\Gamma_\beta)$ coincides with the quiver $Q(\SC(\beta))$ associated to the curve configuration $\SC(\beta)$ in $L_\beta=\Sigma(\bG(\beta))$. Note that the curves in $\SC(\beta)$ are smooth, oriented, embedded and form a basis of $H_1(L_\beta,\Z)$. Therefore, $\SC(\beta)$ is indeed a curve configuration.

It follows from Properties (P1) and (P2) in Subsection \ref{ssec:iteration}, or direct inspection, that $\SC(\beta)$ is reduced and $Q(\SC(\beta))$ contains no 2-cycles. Equally important, Proposition \ref{thm:QPnondeg} implies that the QP $(Q(\SC(\beta)),W(\SC(\beta)))$ is non-degenerate.

Let $(v_1,\ldots,v_\ell)$ be the sequence of mutable vertices in $Q(\SC(\beta))$ given to us in item $(i)$ of the statement. Let us construct an embedded exact Lagrangian filling $L_{k}$ for the seed $\mu_{v_k}\ldots\mu_{v_1}(\mathfrak{c}(L,\Gamma))$ along with an $\L$-compressing system for $L_k$. Let us denote by $\sD_0=\sD_\beta$ the collection of $\L$-compressing disks associated to $\Gamma(\beta)$. Now, the vertex $v_1$ corresponds to an $\L$-compressible curve $\g_1\in\SC(\Gamma(\beta))$. In order to obtain a Lagrangian filling in the seed $\mu_1:=\mu_{v_1}(\mathfrak{c}(L_\beta,\Gamma_\beta))$ we apply Proposition \ref{prop:iterative_step} with $L=L_\beta$ and $D=D_{v_1}\in\sD_0$ the $\L$-compressing disk associated to $\g_1$. The hypothesis of reducedness and non-degeneracy of the initial QP are indeed satisfied by the previous paragraph.

Proposition \ref{prop:iterative_step} now produces an $\L$-compressing system $\sD_1$ for the Lagrangian filling $L_1:=\mu_{D_{v_1}}(L_\beta)$ whose associated QP is reduced and non-degenerate. In addition, the Lagrangian disks in $\sD_1$ are in a specified bijection with the disks in $\sD_0$. This bijection is established as follows. The $\L$-compressing disks are indexed by the curves in the respective configurations. Lemma \ref{lem:configuration_Lagrangiansurgery} implies that the curves undergo a $\g$-exchange under Lagrangian disk surgery. By construction, curves in a configuration before and after a $\g$-exchange are in a specified bijection, as the vertices of the corresponding quiver are identified (via the identity), cf.~Section \ref{sssec:definition_exchange}. A reduction of a configuration also determines a unique bijection, as Lemmas \ref{lem:triplepoint} and \ref{lem:bigon} show that the vertices of the quiver remain (identically) the same under triple point moves and local bigon moves. Therefore, the Lagrangian disks in $\sD_1$ are indeed in a specified bijection with the disks in $\sD_0$.

Consider the next vertex $v_2$ of $Q(\SC(\beta))$ at which we must mutate. By the bijection above, this specifies a unique $\L$-compressing disk $D_{v_2}$ in $\sD_1$ in the $\L$-compressing system of $\mu_{D_{v_1}}(L_\beta)$. Since the QP associated to $\sD_1$ is reduced and non-degenerate, we can apply Proposition \ref{prop:iterative_step} to $\mu_{D_{v_1}}(L_\beta)$ and $\sD_1$. This produces a Lagrangian filling $L_2=\mu_{D_{v_2}}\mu_{D_{v_1}}(L_\beta)$ in the cluster seed $\mu_{v_2}\mu_{v_1}(\mathfrak{c}(L,\Gamma))$ with an $\L$-compressing system $\sD_2$. Since the QP associated to $\sD_2$ is again reduced and non-degenerate, we can iteratively apply Proposition \ref{prop:iterative_step} without any constraints. By Theorem \ref{thm:cluster}, this procedure indeed constructs the required Lagrangian fillings $L_k=\mu_{D_{v_k}}\ldots\mu_{D_{v_2}}\mu_{D_{v_1}}(L_\beta)$ with $\L$-compressing systems $\Gamma_k=\sD_k=\mu_{D_{v_k}}\ldots\mu_{D_{v_2}}\mu_{D_{v_1}}(\sD_\beta)$ in the required cluster seeds. This proves item $(i)$ of the statement. The construction we used, applying Proposition \ref{prop:iterative_step}, implies item $(ii)$.
\hfill$\Box$

\subsection{A related statement} Let $\La\sse(\R^3,\xi_{st})$ be a Legendrian link and $L\sse(\R^4,\la_{st})$ an embedded exact Lagrangian filling. Suppose that $\Gamma$ is a collection of $\ell$ oriented simple $\L$-compressible curves in $L$ which are linearly independent in $H_1(L;\Z)$, but not necessarily spanning. Let us refer to such a collection as a partial $\L$-compressing system of rank $\ell$. A curve QP $(Q(\Gamma),W(\Gamma))$ is still defined.\footnote{In fact, $(Q(\Gamma),W(\Gamma))$ is defined for any collection of oriented simple curves, even if they are not linearly independent in $H_1(L;\Z)$. Linear independence is only used in proving Proposition \ref{prop:local_bigons}. Thus one may proceed by just assuming that the assumptions in this proposition, non-degeneracy and the bigon condition, hold instead.} If we want to emphasize the dependence on $L$, we write $Q(L,\Gamma)=Q(\Gamma)$. If such QP $(Q(\Gamma),W(\Gamma))$ is non-degenerate, one may proceed as above and conclude a statement in line with Theorem \ref{thm:main} for more general Legendrian links $\La\sse(\R^3,\xi_{st})$. Following the same steps in the proof of Theorem \ref{thm:main}, we obtain the following:

\begin{thm}\label{thm:main_general}
Let $\La\sse(\R^3,\xi_{st})$ be a Legendrian link, $T\sse\La$ a set of marked points with one marked point per component and $X(\La,T)$ the affine scheme given by the spectrum of $H^0$ of the Legendrian contact dg-algebra of $\La\sse(\R^3,\xi_{st})$. Suppose that there exists an orientable embedded exact Lagrangian filling $L\sse(\R^4,\la_{st})$ of $\La$ and a partial $\mathbb{L}$-compressing system $\Gamma$ for $L$ of rank $\ell$ such that $(Q(\Gamma),W(\Gamma))$ is non-degenerate. Then:

\begin{itemize}
    \item[(i)] If $\mu_{v_\ell}\ldots\mu_{v_1}$ is any sequence of mutations, where $v_1,\ldots,v_\ell$ are mutable vertices of the quiver $Q(L,\Gamma)$, then there exists a sequence of embedded exact Lagrangian fillings $L_k$ of $\La$ with associated quivers $Q(L_k,\Gamma_k)=\mu_{v_k}\ldots\mu_{v_1}(Q(L,\Gamma))$, for all $k\in[\ell]$.\\

    \item[(ii)] Each embedded exact Lagrangian filling $L_k$ is equipped with a partial $\mathbb{L}$-compressing system $\Gamma_k$ of rank $\ell$ such that Lagrangian disk surgery on $L_k$ along any Lagrangian disk in $\sD(\Gamma_k)$ yields an $\mathbb{L}$-compressing system. Furthermore, $\Gamma_{k+1}$ is equivalent to such a partial $\mathbb{L}$-compressing system via a sequence of triple point moves and local bigon moves.
    
\end{itemize}
\color{black}
In addition, suppose that $\C[X(\La,T)]$ is a cluster algebra such that:
\begin{enumerate}
    \item $(L,\Gamma)$ defines a cluster seed $\mathfrak{c}(L,\Gamma)$ for $\C[X(\La,T)]$, up to quasi-cluster isomorphism, where $\Gamma$ is completed to a basis of $H_1(L,\Z)$ as in \cite{CasalsWeng22}. (Thus adding only frozens: here the resulting seed is independent of the chosen completion up to quasi-cluster isomorphism.)\\

    \item Lagrangian disk surgery corresponds to cluster mutation in $\C[X(\La,T)]$.
\end{enumerate}

Then $\mathfrak{c}(L_k,\Gamma_k)=\mu_{v_k}\ldots\mu_{v_1}(\mathfrak{c}(L,\Gamma))$ in $\C[X(\La,T)]$, for all $k\in[\ell]$. In particular, the map $\mathfrak{C}:\mbox{Lag}^c(\La)\lr\mbox{Seed}(\C[X(\La,T)])$ is surjective, i.e.~there exists an embedded exact Lagrangian filling endowed with an $\L$-compressing system realizing every cluster seed in $\C[X(\La,T)]$.
\hfill$\Box$
\end{thm}
\color{black}
The advantage of Theorem \ref{thm:main_general} is that it applies to general $\La\sse(\R^3,\xi_{st})$. The disadvantage is that, in a given instance being studied, one must construct a partial $\mathbb{L}$-compressing system $\Gamma$, verify the hypothesis that $(Q(\Gamma),W(\Gamma))$ is non-degenerate and show that $\C[X(\La,T)]$ is a cluster algebra with the desired properties. The latter part can sometimes be established using the Starfish lemma \cite[Prop. 6.4.1]{FWZ} and local acyclicity (or a cyclic rotation argument as in \cite[Section 5.4]{CGGLSS}). In the case of Theorem \ref{thm:main}, we have built above a specific $\mathbb{L}$-compressing system $\Gamma=\Gamma_\beta$ for any $\La_\beta\sse (\R^3,\xi_{st})$ and proven, with Proposition \ref{thm:QPnondeg}, the non-degeneracy of its associated QP. For such links, it follows from \cite{CasalsWeng22,CGGLSS} that $\C[X(\La,T)]$ is a cluster algebra. See also \cite{GLSBS,GLSB} for an alternative construction of cluster structures on $\C[X(\La,T)]$ via 3D plabic graphs.
\color{black}
\bibliographystyle{alpha}
\bibliography{main}

\end{document}